\documentclass[a4paper,10pt,reqno]{article}

\usepackage{graphicx}
\usepackage{mathtools}
\usepackage[english]{babel}
\usepackage{color}
\usepackage{amsthm}
\usepackage{amssymb,todonotes}
\usepackage{amsfonts}
\usepackage[T1]{fontenc}
\usepackage{tikz-cd}
\usepackage{bm}
\usepackage[utf8]{inputenc}
\usepackage{bbm}
\usepackage[hidelinks]{hyperref}
\usepackage{cleveref} 
\usepackage{bbold}
\usepackage{todonotes}
\usepackage{orcidlink}

\usepackage[left=2cm,right=2cm,top=2cm,bottom=2cm]{geometry}

\theoremstyle{definition}
\newtheorem{definition}{Definition}[section]
\theoremstyle{plain}
\newtheorem{theorem}[definition]{Theorem}
\theoremstyle{plain}

\theoremstyle{definition}

\theoremstyle{plain}
\newtheorem{cor}[definition]{Corollary}
\theoremstyle{plain}
\newtheorem{prop}[definition]{Proposition}
\theoremstyle{plain}
\newtheorem{lemma}[definition]{Lemma}
\theoremstyle{definition}
\newtheorem{rem}[definition]{Remark}
\theoremstyle{definition}

\counterwithin{equation}{section}

\newcommand{\norm}[1]{\left\lVert #1 \right\rVert} 

\def \N{\mathbb N}
\def \Z{\mathbb Z}
\def \Q{\mathbb Q}
\def \R{\mathbb R}
\def \C{\mathbb C}

\def \P{{\mathbb P}}
\def \pa{{\partial}}
\def \O{\mathcal{O}}
\def \T{\mathbb{T}}

\def \E{\mathbb{E}}
\def \A{\mathcal{A}}
\def \FL{\mathcal{F}L}
\def \B{\mathcal{B}}
\def \H{\mathcal{H}}
\def \M{\mathcal{M}}
\def \J{\mathcal{J}}

\def \k{\bm{k}}
\def \j{\bm{j}}

\def \tc{\mathtt{c}}
\def \tr{\mathtt{r}}
\def \tu{\mathtt{u}}
\def \cR{\mathcal{R}}
\def \tn{\mathtt{n}}
\def \tN{\mathtt{N}}

\title{ {\bf Almost global large deviations principle for the KdV equation}}

\date{}

\author{Riccardo Berforini D'Aquino\footnote{International School for Advanced Studies (SISSA), Via Bonomea 265, 34136, Trieste, Italy. \newline
	\textit{Emails:} \texttt{rberfori@sissa.it}, \texttt{rgrandei@sissa.it}}\, \orcidlink{0009-0006-0436-6167}, 
\ Ricardo Grande$^*$\ \orcidlink{0000-0002-1521-6790}}

\begin{document}

\maketitle
\begin{abstract}
We study extreme wave formation for the Korteweg–de Vries equation on the torus with random initial data of average size $\varepsilon$. 
We establish a large deviations principle for the supremum of the solution over arbitrarily long polynomial timescales $t\leq \varepsilon^{-n}$ for any fixed $n\in\N$. This identifies the leading-order asymptotics of the probability of observing unusually large amplitudes.

In this integrable setting, the dynamics evolves on invariant tori where Fourier moduli are almost conserved, ruling out mechanisms for extreme wave formation based on resonant energy exchange. 
As a result, large amplitudes can only arise through coherent structures or dispersive focusing, which corresponds to the quasi-synchronization of many phases. 
We show that the latter is dominant in the weakly nonlinear regime.

Our approach combines a Birkhoff normal form analysis with probabilistic arguments, exploiting the stability of the integrable dynamics to control the probability of phase quasi-synchronization over long timescales.
\end{abstract}

\tableofcontents

\section{Introduction} \label{section: introduction}

In dispersive wave systems, different modes propagate at different speeds, yet spatially localized structures with large peaks can still form. Such events, often referred to as extreme or rogue waves, correspond to rare deviations from typical amplitudes and are attributed to a combination of linear and nonlinear effects, whose precise role remains unclear \cite{onorato}.

\smallskip

In this work, we adopt a probabilistic viewpoint to study extreme wave formation under random initial data. Following the framework of Dematteis, Grafke, Onorato and Vanden-Eijnden \cite{dematteis2,dematteis}\footnote{Their work was based on the cubic NLS equation and the Dysthe equation on $\T$, but the general framework applies to more complex models of water wave dynamics.
In particular, the one-dimensional Dysthe equation can be transformed into a KdV-type equation \cite{GKS,saut}.}, we characterize the most likely mechanisms leading to such events by establishing a large deviations principle for the supremum of the solution.

\smallskip

One mechanism, known as \emph{nonlinear focusing}, attributes the formation of extreme waves to nonlinear energy transfer across modes. In this scenario, large amplitudes arise through interactions between waves of different wavenumbers. This mechanism was recently shown to be dominant for the beating NLS equation on $\T$ \cite{grande2}.
A second mechanism, known as \emph{dispersive focusing}, attributes extreme wave formation to constructive interference between modes. In this case, large amplitudes arise when phases simultaneously synchronize at a given point. This mechanism was shown to drive extreme wave formation for the cubic NLS equation and the gravity water waves system on $\T$ over nonlinear timescales \cite{bertimasperograndestaffilani,grande}.

\smallskip

These works study rogue wave formation for Gaussian initial data of typical size $\varepsilon$, and are limited to  timescales $t \ll \varepsilon^{-3}$. In this paper, we characterize extreme wave formation for the KdV equation on $\T$ over arbitrarily long polynomial timescales $t \leq \varepsilon^{-n}$ for any fixed $n \in \N$. While we focus on the KdV equation, the same approach applies to the cubic NLS equation.

\smallskip

In this integrable setting, mechanisms based on resonant energy exchanges are incompatible with the dynamics, since solutions evolve on invariant tori determined by the initial Fourier moduli, which are (to leading order) conserved. As a result, large amplitudes can only arise through coherent structures associated with the integrable flow or through constructive interference between modes.

\smallskip

Identifying which of these mechanisms is most likely remains challenging, as it requires propagating precise statistical information for solutions of nonlinear PDEs over long timescales. In particular, the initial data are typically not distributed according to an invariant Gibbs measure of the underlying equation; instead, the probability measure is often supported on high-regularity data obtained empirically from direct measurements in the sea \cite{data1,data2} or constructed to match standard wave spectra \cite{jonswap}.

\smallskip

This leads to the problem of understanding the out-of-equilibrium evolution of the tails of such measures under the flow of nonlinear PDEs. This question has been the subject of significant recent work, including developments in wave turbulence theory \cite{denghani,DIP,grandehani} and quasi-invariant measures \cite{ohtzvetkov,Tzvetkov}. While the latter approach can yield upper and lower bounds for the tails of the evolved measure, they are not yet sharp enough to establish a large deviations principle, which is needed to fully characterize extreme wave formation.

\smallskip

To establish a sharp large deviations principle over such long timescales, we study the probability that the phases of a large number of modes quasi-synchronize. This event lies at the heart of dispersive focusing, which we show to be the dominant mechanism of extreme wave formation for this integrable dynamics. A key difficulty is to control the time evolution of the probability of this quasi-synchronization event. Despite the complete integrability of the system, the Fourier phases are highly nonlinear functions of the initial Fourier moduli and phases, and tracking their exact distribution seems out of reach. Our approach bypasses this obstacle by constructing an explicitly measurable subset of this event, combining random fixed point arguments, stability theory for integrable Hamiltonian systems, and a normal form construction which may be of independent interest.

\subsection{Main results} \label{intro: main results}

Consider the Korteweg–De Vries \cite{korteweg} equation (KdV):
\begin{equation} \label{intro: kdv}
u_t + u_{xxx} + uu_x = 0, \quad x \in \T.
\end{equation}
The KdV equation \eqref{intro: kdv} is a Hamiltonian system,
\begin{equation} \label{intro: hamiltonian}
\partial_t u = \partial_x \nabla H(u), \qquad H(u) = \frac{1}{2} \int_{\T} ( u_x^2 - \frac{1}{3} u^3  ) \, \mathrm{d}x,
\end{equation}
where $\nabla$ denotes the $L^2$-gradient. Beyond the Hamiltonian, this equation enjoys infinitely many conservation laws, including the average $\int_{\T} u(x) \, \mathrm{d}x$ and the mass $\int_{\T} u(x)^2 \, \mathrm{d}x$. The conserved average allow us to consider solutions to \eqref{intro: kdv} evolving in the Sobolev space of real-valued functions with zero average $\dot{H}^s(\T,\R)$, cf. \eqref{def:hs}. 

\smallskip

The KdV equation is globally well-posed in
$L^2(\T)$, which follows from a local well-posedness result and the conserved mass \cite{babin,erdogan}. See also \cite{wellposednesskdv} for the sharp global well-posedness result in $H^s(\T)$ for $s \ge -\frac12$. 

\medskip

We consider \eqref{intro: kdv} with a zero-mean, random initial datum
\begin{equation} \label{intro: random data}
u_{\varepsilon}^{\omega}(0,x) = \varepsilon \sum_{k \in \Z \smallsetminus \{ 0 \}} c_k \eta_k^{\omega} e^{ikx},
\end{equation}
where $(\eta_k^{\omega})_{k \in \N}$ are i.i.d standard, complex, Gaussian r.v. \eqref{intro: complex gaussian}, $\overline{\eta_k^{\omega}} = \eta_{-k}^{\omega}$ and $(c_k)_{k \in \N}$ is a sequence of real coefficients, with
\begin{equation} \label{intro: coefficients}
c_k = c_{-k}\geq 0, \qquad \sum_{k \in \N} |k|^{2s} c_k^2 < +\infty, \qquad \forall s \ge 0.
\end{equation}
In particular, $u_0^{\omega}(x)$ is itself a centered, real Gaussian r.v. with variance $2\varepsilon^2 \sum_{k\in\N} c_k^2$.

\smallskip

The condition \eqref{intro: coefficients} guarantees that $u_{\varepsilon}^{\omega}(0)\in\bigcap_{s\geq 0} \dot{H}^s(\T,\R)$ a.s., which implies the a.s.~existence of a global solution $u_{\varepsilon}^{\omega}(t,x)$ to \eqref{intro: kdv}. Indeed, note that 
\begin{equation} \label{intro: average norm} 
\E \norm{u_{\varepsilon}^{\omega}(0,x)}_{\dot{H}^s}^2 = 2 \varepsilon^2 \sum_{k \in \N} |k|^{2s} c_k^2 ,
\end{equation}
which is finite by \eqref{intro: coefficients}.

\medskip

The main result of this manuscript is a large deviations principle for the sup-norm of the solution to the KdV equation over arbitrarily long polynomial timescales:

\begin{theorem} \label{thm:main_intro}
Consider the KdV equation  \eqref{intro: kdv} with random initial datum \eqref{intro: random data}. Let $u_{\varepsilon}^{\omega}(t,x)$ be the corresponding (a.s. global) unique solution.  Fix $\tn \in \N$, $\lambda > 0$, and $\delta \in (0,1)$. Then we have
\begin{equation} 
\lim_{\varepsilon\to 0^{+}} \sup_{|t|\leq \varepsilon^{-\tn}} \left \lvert \varepsilon^{2\delta} \log \P \left( \sup_{x \in \T} u_{\varepsilon}^{\omega}(t,x) \ge \lambda \varepsilon^{1-\delta} \right) +  \frac{\lambda^2}{4 \sum_{k\in\N} c_k^2} \right \rvert = 0.
\end{equation}
In particular, for any $t=t(\varepsilon)\leq \varepsilon^{-\tn}$ we have the following large deviations principle:
\begin{equation} \label{eq: limit ldp intro}
\lim_{\varepsilon \to 0^{+}} \varepsilon^{2\delta} \log \P \left( \sup_{x \in \T} u_{\varepsilon}^{\omega}(t,x) \ge \lambda \varepsilon^{1-\delta} \right) = - \frac{\lambda^2}{4 \sum_{k\in\N} c_k^2}.
\end{equation}

\end{theorem}

Let us make some comments.

\begin{itemize}
    \item[1.] {\sc Timescales}: Previous works on large deviations principles for dispersive equations have reached timescales $t\ll \varepsilon^{-3}$ \cite{grande,grande2,bertimasperograndestaffilani}. To the best of our knowledge, this is the first result to achieve arbitrarily long polynomial timescales. To reach these times, we exploit the integrability of the KdV equation in order to track the probability that a large collection of nonlinear phases of $u_{\varepsilon}^{\omega}(t)$ quasi-synchronize, building on the novel approach introduced in \cite{bertimasperograndestaffilani}.

    \item[2.] {\sc Dispersive focusing}: As a byproduct of our proof, we show in \Cref{thm: intro dispersive focusing} below that the dominant mechanism of formation of extreme waves in the KdV equation in the weakly nonlinear regime is \emph{dispersive focusing}, i.e. the quasi-synchronization of a large number of phases of the Fourier modes giving rise to constructive interference. This is in contrast to \cite{grande2}, where energy exchange between Fourier modes is the main mechanism. 

    \item[3.] {\sc Regularity}:  
    Our result only requires a finite amount of regularity. In particular, hypothesis \eqref{intro: coefficients} can be weakened, as we do in \Cref{thm: almost global ldp intro}. However, we didn't try to optimize the minimal regularity $s$, which depends on the normal form procedure (see \Cref{intro:BNF} below), and we leave this question for future work.   

    \item[4.] {\sc Growth of $L^{\infty}$-norm}: Let us highlight that we do not study the growth of the $L^{\infty}$-norm of individual trajectories. Instead, \Cref{thm:main_intro} and \Cref{thm: intro dispersive focusing} characterize the most likely initial data which give rise to large solutions to \eqref{intro: kdv} at time $t$. Large solutions are, in turn, defined by comparison with the average initial size which is $\O(\varepsilon)$, cf.~\eqref{intro: average norm}.

    \item[5.] {\sc Comparison with \cite{grande}}: In \cite{grande}, the authors establish a large deviations principle analogous to \eqref{eq: limit ldp intro} for the cubic NLS equation on $\T$, valid over timescales $t = \mathcal{O}(\varepsilon^{-2(1-\delta)})$ for a.s.~smooth initial data. We expect that the methods developed in this manuscript can be used to extend the results of \cite{grande} to arbitrary polynomial timescales $t \sim \varepsilon^{-\tn}$ for any $\tn \in \N$. Notably, the KdV equation presents additional challenges due to the presence of derivatives in the nonlinear term, which requires working in high-regularity spaces to avoid loss of derivatives. By contrast, our approach may require significantly weaker regularity assumptions in the NLS setting.
\end{itemize}

As previously announced, the second main result of this manuscript is the identification of the main mechanism behind the formation of an extreme wave.
Consider the a.s. global solution to \eqref{intro: kdv}:
\begin{equation} \label{intro: exact solution}
u_{\varepsilon}^{\omega}(t,x) = 2 \sum_{k \in \N} |u_k^{\omega}(t)| \cos\left(\psi_k^{\omega}(t)+ kx\right),
\end{equation}
where $\psi_k^{\omega}(t)$, $k \in \N$, are the phases of the Fourier coefficients of the solution at time $t$. Let $\tN \in \N$, $\delta\in (0,1)$ and define the set 
\begin{equation} \label{intro: fasi piccole}
\mathfrak{P}(\tN,\delta,\varepsilon) = \left \{ \omega \in \Omega \, | \, \left \lvert \psi_k^{\omega}(t) \ \mbox{mod}\ 2\pi \right \rvert \le \varepsilon^{ \frac{1-\delta}{15}}, \, \, \forall k \in \{1,\ldots, \tN\}\ \mbox{such that}\ c_k>0\right \},  
\end{equation}
i.e. the event that the first $N$ phases are very close to zero. We note that the contributions towards the $L^{\infty}$-norm from Fourier coefficients $u_k^{\omega}(t)$ such that $u_k^{\omega}(0)=0$, cf. \eqref{intro: random data}, are negligible on account of the quasi-conservation of the Fourier moduli by the KdV dynamics (cf.~\eqref{eq: approx fasi lineari}). As a result, phases associated to such Fourier coefficients are ``free''.

The following result shows that the main mechanism for the formation of extreme waves is the quasi-synchronization of the phases.

\begin{theorem}[Dispersive Focusing] \label{thm: intro dispersive focusing}
Under the same assumptions of \Cref{thm:main_intro} we have 
\begin{equation} \label{intro: dispersive focusing}
\lim_{\varepsilon\to 0^{+}} \sup_{|t|\leq \varepsilon^{-\tn}} \left \lvert  \varepsilon^{2\delta} \log \P \left( \left \{ \sup_{x \in \T} u_{\varepsilon}^{\omega}(t,x) \ge \lambda \varepsilon^{1-\delta} \right \} \cap \mathfrak{P}(\tN,\delta,\varepsilon) \right) + \frac{\lambda^2}{4 \sum_{k\in\N} c_k^2} \right \rvert = 0.
\end{equation}
\end{theorem}

\Cref{thm:main_intro} and \Cref{thm: intro dispersive focusing} are proved in \Cref{subsection: probability}, see \Cref{thm: almost global ldp intro} and \Cref{thm: dispersive focusing}, respectively. We outline the main ideas of the proof below.

\subsection{Ideas of the proof} \label{intro: ideas of the proof}

\Cref{thm: intro dispersive focusing} essentially follows from the proof of \Cref{thm:main_intro}, so we only discuss the latter. The proof of \Cref{thm:main_intro} relies on the construction of a sequence of approximate solutions to the KdV dynamics \eqref{intro: kdv}-\eqref{intro: random data} via a Birkhoff normal form procedure. These approximations allow us to track the probability that a large number of Fourier modes quasi-synchronize. Ensuring that this probability does not decrease too quickly over long timescales is a key idea of the proof.

\smallskip

In order to derive an approximate solution to the KdV dynamics, we use the Birkhoff normal form procedure to construct a close-to-the-identity local diffeomorphism $\Phi : B_s(0,\varepsilon) \subseteq \dot{H}^s \longrightarrow \dot{H}^s$ such that
\begin{equation} \label{intro: normal form}
\H \circ \Phi (v) = \sum_{m=1}^{\lfloor \frac{r}{2}\rfloor} \widehat{\H}_{2m}(\mathfrak{J}) + \mathtt{R}_r (v), \qquad \mathfrak{J}=(|v_k|^2)_{k \in \Z \smallsetminus \{ 0 \}}.
\end{equation}
where $\H$ is the KdV Hamiltonian \eqref{intro: hamiltonian} in Fourier coordinates, and $\mathtt{R}=\mathtt{R}_r$ is a remainder. 
The number of normal form steps $r\geq 3$ is chosen to ensure that the desired timescales $t\sim\varepsilon^{-\tn}$ can be reached, while the minimal regularity $s$ depends on $r$ through the normal form procedure, cf.~\eqref{intro: parameters}. We detail the main ideas in the construction of $\Phi$ in \Cref{intro:BNF}.

\smallskip

In the new coordinates $v = \Phi^{-1}(u)$, the KdV equation \eqref{intro: kdv} in Fourier variables reads
\begin{equation}\label{intro:newdyn}
\partial_t v_k(t) = i\, \theta_k(\mathfrak{J}(t))\, v_k(t) + X_{\mathtt{R}}(v)_k,\qquad \theta_k (\mathfrak{J}) = \frac{k}{\pi}\, \sum_{m=1}^{\lfloor \frac{r}{2}\rfloor} \pa_{\mathfrak{J}_k} \widehat{\H}_{2m}(\mathfrak{J}), \qquad k \in \Z \smallsetminus \{ 0 \},
\end{equation}
As justified in \Cref{intro:BNF}, we are able to neglect $X_{\mathtt{R}}(v)$ which leads to the approximate solution
\begin{equation} \label{intro: approx solution in new var}
v_k(t)\approx e^{i t\,\theta_k(\mathfrak{J}(0))} v_k(0),  \qquad k \in \Z \smallsetminus \{ 0 \}.
\end{equation}
Neglecting $X_{\mathtt{R}}(v)$ is highly nontrivial, as it may be unbounded in $\dot{H}^s$. 
In \eqref{intro:approx_proof} we establish the approximation \eqref{intro: approx solution in new var} by finding an \emph{a priori} bound which yields control of $v$ in a high norm (see \eqref{intro:bound1}--\eqref{intro:bound2}), allowing us to justify \eqref{intro: approx solution in new var} in a low norm.

\subsubsection{Non-Gaussian approximation}

Although the approximation \eqref{intro: approx solution in new var} depends solely on the initial datum $v(0)=\Phi^{-1}(u(0))$, the complexity of the nonlinear change of coordinates $\Phi$ makes it extremely difficult to characterize its probability distribution. However, using the fact that $\Phi$ is close-to-the-identity (cf. \eqref{intro: close to id}), we can replace $v_k(0)$ by $u_k(0)$. This leads to the following approximation of the KdV dynamics:
\begin{equation} \label{intro: approssimazione bella}
u_{\mathrm{app},\varepsilon}^{\omega}(t,x) = 2 \varepsilon \sum_{k \in \N} c_k \, R_k^{\omega} \cos \left ( \phi_k^{\omega} + t \,\theta_k( \vec{J}^{\omega}) + kx \right ), \qquad \vec{J}^{\omega} = ( |\Phi^{-1} (u_{\varepsilon}^{\omega}(0))_k|^2 )_{k\in\N},
\end{equation}
where we used the explicit formulae for $u_k(0) = \varepsilon c_k \eta_k^{\omega}$ from \eqref{intro: random data}, and expressed each complex Gaussian in polar coordinates as $\eta_k^{\omega}=R_k^{\omega} e^{i\phi_k^{\omega}}$, where $R_k^{\omega}$ is Rayleigh-distributed and $\phi_k^{\omega}$ is uniformly distributed on $[0,2\pi)$.

A big advantage of the approximation \eqref{intro: approssimazione bella} is that the distributions of the moduli are explicit. This allows us to estimate
\begin{equation}\label{intro:upper_bound}
\lambda \varepsilon^{1-\delta}\leq \sup_{x\in\T} u^{\omega}_{\varepsilon}(t,x) \stackrel{\eqref{intro: approssimazione bella}}{\leq} 2\varepsilon\,\sum_{k\in\N} c_k R_k^{\omega} + o(\varepsilon^{1-\delta})
\end{equation}
up to a set of negligible probability, cf.~\eqref{eq: almost global ldp - lower bound}. The inequality \eqref{intro:upper_bound} leads to sharp upper bounds for the left-hand side of \eqref{eq: limit ldp intro} over the desired timescales, see \Cref{thm: almost global ldp}.

\smallskip

Obtaining sharp lower bounds for $\sup_{x\in\T} u^{\omega}_{\varepsilon}(t,x)$ is one of the main challenges in this paper. While the nonlinear phases in \eqref{intro: approssimazione bella} can be ignored in the upper bound, they play a fundamental role in the lower bound. In particular, note the phases $\phi_k^{\omega} + t \,\theta_k( \vec{J}^{\omega})$ are in general not uniformly distributed in $[0,2\pi)$ on account of the highly nonlinear dependency of $\vec{J}^{\omega}$ (via $\Phi^{-1}$) on $(\phi_k^{\omega})_{k\in\N}$ and $(R_k^{\omega})_{k\in\N}$.

This issue was avoided in \cite{grande} by approximating\footnote{Note that for small timescales $t\ll \varepsilon^{-2}$, one can approximate $\vec{J}^{\omega}\approx 0$ in \eqref{intro: approssimazione bella}. In this case, $\theta_k(0)$ is precisely the dispersive relation $k^3$ (resp. $k^2$ in NLS) and the approximation is the linear evolution, which trivially preserves the Gaussianity of the initial data.
} 
\begin{equation}\label{intro:phase_trick}
\vec{J}^{\omega} = ( |\Phi^{-1} (u_{\varepsilon}^{\omega}(0))_k|^2 )_{k\in\N} \approx ( |(u_{\varepsilon}^{\omega}(0)_k|^2 )_{k\in\N}
\end{equation}
which led to the approximation 
\begin{equation} \label{intro: brutta}
\widetilde{u}_{\mathrm{app},\varepsilon}^{\omega}(t,x) = 2 \varepsilon \sum_{k \in \N} c_k \, R_k^{\omega} \cos \left ( \phi_k^{\omega} + t \,\theta_k( \varepsilon\vec{R}^{\omega}) + kx \right ), \qquad \vec{R}^{\omega} = ( R_k^{\omega})_{k\in\N}.
\end{equation}
Exploiting that $\vec{R}^{\omega}$ is independent of $\vec{\phi}^{\omega}$, one can prove that $\widetilde{u}_{\mathrm{app},\varepsilon}^{\omega}(t,x)$ is Gaussian, which can be used to derive sharp lower bounds.  

\smallskip

However, the approximation \eqref{intro: brutta} is only valid over ``short'' timescales\footnote{
This timescale is due to the fact that $\norm{\Phi^{-1}(u)-u} \approx \norm{u}^2$, cf.~\eqref{eq: close to the identity - thm ham}, but may be longer 
if this difference is cubic or quartic. The degree of this error in turn depends on the degree of the nonlinearity in the PDE.
}, 
since $|t \vec{J}^{\omega} - t \varepsilon\vec{R}^{\omega}| \gg 1$ for timescales 
$t\gg \varepsilon^{-3}$. 
In order to reach timescales $t \sim \varepsilon^{-\tn}$ for any $\tn\in\N$, we do not attempt to track the probability distribution of the nonlinear phases \eqref{intro: approssimazione bella}. Instead, we adapt a novel idea in \cite{bertimasperograndestaffilani} based on studying the probability of quasi-synchronization of many phases.

\subsubsection{Quasi-synchronization and large deviations} 

For 
$M\in\N$, we define the quasi-synchronization event:
\begin{equation}\label{intro:quasisync_big}
\mathfrak{M}(t;\varepsilon, M) := \{ \omega\in\mathbb{\Omega} \mid |\varphi_k (t; \omega, \vec{\phi}^{\omega}) \ \mbox{mod}\ 2\pi| \leq \varepsilon \quad \mbox{for}\ k=1,\ldots,M \}.
\end{equation}
with
\[
\varphi_k (t; \omega, \vec{\phi}) = \phi_k + t \,\theta_k \left( (|\Phi^{-1} (u_{\varepsilon}^{\omega}(0; \vec{\phi}))_j|^2)_{j\in\N}\right),\qquad \vec{\phi} \in \R^M,
\]
where $u_{\varepsilon}^{\omega}(0; \vec{\phi})$ is the ``partially randomized'' initial datum whose first $M$ modes have deterministic phases $\vec{\phi}$, while all moduli and the rest of the phases remain random, cf.~\eqref{def: partially rid}. Note that, when we plug in the uniformly distributed phases $\vec{\phi}^{\omega}$, the nonlinear phases $\varphi_k ( t; \omega, \vec{\phi}^{\omega})$ in \eqref{intro:quasisync_big} are precisely those in \eqref{intro: approssimazione bella}.

\smallskip

For $M\gg 1$, to be fixed later, and $\omega \in \mathfrak{M}(t;\varepsilon ,M)$, note that we can essentially invert \eqref{intro:upper_bound}, namely
\begin{equation}\label{intro:lb}
\sup_{x\in\T} u^{\omega}_{\varepsilon}(t,x) \stackrel{\eqref{intro: approssimazione bella}}{\geq} 2\varepsilon\, (1-\varepsilon )\,\sum_{k=1}^{M} c_k R_k^{\omega} + o(\varepsilon^{1-\delta}+ M^{1-s})
\end{equation}
up to a set of negligible probability, cf.~\eqref{eq: almost global ldp - lower bound}. Note that we exploit the decay of the coefficients $c_k$, cf.~\eqref{intro: coefficients}, to control the terms with $k>M$ with high probability. In particular, we fix $M=\lfloor \varepsilon^{-\delta} \rfloor$ and $s\geq \delta^{-1}$ so that $o(M^{1-s})=o(\varepsilon^{1-\delta})$ in \eqref{intro:lb}.

\smallskip

The key obstacle is guaranteeing that $\P(\mathfrak{M}(t;\varepsilon, M))$ is not too small\footnote{Compared with the probability of an extreme wave $\P ( \sup_{x\in\T} u^{\omega}_{\varepsilon}(t,x) \geq \lambda \varepsilon^{1-\delta})$ which, as shown via the upper bound \eqref{intro:upper_bound}, is at most exponential in $-\varepsilon^{-2\delta}$.} as time passes, which is highly nontrivial since $\varphi(t)$ is a non-explicit, nonlinear function of $\vec{\phi}^{\omega}$. 
To do so, we first construct a single random variable $\vec{\phi}^{*,\omega}$, via a Brouwer fixed point argument, such that $\varphi (t;\omega, \vec{\phi}^{*,\omega})=0$. Then we construct a neighborhood around this fixed point 
\begin{equation}\label{intro:nhood}
    \mathcal{N}(\beta):= \left \{ \omega\in\mathbb{\Omega} \, \Big | \, \norm{\vec{\phi}^{*,\omega} - \vec{\phi}^{\omega}}_{\ell^{\infty}} < \beta \right \} \subseteq \mathfrak{M}(t; \varepsilon , M)
\end{equation}
with $\beta=\beta(t,\varepsilon,M)\in (0,\pi)$ well-chosen so that the inclusion holds. An advantage of this neighborhood is that we can prove that $\P(\mathcal{N}(\beta))= (\beta/\pi)^{M}$, cf.~\Cref{prop: properties of random fixed point}. Moreover, we prove a key \emph{factorization property}: for any event $\mathcal{A}$ in the $\sigma$-algebra $\mathcal{G}=\sigma((R_j^{\omega})_{j \in \N}, (\phi_j^{\omega})_{j>M})$,
\begin{equation}\label{intro:facprop}
    \P( \mathcal{A}\cap \mathcal{N}(\beta)) = \P( \mathcal{A}) \, \P(\mathcal{N}(\beta)) = \left(\frac{\beta}{\pi}\right)^{M}\, \P( \mathcal{A}) 
\end{equation}
despite the fact that $\mathcal{N}(\beta)$ depends on the vector $\vec{\phi}^{*,\omega}$, which is $\mathcal{G}$-measurable.

\medskip

The factorization property \eqref{intro:facprop}, together with \eqref{intro:lb}, allows us to obtain an explicit lower bound:
\begin{align}
        \P \left ( \sup_{x\in\T} u^{\omega}_{\varepsilon}(t,x) \geq \lambda \varepsilon^{1-\delta} \right ) & \geq \P\left( \left\{\sup_{x\in\T} u^{\omega}_{\varepsilon}(t,x) \geq \lambda \varepsilon^{1-\delta}\right\}\cap \mathcal{N}(\beta)\right) \nonumber \\
        & \stackrel{\eqref{intro:lb},\eqref{intro:facprop}}{\approx} \P\left( 2\sum_{k=1}^{M} c_k R_k^{\omega} \geq \lambda \varepsilon^{-\delta} + o(\varepsilon^{-\delta})\right) \, \P (\mathcal{N}(\beta))\nonumber\\
        & \stackrel{M=\lfloor \varepsilon^{-\delta} \rfloor}{\geq} \exp\left( -\frac{\lambda^2\varepsilon^{-2\delta}}{4\sum_{k\in\N} c_k^2} + \varepsilon^{-\delta}\,\log\frac{\beta}{\pi} + o(\varepsilon^{-2\delta})\right),\label{intro:endgame}
\end{align}
where we used sharp lower bounds for the tails of the random variable $\sum_{k=1}^{M} c_k R_k^{\omega}$, cf.~\Cref{prop: grande} -- see \Cref{thm: almost global ldp - lower bound} for the full details.

\smallskip

Crucially, we prove in \Cref{thm: almost global ldp - lower bound} that we can choose $\beta \sim t^{-1}$ for $t\geq 1$ while guaranteeing the inclusion property \eqref{intro:nhood}. This guarantees that the right-hand side in \eqref{intro:endgame} is not too small over timescales $t\sim\varepsilon^{-\tn}$ for any $\tn\in\N$.

While a similar strategy was used in \cite{bertimasperograndestaffilani}, the size of the neighborhood \eqref{intro:nhood} depended on the Lipschitz constant of the nonlinear flow. As such, $\beta$ decays exponentially in $t$ which, in view of \eqref{intro:endgame}, is an obstacle to reaching long timescales $t\sim \varepsilon^{-\tn}$. In this work, however, we prove that $\beta$ depends only on the Lipschitz constant of the flow of an integrable system\footnote{In our case, this integrable system is not KdV but \eqref{intro:newdyn} without the term $X_{\mathtt{R}}$. In fact, our strategy will yield a large deviations principle for a system which is only integrable up to order $2m$, up to timescales $t \ll \varepsilon^{-m}$.} in \emph{action-angle coordinates}. As a result, $\beta$ decays only as $t^{-1}$, which is the key to reaching the arbitrarily long polynomial timescales of \Cref{thm:main_intro}.

\subsubsection{Birkhoff normal form}\label{intro:BNF}

\paragraph{Formal transformations.} Let us now explain the Birkhoff normal form procedure we use to construct the change of coordinates $\Phi$ in \eqref{intro: normal form}. Recall the KdV Hamiltonian \eqref{intro: hamiltonian} in Fourier coordinates and the Poisson bracket:
\begin{equation}
\H(u) = \H_2(u) + \H_3(u) = \pi \sum_{k \in \Z^*} k^2 |u_k|^2 - \frac{\pi}{3} \sum_{\k \in \M_3} u^{\k}, \qquad \{ F , G \} = \sum_{k \in \Z^*} \frac{ik}{2\pi} \frac{\pa F}{\pa u_k} \frac{\pa G}{\pa u_{-k}},
\label{intro: kdv hamiltonian in fourier}
\end{equation}
where $\k \in \M_3$ means $\k =(k_1,k_2,k_3)\in(\Z\setminus\{0\})^3$ with $k_1+k_2+k_3=0$ and $u^{\k}=u_{k_1} u_{k_2} u_{k_3}$
(see \Cref{subsec:notation} for the notation).
To construct $\Phi$, we make use of the full KdV hierarchy (see \cite{lax,lax2,magri,rangointero} and \Cref{lem:first_integrals}), i.e. a countable sequence of conserved quantities in involution, which can be written as:
\begin{equation}
F^{(j)}(u) = \sum_{n=2}^{j+2} F^{(j)}_n(u) = \pi \sum_{k \in \Z^*} k^{2j} |u_k|^2 + \sum_{n=3}^{j+2} \sum_{\k \in \M_n} f_{\k}^{(j)} u^{\k}, \qquad \{ F^{(j)} , F^{(l)} \} = 0 \qquad \forall \, j,l \in \N,
\label{intro: first integrals in fourier}
\end{equation}
where $F^{(1)}=\mathcal{H}$. We construct  $\Phi$ as a composition of several changes of coordinates $\Phi_n$, $n\geq 3$. More precisely, each transformation is written as  $\Phi_n=\Phi_n^{(1)} \circ \ldots \circ \Phi_n^{(n-1)}$ where each $\Phi_n^{(l)}$ puts the $l$-th Hamiltonian of the KdV hierarchy into Birkhoff normal form at order $n$. In particular, $\Phi_n^{(l)}$ is the time-1 flow of an auxiliary Hamiltonian $G_n^{(l)} (u) = \sum_{\k \in \M_n} g_{\k}^{(l)} u^{\k}$ of degree $n$ solving the homological equation:
\begin{equation}\label{intro:hom_eq}
F_n^{(l)}(u)+ \{ F_2^{(l)}, G_n^{(l)} \}(u) = \widehat{F}_n^{(l)}(u)=\sum_{\k \in \M_n} \widehat{f}_{\k}^{(l)} u^{\k} 
\end{equation}
where, using the fact that $ \{ F_2^{(l)}, G_n^{(l)} \}(u)=-i \sum_{\k \in \M_n} \Omega_{l}(\k) g_{\k}^{(l)} u^{\k}$ (cf.~\eqref{eq: facts about poisson bracket} below), we have that
\begin{equation}\label{intro:choice_G}
\widehat{f}_{\k}^{(l)}=0 \qquad \mbox{if}\quad \Omega_{l}(\k)=k_1^{2l+1} + \ldots + k_n^{2l+1} \neq 0, \qquad \mbox{by choosing}\quad  g_{\k}^{(l)} = \frac{f_{\k}^{(l)}}{i\,\Omega_{l}(\k)}\, \mathbbm{1}_{\Omega_{l}(\k)\neq 0}.
\end{equation}

Although each transformation is tailored to a single Hamiltonian $F^{(l)}$ in the hierarchy, the key observation is that it also eliminates non-resonant terms with respect to $F^{(l)}_2$ in the other Hamiltonians, \emph{without generating new terms}. This is a consequence of the fact that Hamiltonians in the KdV hierarchy Poisson-commute, cf.~\eqref{intro: first integrals in fourier} and \cite[Theorem~G.2]{kappeler}. 
More precisely, assuming the KdV hierarchy is in integrable normal form up to order $n-1$, the transformed Hamiltonians still Poisson-commute $\{\tilde{F}^{(j)},\tilde{F}^{(l)}\}=0$ and the terms of degree $n$ must vanish for any $j\in\N$, i.e.
\begin{equation} \label{intro: magic formula}
\{ F_2^{(j)} , \widetilde{F}_n^{(l)} \} + \{ \widetilde{F}_n^{(j)} , F_2^{(l)} \} = 0 \quad \implies \quad \Omega_{l}(\k)\, \widetilde{f}_{\k}^{(j)} =  \Omega_{j}(\k) \,\widetilde{f}_{\k}^{(l)}.
\end{equation}
Using \eqref{intro: magic formula}, one immediately obtains
\begin{equation} \label{intro:hom_eq2}
    \widetilde{F}_n^{(j)}(u) + \{ F_2^{(j)} , G_n^{(l)} \}(u) = \sum_{\k \in \M_n} \widetilde{f}_{\k}^{(j)} \mathbbm{1}_{\Omega_{l}(\k) = 0} \ u^{\k}.
\end{equation}
After applying all the transformations $\Phi_n=\Phi_n^{(1)} \circ \ldots \circ \Phi_n^{(n-1)}$, the only remaining terms will be supported in 
\begin{equation}
    \Omega_{0}(\k) = \ldots = \Omega_{n-1}(\k) =0, \qquad \Omega_{l}(\k)\ \mbox{in \eqref{def: resonant relation}}.
\end{equation}  
A Vandermonde argument (cf.~\Cref{cor: system}) shows that all such terms are integrable, i.e. they depend only on $(|u_j|^2)_{j\in\N}$. We highlight that the cubic NLS equation on $\T$ possesses a hierarchy of first integrals similar to \eqref{eq: first integrals in fourier} but with quadratic part of the form $\sum_{k\in\Z^{*}} k^{j} |u_k|^2$ for each $j\in\N\cup \{0\}$, see for instance \cite[Chapter 3]{akns} for explicit formulae for the AKNS hierarchy\footnote{A completely integrable system of two coupled Schr\"odinger-type PDEs on with unknowns $p$ and $q$ which reduces to the cubic NLS equation when $\overline{p}=q$.}. As a result, the arguments presented above naturally extend to the NLS hierarchy. See also \cite{feola,feola2} for a similar idea applied to the Degasperis-Procesi equation.

\paragraph{Resonant normal form.} 
The main obstacle to making the above steps into a rigorous theorem is closing the Birkhoff normal form procedure in a common space $\dot{H}^s(\T)$. When trying to do so, one encounters two key challenges:
\begin{itemize}
\item[(i)] A first problem is showing that the maps $\Phi_n^{(l)}$ generated by $G_n^{(l)}$ are well-defined in a common space, which requires compensating for the loss of derivatives induced by the Poisson bracket \eqref{intro: kdv hamiltonian in fourier}. This is easy to achieve for $\Phi_3$ and $\Phi_4$, since the polynomials $\Omega(\k)$ with $\k \in \M_n$, $n=3,4$, in \eqref{intro:choice_G} can be factorized, cf.~\Cref{subsection: cases 3 and 4}. However, this strategy fails for $n\geq 5$, as a factorization of $\Omega(\k)$ with $\k \in \M_n$ is not available.
\item[(ii)] A second problem is showing that the higher order terms produced by taking Poisson brackets of each Hamiltonian with $G_n^{(l)}$ can be controlled in a common space $\dot{H}^s(\T)$, a significantly more challenging task than (i).
\end{itemize}

In order to overcome these difficulties, we implement a resonant normal form procedure inspired by the work of Bernier-Gr\'ebert \cite{bernier}. Roughly speaking, the key idea is to truncate the support of the auxiliary Hamiltonian $G_n^{(l)}$ to 
\begin{equation}\label{intro:J}
   \k\in\mathcal{M}_n \quad \mbox{such that} \quad \dfrac{\norm{\k}_{\ell^{\infty}}^{2l-1}}{|\Omega_l(\k)|} \le N\qquad  \mbox{for some}\ N\geq 1 \ \mbox{to be fixed later}.
\end{equation}
The exponent $2l-1$ reflects both the derivative loss from the Poisson bracket and the growth of the coefficients $f_{\k}^{(l)}$ in \eqref{intro: first integrals in fourier}, which are polynomial in $\k$ of degree at most $2l-2$ (see \Cref{lem:first_integrals}). This extends \cite{bernier}, where the coefficients are bounded, and only \eqref{intro:J} with $l=1$ is needed.

The truncation in \eqref{intro:J} readily overcomes problem (i) above, yielding a local diffeomorphism $\Phi_n^{(l)}$ defined in a small ball $B_s(0,\varepsilon) \subseteq \dot{H}^s(\T)$ satisfying
\begin{equation} \label{intro: close to id}
    \norm{\Phi_n^{(l)} (u) - u}_{\dot{H}^s} \lesssim N^{n-3}\, \norm{u}_{\dot{H}^s}^{n-1}\leq N \,\norm{u}_{\dot{H}^s}^2\quad \mbox{with}\ N \ll \varepsilon^{-1},
\end{equation}
see \Cref{lemma: transformations preserve Hs} below. In order to overcome problem (ii), we introduce the spaces of Hamiltonians $F(u)$ such that
\begin{equation}\label{intro:space_coeff}
 \exists\ m \geq 3 \quad \mbox{such that}\quad F(u) = \sum_{\k\in\mathcal{M}_m} c_{\k} u^{\k} \qquad \mbox{with}\quad \norm{\frac{c_{\bm{k}}}{\norm{\k}_{\ell^{\infty}}^{2j-2}}}_{\ell^{\infty}} < \infty.
\end{equation}
Thanks to the truncation \eqref{intro:J}, these spaces are closed under Poisson brackets with $G_n^{(l)}$ for any $j\in\N$ (see \Cref{lemma: stabilità parentesi poisson}), which is the key to solving (ii) above.

\smallskip

Unfortunately, this truncation induces important changes to the formal computations we did before. In particular, some terms from the Hamiltonians $F_n^{(j)}$ are not removed by the transformation $\Phi_n^{(l)}$, $l=1,\ldots, n-1$, as $G_n^{(l)}$ does not fully solve the homological equation \eqref{intro:hom_eq2}. This, in turn, modifies the computation in \eqref{intro: magic formula}, as the left-hand side of the equation now includes additional terms coming from Poisson brackets between potentially non-integrable terms, which have not been removed by the transformations $\Phi_3,\ldots,\Phi_{n-1}$. Controlling the contribution of such additional terms is one of the main difficulties addressed in \Cref{prop: magic formula}.

All in all, we apply this resonant normal form procedure a finite number $r \ge 3$ of times, which puts in approximate normal form the first $r-1$ Hamiltonians of the KdV hierarchy up to degree $r$, see \eqref{intro: normal form}. Note that $\mathtt{R}_r(v)$ also includes terms of degree $n \leq r$ which were not removed by the transformations $\Phi_3,\ldots,\Phi_{r}$ and which are supported in the complement of \eqref{intro:J} for some $l\in \{1,\ldots,n-1\}$.

\medskip

Let us mention an alternative normal form procedure by Kappeler and P\"oschel \cite{kappeler}, who use the Lax pair of the KdV equation \eqref{intro: kdv} to construct action-angle variables. 
Here we follow a different strategy, showing that a finite number $r$ of first integrals suffices to carry out $r$ steps of Birkhoff normal form procedure. 
This method is potentially applicable to other Hamiltonian systems, even in the absence of complete integrability, provided that the available conserved quantities satisfy suitable structural properties. On the other hand, the normal form of Kappeler–P\"oschel may allow one to lower the minimal regularity $s$ required in \Cref{thm:main_intro}, a direction we plan to investigate in future work.

\paragraph{Dynamical consequences.} We now study the KdV dynamics in the new coordinates $v=\Phi^{-1} (u)$ using the transformed Hamiltonian \eqref{intro: normal form}, which in Fourier variables reads as \eqref{intro:newdyn}. The main obstacle is that we do not know whether $X_{\mathtt{R}}(v)$ maps $\dot{H}^s(\T)$ into itself. Since the well-posedness of \eqref{intro:newdyn} follows from that of the KdV equation, it suffices to control the $\dot{H}^s$-norm of the solution $v$ to \eqref{intro:newdyn}, which solves:
\begin{equation}\label{intro:norm_v}
    \pa_t \norm{v(t)}_{\dot{H}^s}^2 = \{ \norm{\cdot}_{\dot{H}^s}^2, \mathtt{R}_r \} (v(t))=  \{ \norm{\cdot}_{\dot{H}^s}^2, \mathtt{R}_r^{(1)} \} (v(t)) + \{ \norm{\cdot}_{\dot{H}^s}^2, \mathtt{R}_r^{(2)} \} (v(t)) 
\end{equation}
where
\begin{itemize}
    \item $\mathtt{R}_r^{(1)}$ are remainder terms of degree at least $r+1$. Using the truncation \eqref{intro:J}, we show that 
    \begin{equation}\label{intro:apriori1}
         \{ \norm{\cdot}_{\dot{H}^s}^2, \mathtt{R}_r^{(1)} \} (v(t)) \lesssim_{s,r} ( N \norm{v(t)}_{\dot{H}^s})^{r+1}.
    \end{equation}
    \item $\mathtt{R}_r^{(2)}$ is made of remainder terms of degree $n$, $5\leq n\leq r$, supported on indices $\k$ in the complement of \eqref{intro:J} for some $l=1,\ldots,n-1$. In \Cref{teorema multiindici rimasti}, we show that such indices $\k$, when ordered such that $|k_1|\geq \ldots \geq |k_n|$, satisfy one of two possible properties: 
    \begin{equation}\label{intro:two_cases}
    \mbox{either} \quad k_1=-k_2 \quad \mbox{and} \quad |k_1|\geq N^{\frac{1}{2r}},\qquad \mbox{or}\quad |k_3|\geq N^{\frac{1}{2r}}. 
    \end{equation}
    Monomials satisfying the first option can be written as $v^{\k} = |v_{k_1}|^2 v^{\k'}$, where $\k'$ is the vector $\k$ without the first two indices. When taking the Poisson bracket between this monomial and the $\dot{H}^s$-norm, we obtain 
    \begin{equation}\label{intro:apriori2}
        \left \lvert \{\norm{v}_{\dot{H}^s}^2, v^{\k}\} \right \rvert =
        |v_{k_1}|^2 \left \lvert \{ \norm{v}_{\dot{H}^s}^2 , v^{\k'} \} \right \rvert \stackrel{\eqref{intro:two_cases}}{\le} N^{-\frac{s}{2r}} \norm{v}_{\dot{H}^s}^2\, \left \lvert \{ \norm{v}_{\dot{H}^s}^2 , v^{\k'} \} \right \rvert  \lesssim N^{-\frac{s}{2r}} \norm{v}_{\dot{H}^s}^{n}.
    \end{equation}
Monomials satisfying the second option in \eqref{intro:two_cases}, give rise to terms in  $ \{ \norm{\cdot}_{\dot{H}^s}^2, \mathtt{R}_r^{(2)} \}$ with at least three large indices $|k_1|\geq |k_2|\geq |k_3|\geq N^{\frac{1}{2r}}$. We use two of these indices among which to split $2s+1$ derivatives, while the third index is used to gain a factor of $N^{-\frac{s}{2r}}$, cf.~\Cref{prop: controllo resti}.
\end{itemize}

The \emph{a priori} bounds \eqref{intro:apriori1}--\eqref{intro:apriori2}, combined with \eqref{intro:norm_v}, yield the quasi-conservation of $\norm{v(t)}_{\dot{H}^s}$ over arbitrarily long polynomial timescales: 
\begin{equation} \label{intro:bound1}
   \norm{v(t)}_{\dot{H}^s}^2 \lesssim \norm{v(0)}_{\dot{H}^s}^2  + t \, (N\,\norm{v}_{L^{\infty}([0,t],\dot{H}^s)})^{r+1} + t\, \,N^{-\frac{s}{2r}}\, \sum_{n=5}^r N^{n-3} \, \norm{v}_{L^{\infty}([0,t],\dot{H}^s)}^{n} \lesssim \varepsilon^2 
\end{equation}
after choosing
\begin{equation} \label{intro: parameters}
N\ll \varepsilon^{-1/2}, \qquad s \gtrsim r^2 \qquad  \mbox{and} \quad t\ll \varepsilon^{-r/4}.
\end{equation}

\smallskip

Using similar bounds to \eqref{intro:apriori1}--\eqref{intro:apriori2}, we can control the vector field $X_{\mathtt{R}}$ in $\dot{H}^1(\T)$ despite the loss of derivatives:
\begin{equation} \label{intro:bound2}
    \norm{X_{\mathtt{R}}(v(t))}_{\dot{H}^1} \lesssim  (N\,\norm{v(t)}_{\dot{H}^2})^r + \sum_{n=5}^r N^{n-3} \,\norm{v(t)}_{\dot{H}^2}^{n-2} \,\norm{v(t)}_{\dot{H}_{|k| \ge N^{1/2r}}^2} \lesssim 
    \varepsilon^{\frac{r}{2}} + N^{-\frac{s}{2r}} \norm{v(t)}_{\dot{H}^s}
    \stackrel{\eqref{intro:bound1}}{\lesssim} \varepsilon^{\frac{r}{2}}.
\end{equation}

This bound allows us to justify the approximation \eqref{intro: approx solution in new var} by integrating \eqref{intro:newdyn}:
\begin{equation}\label{intro:approx_proof}
    v_k(t) \stackrel{\eqref{intro:bound2}}{\approx} e^{i\int_0^t \theta_k (\tau)\, d\tau} v_k(0) \approx e^{i t\, \theta_k (0)} v_k(0),
\end{equation}
where in the second approximation we exploit the quasi-conservation of the moduli $|v_k(t)|$ in time, as well as the fact that $\theta_k$ depends only on these moduli\footnote{
Let us highlight that the fact that $\theta_k$ is real valued alone is not sufficient to justify this approximation. We critically use the integrability of $\widehat{\mathcal{H}}_{2m}$ in \eqref{intro: normal form}. The approximation \eqref{intro: approx solution in new var} is, in turn, the key to being able to choose $\beta$ in \eqref{intro:nhood} as $\O(t^{-1})$.
}.












\subsection{Notation}\label{subsec:notation}

\paragraph{Sets and indices.} We denote by $\N$ the set of positive natural numbers
$\N = \{ 1,2,3,...\}$ and by $\Z^* =\Z\setminus\{0\}$ the set of nonzero integer numbers. Given two sets $A$ and $B$, we denote by $A\sqcup B$ their disjoint union, i.e. $A\cup B$ such that $A\cap B=\emptyset$. We denote by $\lfloor x\rfloor$ and $\lceil x\rceil$ the floor and ceiling functions of $x\in\R$.

For $n \ge 3$ and $l \in \N \cup \{0\}$, we will consider the following sets of indices:
\begin{align}
\M_n &= \{ \bm{k} = (k_1,...,k_n) \in (\Z^*)^n \, | \, k_1+\ldots+k_n=0 \}, \label{def: indices zero momentum} \\
\widetilde{\M}_n &= \frac{\M_n}{\mbox{Sym}(n)}, \quad \text{where $\mbox{Sym}(n)$ is the symmetric group}, \label{def: indices zero momentum quotient}  \\
\mathcal{D}_n &= \{ \bm{k}  \in \mathcal{M}_n\, | \, |k_1| \ge |k_2| \ge \ldots |k_n| \}, \label{def: ordered indices zero momentum}  \\
\mathcal{R}_n^{l} &= \{ \bm{k} \in \mathcal{M}_n\, | \, k_1^{2l+1}+\ldots+k_n^{2l+1}=0 \}, \label{def: resonant indices zero momentum} 
\end{align}
and we set
\[
\mathcal{M} = \bigcup_{n \ge 3} \mathcal{M}_n, \qquad \qquad \mathcal{D} = \bigcup_{n \ge 3} \mathcal{D}_n, \qquad \qquad \widetilde{\M} = \bigcup_{n \ge 3} \widetilde{\M}_n.
\]
Given $\k \in \M_n$, $n\geq 3$, we denote by $\mu_3(\k)$ the third largest (in absolute value) index in $\k$. For $M>0$, we define the sets:
\begin{align}
\M_n^{(\mu_3 \ge M)} & = \{ \k \in \M_n \, | \, \mu_3(\k) \ge M \}, \label{insieme terzo indice grande} \\
\M_n^{(I \ge M)} & = \{ \k \in \M_n \, | \, \exists \, j \neq l \quad \text{s.t.} \quad k_j = -k_l,\ |k_j| \ge M \}. \label{insieme azione grande} 
\end{align}

We will denote by $\#\bm{k}$ the length of the vector $\bm{k}$, \textit{i.e.} $\#\bm{k} = n$ if $\bm{k} \in (\Z^*)^n$. 
Moreover, if $\bm{k} = (k_1,\ldots,k_n)$, we will write $u^{\bm{k}}=u_{k_1}\ldots u_{k_n}$. For two vectors $\k,\j\in\M_n$, we write $\k\sim\j$ if $[\k]=[\j]$ in $\widetilde{\M}_n$, i.e. there exists $\sigma\in\mbox{Sym}(n)$ such that $k_l=j_{\sigma (l)}$ for all $l=1,\ldots,n$.

Finally, we introduce the quantities:
\begin{equation}
\Omega_l(\bm{k}) = k_1^{2l+1}+\ldots+k_n^{2l+1}.
\label{def: resonant relation}
\end{equation}

\paragraph{Functional spaces.} For $1\le p < \infty$, we denote by $\ell^p(\Z)$ the Banach space of sequences
\[
\ell^p(\Z)=\biggl \{ x=(x_k)_k \subseteq \mathbb{C} \, \biggl | \, \sum_{k \in \Z}|x_k|^p < \infty \biggr \}
\]
with the usual norm $\norm{x}_{\ell^p}=\left(\sum_{k \in \Z}|x_k|^p \right)^{\frac{1}{p}}$, and the usual generalization to $p=\infty.$ 

We consider functions $f$ defined on the torus $\T \cong \frac{\R}{2\pi \Z}$, which may be identified with $2\pi$-periodic functions $f: \R \rightarrow \C.$ 
For $s\in\R$ and $1\leq p\leq \infty$, we denote by $\FL^{s,p}(\T)$ the Banach space of measurable functions $f:\T \rightarrow \C$ whose Fourier coefficients $(f_k)_{k\in\Z}$ satisfy
\[
\norm{\langle k \rangle^s f_k}_{\ell^p_k} < \infty,
\]
where $\langle k \rangle=\max\{ 1,|k|\}$ is the  Japanese bracket. 
In the case $p=2$ these spaces coincide with the usual Sobolev spaces $H^s(\T)$, which we identify with the space of sequences 
\[
H^s(\T) = \left \{ u = \sum_{k \in \Z} u_k e^{ikx} \, \bigg| \, \norm{u}_{H^s}^2 = \sum_{k \in \Z} \langle k \rangle^{2s} |u_k|^2 < \infty \right \}.
\]
In the case of $s=0$, the scalar product in $L^2(\T,\R)$ is given by 
\[
\langle f,g \rangle_{L^2} = \int_{\T} f(x)g(x) \, \mathrm{d} x.
\]
Since we often work with mean zero, real-valued functions, we introduce the subspace
\begin{equation} \label{def:hs}
\dot{H}^s(\T,\R) = \{ u \in H^s \, | \, u_0=0 \quad \text{and} \quad u_{-k} = \overline{u_k} \quad \forall k \in \Z \}\quad \mbox{with norm}\quad
\norm{u}_{\dot{H}^s}^2 = \sum_{k \in \Z^*} |k|^{2s} |u_k|^2.
\end{equation}
Moreover, we write $B_s(0,r)$ for the ball in $\dot{H}^s(\T,\R)$ centered at the origin and with radius $r>0$.
Similarly, we introduce the space $\dot{\FL}^{s,p}(\T,\R)$ of mean-zero real valued functions, with norm
\begin{equation} \label{def: norm flsp}
\norm{u}_{\dot{\FL}^{s,p}}^p = \sum_{k \in \Z^*} |k|^{ps} |u_k|^p.
\end{equation}

Finally, we often write $u_x$ instead of $\partial_x u$, meaning the partial derivative of $u$ with respect to $x$.

\paragraph{Inequalities.} We  write $a \lesssim b$ if there exists a positive constant $C$ such that $a \le Cb$. If the constant depends on some other quantity $p$, we will write $a \lesssim_p b.$ When the constant $C$ is explicitly written, it may change from line to line. We also employ the notation $f(\varepsilon)=\O( g(\varepsilon))$ if $f(\varepsilon)\lesssim g(\varepsilon)$, and $f(\varepsilon)=o(g(\varepsilon))$ if $\lim_{\varepsilon\rightarrow 0^{+}} \frac{f(\varepsilon)}{g(\varepsilon)} =0$.

Given functions $u,v:\T \rightarrow \C$, we will often use the Young convolution inequality:
\begin{equation}
\norm{uv}_{\FL^{0,r}} \le \norm{u}_{\FL^{0,p}} \norm{v}_{\FL^{0,q}} \qquad \text{where} \qquad 1+\frac{1}{r} = \frac{1}{p} + \frac{1}{q}.
\label{young convolution inequality}
\end{equation}

We also recall the algebra property of $H^s(\T)$ when $s > \frac{1}{2}$, 
\[
\norm{uv}_{H^s} \lesssim_s \norm{u}_{H^s} \norm{v}_{H^s}.
\]

Finally, we will often use the embedding of $H^s(\T)$ into $\FL^{0,1}(\T)$ which follows from the Cauchy-Schwarz inequality:
\begin{equation}
\norm{u}_{\FL^{0,1}} = \sum_{k \in \Z} |u_k| \le \left ( \sum_{k \in \Z} \frac{1}{\langle k \rangle^{2s}} \right )^{\frac{1}{2}} \left ( \sum_{k \in \Z} \langle k \rangle^{2s} |u_k|^2 \right )^{\frac{1}{2}} \lesssim_s \norm{u}_{H^s}.
\label{eq: embedding hs in fl01}
\end{equation}

\paragraph{Random variables.} Given a random variable $X$, i.e. a measurable function $X :(\mathbb{\Omega},\mathcal{F}) \to (\R, \mathcal{B}(\R) )$ over a probability space 
$(\mathbb{\Omega}, \mathcal{F}, \P)$, we denote by $\sigma(X):=\{ X^{-1}(B) \colon B \in \mathcal{B}(\R)\}$ the $\sigma$-algebra generated by $X$. Given two random variables $X$ and $Y$, we write $\sigma(X,Y)$ for the smallest $\sigma$-algebra containing $\sigma(X) \cup \sigma(Y)$. 

Given a sub-$\sigma$-algebra $\mathcal{G} \subseteq \mathcal{F}$ and a random variable $X\in L^1(\mathbb{\Omega},\mathcal{F}, \P)$, we denote by $\E[X^{\omega}\, \lvert \, \mathcal{G}]$ the conditional expectation of $X$ with respect to $\mathcal{G}$.

A standard complex Gaussian is a random variable of the form 
\begin{equation} \label{intro: complex gaussian}
\eta^{\omega} = \mathrm{Re}(\eta^{\omega}) + i \, \mathrm{Im}(\eta^{\omega}), 
\end{equation}
where $\mathrm{Re}(\eta^{\omega})$ and $\mathrm{Im}(\eta^{\omega})$ are independent, real, centered Gaussian r.v. with variance $1/2$. We will often write $\mathrm{Re}(\eta^{\omega}),\mathrm{Im}(\eta^{\omega})\sim \mathcal{N}_{\R}(0,1/2)$
and $\eta^{\omega}\sim \mathcal{N}_{\C}(0,1)$. In particular,
\begin{equation} \label{intro: valori attesi}
\E [\eta^{\omega}] = 0, \qquad \E \left [ |\eta^{\omega}|^2 \right ] = 1, \qquad \E[(\eta^{\omega})^2] = 0.
\end{equation}



\paragraph{Acknowledgements.} We thank Joackim Bernier, Alberto Maspero and Michela Procesi for many useful comments. This research was backed by the GNAMPA-INdAM and the European Union. Views and opinions expressed are however those of the authors only and do not necessarily reflect those of the European Union or the European Research Council. Neither the European Union nor the granting authority can be held responsible for them.

\section{Birkhoff normal form for the KdV hierarchy} \label{section: normal form}

\subsection{Preliminaries} \label{subsection: preliminaries}

In this section we study the periodic 
KdV equation with real-valued initial datum: 
\begin{equation}\label{eq: kdv}
\begin{cases}
u_t + u_{xxx} + uu_x = 0, \quad x \in \T \\
u |_{t=0}= u(0) \in H^s(\T),  \quad s \ge 0.
\end{cases}
\end{equation}
from a deterministic point of view.
The KdV equation admits the following formally conserved quantities:
\begin{equation}\label{eq:conserved_L2}
\int_{\T} u(t,x) \, \mathrm{d}x \quad \mbox{and}\quad 
\int_{\T} u(t,x)^2 \, \mathrm{d}x.
\end{equation}

Using the conserved $L^2$-norm, one can show that the KdV equation admits a unique real-valued global solution in $H^s(\T)$ for $s \ge 0$,  see \cite[Chapter~3]{erdogan} and the references therein for a detailed presentation of the well-posedness theory.

Moreover, using the conserved mean \eqref{eq:conserved_L2}, it suffices to study the KdV flow for zero-mean initial data. From now on, we can thus consider $u(0) \in \dot{H}^s(\T)$, and the global-in-time solution will have zero mean.

The KdV equation is a Hamiltonian system which can be written as
\begin{equation}
\pa_t u = \pa_x \nabla_{L^2} H, \qquad \mbox{with}\qquad H(u) = \frac{1}{2} \int_{\T} ( u_x^2 - \frac{1}{3} u^3  ) \, \mathrm{d}x.
\label{eq: kdv in hamiltonian form}
\end{equation}
where $\nabla_{L^2} H$ denotes the $L^2$-gradient of $H$, which represents the Fr\'echet derivative of $H$ with respect to the standard scalar product in $L^2(\T)$. In the phase space $\dot{H}^s(\T)$ we will use the Fourier coefficients  as coordinates
\begin{equation}
u_k(t) = \frac{1}{2\pi} \int_{\T} u(t,x) e^{-ikx} \, \mathrm{d}x, \quad k \in \Z^*, \quad u_0 = 0, \quad \overline{u_k}=u_{-k}
\end{equation}
where the last equality follows from the fact that $u$ is real-valued.
In the Fourier setting, the Hamiltonian in \eqref{eq: kdv in hamiltonian form} reads \eqref{intro: kdv hamiltonian in fourier}.
The Poisson tensor in \eqref{eq: kdv in hamiltonian form} is associated to the Poisson bracket \eqref{intro: kdv hamiltonian in fourier}
which is (at least formally) bilinear, antisymmetric and satisfies the Jacobi identity.
We will prove that it is well-defined on the set of formal polynomials (cf. \Cref{def:formal_pol}) and meaningful if $u$ belongs to some $\dot{H}^s$ for $s$ sufficiently large (cf. \Cref{lemma: facts about poisson bracket}).

All in all, the KdV equation \eqref{eq: kdv} can be written in Fourier variables as
\begin{equation}
\pa_t u_k = \{ u_k , \H \} = i k^3 u_k - \frac{ik}{2} \sum_{k_1+k_2=k} u_{k_1} u_{k_2}.
\end{equation}

\paragraph{First integrals.}
As in any autonomous Hamiltonian system, the Hamiltonian is conserved along its own flow.
Moreover, all the functionals $F$ that Poisson-commute with the Hamiltonian $\H$, \textit{i.e.} $\{ F, \H \} = 0,$ are preserved by the flow.
Such functionals are called first integrals, or simply conserved quantities. 
It is well-known \cite{lax, magri, gardner, miura} that the KdV equation has infinitely many conserved quantities of the form
\begin{equation} \label{eq: pre first integrals}
F^{(j)}(u) 
=\int_{\T} P^{(j)}(u,\pa_x u,\ldots,\pa_x^{j} u) \, \mathrm{d} x,
\end{equation}
where $P^{(j)}$ is a polynomial (whose monomials are of degree $\ge 2$) with rational coefficients. In particular, $F^{(1)}=H$, cf.~\eqref{eq: kdv in hamiltonian form}.
These quantities, known as KdV hierarchy, are in involution, i.e.
\begin{equation}
\{ F^{(j)} , F^{(l)} \} = 0 \qquad \forall \, j,l \in \N.
\label{eq: first integral in involution}
\end{equation}
This will be crucial, as we will use the full KdV hierarchy in order to put the KdV Hamiltonian $\H$ in integrable normal form (cf.~\Cref{def: functions in normal form}) up to a small remainder.
In the following lemma, which is proven in \Cref{proof lemma first integrals}, we collect some elementary albeit important properties of the KdV hierarchy.

\begin{lemma}[Structure of the KdV hierarchy] \label{lem:first_integrals} 
For each $j \in \N$ we have 
\begin{equation}
F^{(j)}(u) = \frac{1}{2} \int_{\T} \left [ \left ( \pa_x^j u \right )^2 + \sum_{n=3}^{j+2} P_n(u,\pa_x u,\ldots,\pa_x^{j-1} u) \right ] \, \mathrm{d}x,
\label{eq: first integrals}
\end{equation}
where $P_n$ is a polynomial of degree $n$ with coefficients in $\Q$, which in Fourier can be written as
\begin{equation}
F^{(j)}(u) = \sum_{n=2}^{j+2} F^{(j)}_n(u) = \pi \sum_{k \in \Z^*} k^{2j} |u_k|^2 + \sum_{n=3}^{j+2} \sum_{\k \in \M_n} c_{\k}^{(j)} u^{\k}
\label{eq: first integrals in fourier}
\end{equation}
where for $n \ge 3$ and $\k \in \M_n$, $c_{\k}^{(j)} = c(k_1,\ldots,k_n)^{(j)}$ is a polynomial in $n$ variables with rational coefficients and of degree at most $2j-2$.
Moreover, the reality condition \eqref{def: reality condition} holds. Each term $F^{(j)}_n(u)$ will be called homogeneous $j$-formal polynomial of degree $n$ (cf. \Cref{def:formal_pol}).
\end{lemma}

\paragraph{Functional setting and formal polynomials.}

Taking into account that $\H=F^{(1)}$, the coefficients $c_{\k}^{(1)}$ are bounded uniformly in $\k$, cf. \eqref{intro: kdv hamiltonian in fourier} and \eqref{eq: first integrals in fourier}. As a result,
\[
c_{\k}^{(1)} \in \{ \bm{c} = (c_{\k})_{\k} \in \M_3 \, \lvert \, c_{-\k}=\overline{c_{\k}} \quad \text{and} \quad \norm{\bm{c}}_{\ell^{\infty}} < \infty \}.
\]
This boundedness property of the coefficients does not hold for the whole KdV hierarchy. 
However, \Cref{lem:first_integrals} guarantees that $c_{\k}^{(l)}$, $l\in\N$, is a polynomial in $\k$ of degree at most $2l-2$. 
In order to exploit this property, for $n \ge 3$ and $l \in \N$, we define
\begin{equation}
\begin{split}
Y_n^{l} = \bigg \{ \bm{c} = (c_{\bm{k}})_{{\bm{k}} \in \mathcal{M}_n} \, & \bigg | \, c_{-\bm{k}}= \overline{c_{\bm{k}}} \, \, \text{and} \, \, \norm{\bm{c}}_{Y^l} = \norm{\frac{c_{\bm{k}}}{\norm{\k}_{\ell^{\infty}}^{2l-2}}}_{\ell^{\infty}} < \infty \bigg \}.
\end{split}
\label{def: functional space}
\end{equation}
In particular, $\bm{c}^{(l)} \in Y^l$ for all $l \in \N$.

In \Cref{prop: justification normal form} we will perform canonical transformations which modify the structure of the whole KdV hierarchy in order to put it in normal form.
The spaces \eqref{def: functional space} will be preserved by such transformations, in the sense that the new coefficients of $F_n^{(j)}$ will belong again to $Y_n^j.$ 

\begin{definition}[Formal polynomials] \label{def:formal_pol}
Given $l\in\N$, an $l$-\emph{formal polynomial} $F$ is a Hamiltonian of the form:
\begin{equation}\label{def:formal_polynomial}
 F(u)= F_2(u) + F_{\geq 3}(u)=\sum_{k \in \Z^*} b_k |u_k|^2 + \sum_{\k \in \M} c_{\k} u^{\k} 
 \end{equation}
such that there exist $M>0$ such that
\begin{equation}\label{M_cond}
\sup_{k\in\Z^{*}} \frac{|b_k|}{|k|^{2l}} <\infty,\qquad 
 \sup_{n \ge 3} \norm{ (c_{\k} M^{-\#{\k}})_{\#{\k} = n} }_{Y_n^l} < \infty.
 \end{equation}
$F_2(u)$ is called homogeneous $l$-formal polynomial of degree 2, or simply quadratic  $l$-formal polynomial, while for $n \ge 3$ each $F_n(u) = \sum_{\k \in \M_n} c_{\k} u^{\k}$ is called homogeneous  $l$-formal polynomial of degree $n$.\footnote{In \cite[Definition~1]{polynomials}, a definition of homogeneous polynomials between Banach spaces is given: notice that $F_n$, $n \ge 2$, in \eqref{def:formal_polynomial} is a homogeneous polynomial according to that definition.}
If $F$ is a real valued function, it means that its coefficients in \eqref{def:formal_polynomial} satisfy the \emph{reality condition}
\begin{equation} \label{def: reality condition}
\overline{b_k} = b_{-k} \quad \forall k \in \Z^*, \qquad \qquad \overline{c_{\k}} = c_{-\k} \quad \forall \k \in \M.
\end{equation}
\end{definition}

\begin{rem}
    We note the class of $l$-formal polynomials is closed under addition and scalar multiplication, as is the class of homogeneous $l$-formal polynomials of a fixed degree.
\end{rem}

The next lemma shows that condition \eqref{M_cond}
guarantees that formal polynomials of the form \eqref{def:formal_polynomial} are smooth functionals on a small ball of $\dot{H}^s$.
This lemma generalizes \cite[Lemma~3.2]{bernier}
for the functionals in \eqref{eq: first integrals in fourier} whose coefficients depend polynomially in $\k$. In particular, \cite[Lemma~3.2]{bernier} corresponds to the case $l=1$ in \eqref{M_cond}.

\begin{lemma} \label{lemma: convergence of hamiltonian function}
Let $F$ be an $l$-formal polynomial (cf. \Cref{def:formal_pol}), $l\in\N$, with coefficients $\bm{c}$ satisfying \eqref{M_cond} for some $M>0$. For any $s \ge l$, $F(u)$ is a smooth function on $B_s \left ( 0, \frac{1}{c_s M} \right )$ where $c_s = \left ( \sum_{j \in \Z^*} \frac{1}{j^{2s}} \right )^{\frac{1}{2}} \le \frac{\pi}{\sqrt{3}}$.
Here, $B_s(0,r)$ is the ball of $\dot{H}^s$ centered at the origin and with radius $r > 0$, cf. \Cref{subsec:notation}.
\end{lemma}
\begin{proof}
Note that $F_2(u)$ in \eqref{def:formal_polynomial} is absolutely convergent with series bounded by $\O(\norm{u}_{\dot{H}^l}^2)$ on account of \eqref{M_cond}. As a result, $F_2(u)$ is well-defined for $u\in\dot{H}^s$ since $s\geq l$.

Since $\M = \cup_{n \ge 3} \M_n$, we may write $F_{\geq 3} (u) = \sum_{n \ge 3} F_n(u) = \sum_{\k \in \M} c_{\k} u^{\k}$, which is a function at least cubic in $u$.
The homogeneous polynomial $F_n(u) = \sum_{\k \in \M_n} c_{\k} u^{\k}$ satisfies
\begin{align}
|F_n(u)| & \le \sum_{\k \in \M_n} |c_{\k}| |u^{\k}| = M^n \sum_{\k \in \M_n} \frac{|c_{\k}|M^{-\#\k}}{\norm{\k}_{\ell^{\infty}}^{2l-2}} \norm{\k}_{\ell^{\infty}}^{2l-2} |u^{\k}| \label{key_bound_Y} \\
&\le  M^n\, \norm{ (c_{\k} M^{-\#{\k}})_{\#{\k} = n} }_{Y_n^l}  \,\sum_{\k \in \M_n} \norm{\k}_{\ell^{\infty}}^{2l-2} |u^{\k}|  \le n C M^n \sum_{ \substack{ \k \in \M_n \\ |k_1| \ge |k_2|,\ldots,|k_n|} } |k_1|^{2l-2} |u^{\k}| \nonumber
\end{align}
with $C= \sup_{n \ge 3} \norm{ (c_{\k} M^{-\#{\k}})_{\#{\k} = n} }_{Y_n^l}$, where the multiplication by $n$ is due to the ordering $|k_1|\geq |k_2|,\ldots, |k_n|$. 
Indeed, the sum in $\k\in\M_n$ can be split into $n$ sums depending on which component $k_j$ is maximal. Each of these sums coincides with the sum over $k_1$ by a change of variables.

Since $l\geq 1$, the Jensen inequality and the definition of $\M_n$ \eqref{def: indices zero momentum} yield
\begin{equation}
|k_1|^{l-1} = |k_2+\ldots+k_n|^{l-1} \le (n-1)^{l-2} \left ( |k_2|^{l-1} + \ldots + |k_n|^{l-1} \right )
\label{eq: jensen inequality}
\end{equation}
which we will use frequently. 
Therefore
\[
\begin{split}
|F_n(u)| &\le C n (n-1)^{l-2} M^n \sum_{ \substack{ \k \in \M_n \\ |k_1| \ge |k_2|,\ldots,|k_n|} } |k_1|^{l-1} \left ( |k_2|^{l-1}+\ldots+|k_n|^{l-1} \right ) |u^{\k}| \\
&\le C n (n-1)^{l-1} M^n \sum_{ \substack{ \k \in \M_n \\ |k_1| \ge |k_2|,\ldots,|k_n|} } |k_1|^{l-1}|k_2|^{l-1} |u^{\k}|,
\end{split}
\]
where in the second sum we used the symmetry.
By the Young convolution inequality \eqref{young convolution inequality},
\[
\begin{split}
\sum_{ \substack{ \k \in \M_n \\ |k_1| \ge |k_2|,\ldots,|k_n|} } |k_1|^{l-1}|k_2|^{l-1} |u^{\k}| & \le \norm{ \left ( |\pa_x|^{l-1} u \right )^2 \cdot u^{n-2}}_{\dot{\FL}^{0,\infty}} \le \norm{u}_{\dot{H}^{l-1}}^2 \norm{u}_{\dot{\FL}^{0,1}}^{n-2}. \\
\end{split}
\]
where $|\pa_x|$ is the Fourier multiplier operator with symbol $|k|$.
By \eqref{eq: embedding hs in fl01}, 
\begin{equation}\label{Fn_growth}
|F_n(u)| \le C n^l \left ( c_s M \norm{u}_{\dot{H}^s} \right )^n
\end{equation}
and the thesis follows from the convergence criterion for power series.

Finally, the smoothness of $F$ follows from the following estimate (which replaces \eqref{Fn_growth}):
\[
\norm{\mathrm{d}F_n(u)}_{\mathcal{L}(\dot{H}^s,\R)} \leq C c_s M n^{l+1}\,\left ( c_s M \norm{u}_{\dot{H}^s} \right )^{n-1}.
\]
The series in $n$ has thus the same radius of convergence as that of $F$. Higher-order derivatives are bounded analogously.
\end{proof}

\begin{cor} \label{rmk: homogeneity argument}
Let $F$ be an $l$-formal polynomial (cf. \Cref{def:formal_pol}), $l\in\N$, such that 
\begin{equation}
F(u)= \sum_{n\geq 2} F_n(u) =\sum_{k \in \Z^*} b_k |u_k|^2 + \sum_{\k \in \M} c_{\k} u^{\k} = 0 
\end{equation}
for all $u\in B_s (0,\frac{1}{c_s M})$ (cf.~\Cref{lemma: convergence of hamiltonian function}). Then $b_k=-b_{-k}$ for all $k\in\Z^*$ and 
\begin{equation}
    \tc ( [\k] )= \sum_{\j\sim\k} c_{\j}=0 \qquad \mbox{for all}\ \k\in\M \ \ \mbox{(cf.~\eqref{def: indices zero momentum quotient})}.
\end{equation}
\end{cor}
\begin{proof}
By contradiction, let $n$ be the smallest integer for which $F_n \neq 0$.
Then 
\begin{equation}\label{contr_ub}
F_n(u) = - \sum_{m > n} F_m(u) \quad \stackrel{\eqref{Fn_growth}}{\implies} \quad |F_n(u)| \lesssim_n \norm{u}_{\dot{H}^s}^{n+1}.
\end{equation}
If $n=2$, there must exists some $b_k \neq 0$. Consider $u$ such that $u_k=u_{-k}=\epsilon$ and $u_j = 0$ if $j \neq {\pm k}$. By \eqref{contr_ub},
\[
|F_2 (u)|= |b_k+b_{-k}|\,\epsilon^2 \lesssim \epsilon^3,
\]
which implies that $b_k+b_{-k}=0$. If $n \ge 3$, we use \eqref{def: indices zero momentum quotient} to write
\[
F_n (u)= \sum_{\k \in \M_n} c_{\k}\, u^{\k} = \sum_{[\k] \in \widetilde{\M}_n} \tc ([\k])\, u^{\k}, \qquad \mbox{where}\  \tc ([\k])= \sum_{\j\sim\k} c_{\j}.
\]
Since $F_n \neq 0$, there exists $[\k]$ such that $\tc([\k]) \neq 0$. Define 
\[
\mathfrak{A} = \{ \j\in  \M_n\mid j_1,\ldots,j_n\in \{ k_1,\ldots,k_n\}\}, \quad \widetilde{\mathfrak{A}} = \mathfrak{A}/\mbox{Sym}(n).
\]
Consider $u$ in the subspace of $\dot{H}^s(\T)$ given by $u_r= 0$ if $r\notin \{ k_1,\ldots,k_n\}$ which we identify with $\R^m$, $m=|\{ k_1,\ldots,k_n\}|\leq n$. Then
\[
F_n (u) = \sum_{[\j] \in \widetilde{\mathfrak{A}}} \tc ([\j])\, u^{\j}
\]
is a homogeneous polynomial of degree $n$ in $m$ variables. 
Since $F_n\neq 0$, there exists some $v \in \R^m$ such that $F_n (v)\neq 0$. Setting $u=\epsilon v$, and using \eqref{contr_ub}, we find
\[
|F_n (v)|\, \epsilon^{n} \lesssim \epsilon^{n+1}.
\]
The contradiction follows by taking $\epsilon\rightarrow 0$.
\end{proof}

\begin{rem}\label{diffuk_Frechet}
As a consequence of \Cref{lemma: convergence of hamiltonian function}, a formal polynomial $F$ is smooth and 
\[
\frac{\pa F}{\pa u_k}(u) = \mathrm{d}F(u)[e^{ikx}] \qquad\implies \qquad
\{ u_k , F \} = \frac{ik}{2\pi} \mathrm{d}F(u)[e^{-ikx}] \in \C
\]
is well-defined for any $k \in \Z^*$.
\end{rem}

Next we introduce a special class of $l$-formal polynomials which Poisson-commute (see \Cref{lemma: facts about poisson bracket} (iii) below).

\begin{definition} \label{def: functions in normal form}
Given an $l$-formal polynomial  $F = F(u)$ (cf.~\Cref{def:formal_pol}), we say that it is \emph{integrable} or in \textit{integrable normal form}, if $F = F(I)$, where $I = (I_k)_{k\in\Z^*} = (|u_k|^2)_{k\in\Z^*}.$ 
\end{definition}


The following result shows that Poisson brackets between formal polynomials are well-defined in a ball of $\dot{H}^s$ for sufficiently large $s\geq 1$. 
Moreover we prove that the Poisson bracket between a formal polynomial and the quadratic part $F_2^{(j)}$ of a Hamiltonian of the KdV hierarchy \eqref{eq: first integrals in fourier} is almost the same formal polynomial but where the coefficients $g_{\k}$ are multiplied by $-i \Omega_{j}(\k)$, where $\Omega_j(\k)$ is \eqref{def: resonant relation}.
Finally we prove that integrable formal polynomials Poisson-commute.

\begin{lemma} \label{lemma: facts about poisson bracket}
For $j,l\in\N$, let $F$ be a $j$-formal polynomial and $G$ an $l$-formal polynomial
\[
F(u) = F_2(u) + F_{\ge 3}(u) = \sum_{k \in \Z^*} a_k |u_k|^2 + \sum_{\k \in \M} f_{\k} u^{\k} \quad \mbox{and} \quad G(u) = G_2(u) + G_{\ge 3}(u) = \sum_{k \in \Z^*} b_k |u_k|^2 + \sum_{\k \in \M} g_{\k} u^{\k}
\]
satisfying \eqref{M_cond} with $M_F$ and $M_G$, respectively.
Then we have the following facts:

\medskip

(i) The function $\{ F,G \}(u)$ is a $(j+l+1)$-formal polynomial with $M_{\{F,G \}} > \max ( M_F , M_G )$, cf.~\eqref{M_cond}.



\smallskip

(ii) If $F_2(u) = \pi \sum_{k \in \Z^*} |k|^{2j} |u_k|^2$, then
\begin{equation}\label{eq: facts about poisson bracket}
\{ F_2,G_{\ge 3} \}(u) = - i \sum_{n \ge 3} \sum_{\k \in \M_n} g_{\k}\, \Omega_j(\k) u^{\k},
\end{equation}
where $\Omega_j(\k)$ is \eqref{def: resonant relation}, and this Poisson bracket is smooth in $B_s(0, \frac{1}{c_s M_G})$ for $s \ge j+l-\frac{1}{2}$.

\smallskip

(iii) Letting $I=(I_k)_{k\in\Z^*}=(|u_k|^2)_{k\in\Z^*}$, if $F$ and $G$ are integrable (cf. \Cref{def: functions in normal form})
then $\{ F,G \} = 0$.
\end{lemma}
\begin{proof}
(ii) By \Cref{lemma: convergence of hamiltonian function}, for $s\geq \max\{j,l\}$, $F(u)$ is a smooth function of $u$ in $B_s \left ( 0, \frac{1}{c_j M_F} \right )$ and $G(u)$ is a smooth function of $u$ in $B_s \left ( 0, \frac{1}{c_l M_G} \right )$, the series being absolutely convergent. 
In particular, $\pa_{u_k} F$ and $\pa_{u_k} G$ are well defined, cf. \Cref{diffuk_Frechet}. We therefore have:
\begin{equation} \label{eq:tonelli1}
\{ F_2,G_{\ge 3} \} (u) = \sum_{r \in \Z^*} ir^{2j+1}\overline{u_r} \frac{\pa G_{\ge 3}}{\pa u_{-r}} = \sum_{r \in \Z^*} ir^{2j+1} \overline{u_r} \sum_{n \ge 3} \sum_{\k \in \M_n} \left [ \delta(k_1=-r) + \ldots + \delta(k_n=-r) \right ] g_{\k} \frac{u^{\k}}{u_{-r}}.
\end{equation}
We check that this series is absolutely convergent using Tonelli's theorem.
Using the fact that $u_{-r}=\overline{u_r}$, we obtain
\[
\begin{split}
& \sum_{r \in \Z^*} \sum_{n \ge 3} \sum_{\k \in \M_n}  |r|^{2j+1} \left [ \delta(k_1=-r) + \ldots + \delta(k_n=-r) \right ] |g_{\k}| |u^{\k}| \\
& \lesssim \sum_{n \ge 3} M_G^n \sum_{\k \in \M_n} \left ( |k_1|^{2j+1} + \ldots + |k_n|^{2j+1} \right ) \frac{|g_{\k}|M_G^{-n}}{\norm{\k}_{\ell^{\infty}}^{2l-2}} \norm{\k}_{\ell^{\infty}}^{2l-2} |u^{\k}| \\
& \lesssim \sup_{n \ge 3} \norm{ (g_{\k} M_G^{-\#{\k}})_{\#{\k} = n} }_{Y_n^l} \,\sum_{n \ge 3} M_G^n \sum_{\k \in \M_n} n \norm{\k}_{\ell^{\infty}}^{2j+2l-1} |u^{\k}|
\end{split}
\]
which is analogous to \eqref{key_bound_Y}. Thus, arguing as in \Cref{lemma: convergence of hamiltonian function}, the series above is absolutely convergent in $B_s(0,\frac{1}{c_s M_G})$ for $s \ge j+l-\frac{1}{2}.$
It follows that \eqref{eq: facts about poisson bracket} holds by Fubini's theorem.  In particular, $\{ F_2,G_{\ge 3} \} (u)$ is a $(j+l)$-formal polynomial. 

\medskip

\noindent (i) Firstly, notice that
\begin{equation}\label{eq:PoissonFG}
\{ F,G \} = \{ F_2+F_{\ge 3} , G_2+G_{\ge 3} \} = \{ F_2,G_2 \} + \{ F_2,G_{\ge 3} \} + \{ F_{\ge 3},G_2 \} + \{ F_{\ge 3},G_{\ge 3} \}. 
\end{equation}
Our goal is to show that each of these terms is well-defined for $u$ in $B_s \left ( 0, \frac{1}{c_s M_{\{F,G \}}} \right )$ for $s$ sufficiently large and that they are formal polynomials. Note that the terms $\{ F_2,G_{\ge 3} \}$ and $\{ F_{\ge 3},G_2 \}$ are well-defined $(j+l)$-formal polynomials by point (ii) above and \eqref{M_cond}.

We start with the first term on the right-hand side of \eqref{eq:PoissonFG}. Note that:
\[
\{ F_2,G_2 \} = \sum_{k \in \Z^*} \frac{ik}{2\pi} \, |u_k|^2 \, (a_k +a_{-k} ) ( b_{-k}+ b_{k} ) =0
\]
by symmetry $k\mapsto -k$. Since $|a_k|\lesssim |k|^{2j}$ and $|b_k|\lesssim |k|^{2l}$ (cf.~\eqref{M_cond}) the series above is absolutely convergent for $u\in \dot{H}^s$ with 
\begin{equation}\label{conds1}
s\geq j+l+1/2.
\end{equation}
In particular, $\{ F_2,G_2 \} $ is trivially a $(j+l+1)$-formal polynomial.

\smallskip

Finally, we tackle the last term on the right-hand side of \eqref{eq:PoissonFG}.
\begin{equation}\label{poissonF3G3}
\begin{split}
& \{ F_{\ge 3},G_{\ge 3} \}(u) = \left \{  \sum_{\k \in \M} f_{\k} u^{\k} ,  \sum_{\k' \in \M} g_{\k'} u^{\k'} \right \} \\
& = \sum_{r \in \Z^*} \frac{ir}{2\pi} \left ( \sum_{\k \in \M} \left ( \delta(k_1=r)+\ldots+\delta(k_{\#\k}=r) \right ) f_{\k} \right ) \left ( \sum_{\k' \in \M} \left ( \delta(k_1'=-r)+\ldots+\delta(k_{\#\k'}'=-r) \right ) g_{\k} \right ) \frac{u^{\k} u^{\k'}}{|u_r|^2}.
\end{split}
\end{equation}
We next show that the series converges absolutely. To do so, we note that for $\k,\k'\in\M$
\begin{equation}\label{eq:A_poisson}
\begin{split}
\sum_{r\in\Z^*} |r|\,  [\delta(k_1=r) + \ldots + \delta(k_{\#\k}=r) ] \, &[\delta(k_1'=-r) + \ldots + \delta(k_{\#\k'}'=-r)] \frac{\lvert u^{\k} u^{\k'} \rvert}{|u_r|^2}\\
& = \sum_{\alpha=1}^{\#\k} \sum_{\beta=1}^{\#\k'} |k_{\alpha}|^{\frac{1}{2}} |k_{\beta}'|^{\frac{1}{2}} \delta(k_{\alpha}=-k_{\beta}') \frac{\lvert u^{\k} u^{\k'} \rvert}{\lvert u_{k_\alpha} u_{k_\beta'} \rvert},
\end{split}
\end{equation}
We are ready to prove the absolute convergence of \eqref{poissonF3G3}, using Tonelli's theorem:
\begin{align}
\eqref{poissonF3G3} & \stackrel{\eqref{eq:A_poisson}}{\lesssim} \sum_{\k \in \M} \sum_{\k' \in \M}  |f_{\k}| |g_{\k}|\,\sum_{\alpha=1}^{\#\k} \sum_{\beta=1}^{\#\k'} |k_{\alpha}|^{\frac{1}{2}} |k_{\beta}'|^{\frac{1}{2}} \delta(k_{\alpha}=-k_{\beta}') \frac{\lvert u^{\k} u^{\k'} \rvert}{\lvert u_{k_\alpha} u_{k_\beta'} \rvert}\nonumber \\
& \stackrel{\eqref{M_cond}}{\lesssim} \sum_{\k \in \M} \sum_{\k' \in \M} M_F^{\#\k} M_G^{\#\k'} \norm{\k}_{\ell^{\infty}}^{2j-2} \norm{\k'}_{\ell^{\infty}}^{2l-2} \, \sum_{\alpha=1}^{\#\k} \sum_{\beta=1}^{\#\k'} |k_{\alpha}|^{\frac{1}{2}} |k_{\beta}'|^{\frac{1}{2}} \delta(k_{\alpha}=-k_{\beta}') \frac{\lvert u^{\k} u^{\k'} \rvert}{\lvert u_{k_\alpha} u_{k_\beta'} \rvert} \nonumber \\
& \lesssim \sum_{n,m\geq 3} M_F^n M_G^m  \, \sum_{\alpha=1}^n \sum_{\beta=1}^m 
\sum_{\substack{\k\in \M_n \\ \k' \in \M_m} } \norm{\k}_{\ell^{\infty}}^{2j-\frac32} \norm{\k'}_{\ell^{\infty}}^{2l-\frac32}\,
\delta(k_{\alpha}=-k_{\beta}') \frac{\lvert u^{\k} u^{\k'} \rvert}{\lvert u_{k_\alpha} u_{k_\beta'} \rvert}. \label{eq:FGabs_pre} \end{align}
We now note that the inner sum in $\k\in \M_n$ and $\k' \in \M_m$ is independent of $\alpha$ and $\beta$ on account of the permutation invariance of $\M_n,\M_m$. 
As a result,
\begin{equation}\label{eq:FGabs}
\eqref{eq:FGabs_pre}\lesssim \sum_{\k \in \M} \sum_{\k' \in \M} \#\k \#\k' M_F^{\#\k} M_G^{\#\k'} \delta(k_1=-k_1') \norm{\k}_{\ell^{\infty}}^{2j-\frac{3}{2}} \norm{\k'}_{\ell^{\infty}}^{2l-\frac{3}{2}} \frac{|u^{\k} u^{\k'}|}{|u_{k_1}|^2}.
\end{equation}
We now define the new vector $\bar{\k} = (k_2,\ldots,k_n,k_2',\ldots,k_m')$, and note that $\forall \k \in \M_n$,
\begin{equation}\label{eq:jensen2}
   \frac{\norm{\k}_{\ell^{\infty}}}{\max ( |k_2|,\ldots,|k_{n}| ) }
\stackrel{\eqref{eq: jensen inequality}}{\le}\ (n-1).
\end{equation}
As a result,
\[
\begin{split}
\eqref{eq:FGabs} \stackrel{\eqref{eq:jensen2}}{\lesssim} & \sum_{\alpha \ge 4} \sum_{ \substack{ n,m \ge 3 \\ n+m-2=\alpha } } n^{2j-\frac12} m^{2l-\frac12} M_F^n M_G^m \sum_{\bar{\k} \in \M_{\alpha}} \max ( |\bar{k}_1|,\ldots,|\bar{k}_{\alpha}| )^{2j+2l-3} |u^{\bar{\k}}|\\
\lesssim& \sum_{\alpha \ge 4} \alpha^{2j+2l-1} \max( M_F,M_G)^{\alpha+2}\sum_{ \substack{ n,m \ge 3 \\ n+m-2=\alpha } } \sum_{\bar{\k} \in \M_{\alpha}} \norm{\bar{\k}}_{\ell^{\infty}}^{2j+2l-3} |u^{\bar{\k}}| \\
\lesssim& \sum_{\alpha \ge 4} \sum_{\bar{\k} \in \M_{\alpha}} \alpha^{2j+2l} \max( M_F,M_G)^{\alpha+2} \norm{\bar{\k}}_{\ell^{\infty}}^{2j+2l-3} |u^{\bar{\k}}| .
\end{split}
\]
which is analogous to \eqref{key_bound_Y}. Thus, arguing as in \Cref{lemma: convergence of hamiltonian function}, the series is absolutely convergent for $u\in B_s \left ( 0, \frac{1}{c_s M_{\{F,G \}}} \right )$ with $M_{\{F,G\}} > \max ( M_F,M_G )$ and 
\begin{equation}\label{conds2}
s\geq j+l-\frac32.
\end{equation}

Having established the absolute convergence of the series in \eqref{poissonF3G3}, we may exchange the order of summation and conclude that $\{ F_{\ge 3},G_{\ge 3} \}(u)$ is a $(j+l-1)$-formal polynomial of the form:
\begin{align}
\{ F_{\ge 3},G_{\ge 3} \}(u) & = \sum_{\alpha\geq 4} \sum_{\bar{\k}\in\mathcal{M}_{\alpha}} d_{\bar{\k}} u^{\bar{\k}}, \quad \mbox{where} \label{dk parentesi FG}\\
d_{\bar{\k}}  = \frac{-i}{2\pi}\!\sum_{\gamma = 1}^n \sum_{\beta=1}^m \big (\sum_{j=1}^{n-1} \bar{k}_j\big )\cdot & f(\bar{k}_1,\ldots,\bar{k}_{\gamma-1},\!\!\!\sum_{j=n}^{n+m-2} \!\!\!\!\bar{k}_j,\bar{k}_{\gamma},\ldots,\bar{k}_{n-1}) 
g(\bar{k}_n,\ldots,\bar{k}_{n+\beta-2},\!\sum_{j=1}^{n-1} \bar{k}_j,\bar{k}_{n+\beta-1},\ldots,\bar{k}_{n+m-2}).\nonumber
\end{align}
The proof is concluded by choosing the largest integer among \eqref{conds1}, \eqref{conds2} and point (ii). 
\medskip

\noindent (iii) By point (i), $\{ F,G \}(u)$ is a smooth function in some open ball of $\dot{H}^s$ for $s$ sufficiently large. 
By the absolute convergence of the series defining $\{ F,G \}(u)$ and the bilinearity of the Poisson bracket, it suffices to prove that
\[
\{ I^{\k},I^{\k'} \} = 0
\]
for any $\k \in (\Z^*)^n$ and $\k' \in (\Z^*)^m$. Moreover, by the Leibniz rule:
\[
\{ I^{\k},I^{\k'} \} = \sum_{i=1}^n\sum_{j=1}^m \{ I_{k_i}, I_{k_j'}\} \prod_{\alpha\neq i,\ \beta\neq j} I_{k_{\alpha}} I_{k_{\beta}'}.
\]
Therefore it suffices to prove that 
for any $k_1,k_2 \in \N$,
\[
\{ u_{k_1} \overline{u}_{k_1} , u_{k_2} \overline{u}_{k_2} \} = 0,
\]
which follows from a straightforward computation using \eqref{intro: kdv hamiltonian in fourier}.
\end{proof}

\begin{rem} \label{rem: derivatives poisson bracket}
Note that in general the Poisson bracket loses derivatives, in the sense that any $j_1,j_2$-formal polynomials $F$ and $G$ are well defined in a ball around the origin of $\dot{H}^s(\T)$ if $s = \max ( j_1,j_2 )$, while a priori $\{ F,G \}$ is not.
This is an important obstacle to the convergence of the standard Birkhoff normal form procedure \eqref{eq: taylor formula for Hnew integral remainder}, as it requires an infinite amount of Poisson brackets.
\end{rem}

The key idea is that these brackets are not between \emph{any} formal polynomials, but between the KdV Hamiltonian and the (approximate) solution of a homological equation, see \eqref{claim} below. These appoximate solutions will be chosen in such a way to avoid the loss of derivatives, see \Cref{lemma: stabilità parentesi poisson}.

\subsection{Birkhoff normal form} \label{section: rigorous theory}

In this section we construct a symplectic transformation $\Phi$ in a neighborhood of the origin of $\dot{H}^s$, for $s>0$ large enough, such that the KdV Hamiltonian \eqref{intro: kdv hamiltonian in fourier}  in the new variables is, up to some remainders, in normal form. 

\Cref{thm: normal form for ham} below is a corollary of \Cref{lemma: orders 3 and 4}, \Cref{lemma: h4} and the more general \Cref{prop: justification normal form} below (see also \Cref{rmk: improved normal form}), which put the whole KdV hierarchy in normal form. 
The rest of this section is devoted to the proof of \Cref{thm: normal form for ham}, which has interesting implications on the dynamics, as explained in \Cref{subsection: dynamics}.

\begin{theorem} \label{thm: normal form for ham}
Let $r \ge 3$, $s \ge 2r$. There exists $N_0=N_0(r)$ and a constant $c(r,s)>0$ such that for all $N\geq N_0$ and all $\epsilon>0$ such that $\epsilon\leq c(r,s) N^{-1}$ the following holds.
There exists an invertible symplectic map $\Phi : B_s(0,\epsilon) \rightarrow B_s(0,2\epsilon)$ with inverse $\Phi^{-1} : B_s(0,\epsilon/2) \rightarrow B_s(0,\epsilon)$
which are close to the identity in the sense that
\begin{equation}
\norm{\Phi^{\pm 1}(u)-u}_{\dot{H}^s} \lesssim_{r,s} N \norm{u}_{\dot{H}^s}^2,
\label{eq: close to the identity - thm ham}
\end{equation}
and Lipschitz-continuous:
\begin{equation} \label{eq: lipschitzianity}
\norm{\Phi^{\pm 1}(u) - \Phi^{\pm 1}(\underline{u})}_{\dot{H}^s} \le 2 \norm{u-\underline{u}}_{\dot{H}^s}.
\end{equation}
Moreover, 
\begin{equation}
\H \circ \Phi = \sum_{n = 1}^{\lfloor \frac{r}{2} \rfloor } \widehat{\H}_{2n} + R^{(\mu_3 \ge N^{1/2r})} + R^{(I \ge N^{1/2r})} + R_{\ge r+1},
\label{hamiltoniana dopo tutte le trasformazioni - thm ham}
\end{equation}
where
\begin{align}
\widehat{\H}_2(u) &= \pi \sum_{k \in \Z^*} k^{2} |u_k|^2, \label{parte quadratica thm ham}\\
\widehat{\H}_4(u) &= -\frac{\pi}{12} \sum_{k \in \Z^*} \frac{|u_k|^4}{k^2}, \label{parte quartica - thm ham} \\
\widehat{\H}_{2n}(u) &= \widehat{\H}_{2n}(I) = \sum_{\k \in (\Z^*)^n} b_{\k} I^{\k} \qquad \text{where $I_k = |u_k|^2$ and $|b_{\k}| \lesssim_n N^{2n-3}$}, \label{parte integrabile 2 - thm ham} \\
R^{(\mu_3 \ge N^{1/2r})}(u) &= \sum_{ \substack{ \k \in \mathcal{D} \\ 5 \le \#\k \le r \\ |k_3| \ge N^{\frac{1}{2r}}} } c_{\k} u^{\k} \qquad \text{with $|c_{\k}| \lesssim_{\#\k} N^{\#\k-3}$}, \label{terzo indice grande - thm ham} \\
R^{(I \ge N^{1/2r})}(u) &= \sum_{m = \lceil N^{\frac{1}{2r}} \rceil}^{+\infty} \, \, \sum_{ \substack{ \k \in \mathcal{D} \\ 3 \le \#\k \le r-2} } c_{m,\k} I_m u^{\k} \qquad \text{with $|c_{m,\k}| \lesssim_{\#\k} N^{\#\k-1}$}, \label{azione grande - thm ham} \\
R_{\ge r+1}(u) &= \sum_{ \substack{ \k \in \M \\ \#\k \ge r+1 } } c_{\k} u^{\k} \qquad \text{where $|c_{\k}| \le \rho^{\#\k} N^{\#\k-3}$ with $\rho \lesssim_r 1$} \label{resto grande - thm ham}
\end{align}
and all the Hamiltonian terms have real coefficients and satisfy the reality condition \eqref{def: reality condition}, cf. \Cref{rmk: reality condition} below.
\end{theorem}

\begin{rem}\label{rk: megaremark sotto 210}
Let us make some comments.
\begin{itemize}
    \item The transformation $\Phi$ is defined in a neighborhood of the origin of $\dot{H}^s(\T)$ of maximal radius $\sim_{r,s} N^{-1}$. 
    Equation \eqref{eq: close to the identity - thm ham} tells us that $\Phi$ is close to the identity provided that $N$ is not too large.
    In \Cref{teorema: approssimazione con fasi lineari in t - con eps}, we will carefully fix a value $N$ in order to obtain valuable information on the dynamics of the KdV equation over long timescales.
    
    \item It is possible to weaken the condition $s\geq 2r$ to $s\geq 1$. In particular, $\Phi$ also maps $B_{\sigma}(0,\epsilon)$ to $B_{\sigma}(0,2\epsilon)$ for any $1\leq \sigma\leq s$. We did not optimize the regularity since our approach will require a large $s$ in order to control the dynamics over long timescales, cf. \Cref{teorema: approssimazione con fasi lineari in t - con eps}.
    See \Cref{rk: low regularity} for additional details on how to lower the regularity. 
    
    \item The formal polynomials \eqref{terzo indice grande - thm ham}-\eqref{azione grande - thm ham} are supported on $\M_n^{(\mu_3 \ge N^{\frac{1}{2r}})} \cup \M_n^{(I \ge N^{\frac{1}{2r}})}$, cf. \eqref{insieme terzo indice grande}-\eqref{insieme azione grande}. 
    They are remainders if $N$ is sufficiently large, as proved in \Cref{prop: controllo resti}.
    In particular, they are almost negligible for the dynamics, cf. \Cref{thm: dynamics in the new variables}.
    \item Larger values of $r$ lead to good approximations over longer timescales (cf. \Cref{teorema: approssimazione con fasi lineari in t - con eps}), but require more regularity $s\geq 2r$ to be well-defined. 
    For fixed $r$, $R_{\ge r+1}$ in \eqref{resto grande - thm ham} is a 1-formal polynomial of order $\ge r+1$, which is small in a neighborhood of the origin. Moreover, the flow of $R_{\ge r+1}$ almost preserves the $\dot{H}^s$-norm over long timescales, cf. \Cref{prop: controllo resti}.
    
    \item The terms \eqref{parte integrabile 2 - thm ham} are integrable (cf. \Cref{def: functions in normal form}). By \Cref{lemma: facts about poisson bracket}-(iii), they Poisson-commute with the $\dot{H}^s$-norm.
    This means that, if there were no remainders, the $\dot{H}^s$-norm of the flow of the Hamiltonian would be conserved.
    In particular all the Fourier moduli would be conserved by the flow and the dynamics would affect only the phases.
    We will see that this is approximately true, at least for polynomial timescales $T \sim \norm{u_0}_{\dot{H}^s}^{-m}$, for arbitrary $m \in \N$ and initial datum $u_0$ sufficiently regular and small.
\end{itemize}
\end{rem}

\medskip

The mapping $\Phi$ in \Cref{thm: normal form for ham} is the composition of several maps 
\begin{equation}\label{eq:Phi_comp}
    \Phi=\Phi_3 \circ \Phi_{4} \circ \ldots \circ \Phi_r, \qquad r\geq 3.
\end{equation}
Following the Birkhoff normal form procedure, each transformation $\Phi_n$ is the composition of time-1 flows associated to auxiliary Hamiltonians, so that the dynamics in the new coordinates is simpler. 

\paragraph{The normal form.}\label{subsection: the normal form}
Let us explain the basic idea to construct the transformations $\Phi$ in \eqref{eq:Phi_comp}. 
Given coordinates $v\in \dot{H}^s$, for some $s\geq 0$, the coordinates $u=u(\xi)|_{\xi=1}=\Phi_G^{\xi}(v)|_{\xi=1}$ are the solution to the initial-value problem:
\[
\begin{cases}
\frac{\mathrm{d}}{\mathrm{d} \xi} u_k = \{ u_k,G \} \\
u_k (\xi)|_{\xi=0} = v_k
\end{cases}
\]
where $G$ is an appropriately chosen Hamiltonian, and where we use the letter $\xi$ to indicate the auxiliary time of the flow of $G$, to avoid confusion with the time $t$ of the KdV equation.

The derivatives of such transformations are related to the Poisson bracket in the following way:
\[
\frac{\mathrm{d}^{\alpha}}{\mathrm{d}\xi^{\alpha}}(\H \circ \Phi_G^{\xi}) =  \{ \{ \{ \H, \underbrace{G \}, G \}\ldots\}}_{\alpha\ \mathrm{times}} \circ \Phi_G^{\xi} = \mathrm{ad}_G^{\alpha} \H \circ \Phi_G^{\xi}, \qquad \alpha\in\N,
\]
where $\mathrm{ad}_{G} = \{ \cdot, G \}$.
This allows us to compute the Taylor formula up to order $m\in\N$ centered at $\xi = 0$, which we evaluate at $\xi=1$:
\begin{equation}
\H \circ \Phi_G^1 = \H + \sum_{\alpha=0}^{m} \frac{1}{\alpha!}  \,\mathrm{ad}_G^{\alpha} \H + \frac{1}{m!} \, \int_0^1 (1-\xi)^m \,\mathrm{ad}_{G}^{m+1} \H \circ \Phi_G^{\xi} \, \mathrm{d} \xi,
\label{eq: taylor formula for Hnew integral remainder}
\end{equation}
provided that $\Phi_G^{\xi}$ is well defined up to $\xi=1$. 

The rigorous justification of \eqref{eq: taylor formula for Hnew integral remainder} for the transformations in \eqref{eq:Phi_comp} follows from \Cref{lemma: transformations preserve Hs}. 
The convergence of the Taylor series -- including bounds on the remainder -- are some of the key elements in the proof of \Cref{thm: normal form for ham}, cf. \Cref{conv taylor series} and \Cref{prop: justification normal form}.

The transformation $\Phi_G$ will be proved to be invertible in a neighborhood around the origin of $\dot{H}^s(\T)$ (cf. \Cref{prop: justification normal form}), and it preserves the Poisson bracket, a property which will greatly simplify our computations: for two formal polynomials $F_1$ and $F_2$ we have:
\begin{equation}
\{ F_1,F_2 \} \circ \Phi_G^{\xi} = \{ F_1 \circ \Phi_G^{\xi}, F_2 \circ \Phi_G^{\xi} \}. 
\label{eq: hamiltonian flows preserve poisson structure}
\end{equation}
This follows from the fact that $\Phi_G^{\xi}$ is symplectic, together with the convergence of the Taylor series \eqref{eq: taylor formula for Hnew integral remainder} and elementary properties of formal polynomials, see \Cref{subsection: preliminaries}.

\medskip

Let us briefly explain how to construct the maps $\Phi_n$ in \eqref{eq:Phi_comp}. The maps $\Phi_3$ and $\Phi_4$, which are time-1 flows associated to auxiliary Hamiltonians $G_3$ and $G_4$, are simple because the resonant set $\cR_n^1$, $n=3,4$, is explicit since the resonant relation $\Omega (\k)$ in \eqref{def: resonant relation} factorizes for any $\k \in \M_n$, see \Cref{lemma: orders 3 and 4}. In particular, the Hamiltonians $G_n$, $n=3,4$, are explicit and gain derivatives, which compensates the loss of derivatives in the Poisson bracket \eqref{intro: kdv hamiltonian in fourier}.

For $n\geq 5$, the resonant relation $\Omega (\k)$ does not factorize, and compensating the loss of derivatives in \eqref{intro: kdv hamiltonian in fourier} becomes a major obstacle. 
In order to overcome this, we construct 
\begin{equation}\label{eq:Phin}
    \Phi_n=\Phi_n^{(1)} \circ \ldots \circ \Phi_n^{(n-1)},
\end{equation} 
where each $\Phi_n^{(l)}$ is a time-1 flow of a Hamiltonian $G_n^{(l)}$ in \eqref{eq: other lie transforms}, $l=1,\ldots,n-1$. This Hamiltonian is supported on a special set of indices $\mathcal{J}_{n,l,N}$ defined in \eqref{def: set JnlN}, in such a way that $\Phi_n$ cancels the monomials of degree $n$ in the Hamiltonian $\H \circ \Phi_1\circ\ldots\circ\Phi_{n-1}$ supported on the set $\mathcal{J}_{n,l,N}$. The rest of this subsection is devoted to proving two key properties of these transformations:

\begin{itemize}
    \item $\Phi_n$ does not cancel any monomials in $\H\circ \Phi_1\circ\ldots\circ\Phi_{n-1}$ supported on 
    \begin{equation}\label{eq:complement of J-indices}
    \mathfrak{J}_{n,N} = \M_n \setminus \left ( \bigcup_{l = 1}^{n-1} \mathcal{J}_{n,l,N} \right ) \ .
    \end{equation}
    In particular, we shall prove that all terms supported on $\mathfrak{J}_{n,N}$ are either integrable \eqref{parte integrabile 2 - thm ham} or correspond to remainders \eqref{terzo indice grande - thm ham}-\eqref{azione grande - thm ham}, cf. \Cref{teorema multiindici rimasti}.
    
    \item The choice of set of indices $\mathcal{J}_{n,l,N}$ guarantees that the Hamiltonian vector field associated to $G_n^{(l)}$ maps $\dot{H}^s$ into itself (cf. \Cref{lemma: transformations preserve Hs}), and that the Poisson bracket of $G_n^{(l)}$ with formal polynomials does not lose derivatives (cf. \Cref{lemma: stabilità parentesi poisson}).
\end{itemize}

We start by constructing the sets $\mathcal{J}_{n,l,N}$ for all $n\geq 3$. As mentioned before, the cases $n=3,4$ are simpler and can be treated directly, cf. \Cref{subsection: cases 3 and 4}. However, since our construction is necessary for $n\geq 5$, we prefer to present all cases simultaneously.

\paragraph{The support of the maps $\Phi_n$.} 
For $n \ge 3$, $l \in \N$, and $N \gg 1$ to be fixed later, we define:
\begin{equation} \label{def: set JnlN}
\mathcal{J}_{n,l,N} = \left \{ \k \in \M_n \setminus \mathcal{R}_n^l \, \bigg | \,  \dfrac{\norm{\k}_{\ell^{\infty}}^{2l-1}}{|\Omega_l(\k)|} \le N \right \} \setminus \left (  \M_n^{(\mu_3 \ge N^{\frac{1}{2n}})} \cup \M_n^{(I \ge N^{\frac{1}{2n}})} \right ), 
\end{equation}
where the resonant relation $\Omega_l(\k)$ is defined in \eqref{def: resonant relation}, and the sets of indices $\M_n$, $\mathcal{R}_n$, etc. are defined in \eqref{def: indices zero momentum}, \eqref{def: resonant indices zero momentum}, \eqref{insieme terzo indice grande} and \eqref{insieme azione grande}.
Notice that for $l = 1$, the set \eqref{def: set JnlN} is a subset of the set $\mathcal{J}_{n,N}$ in \cite[Definition 3.5]{bernier}. 
For $l>1$, this set is a generalization of \cite[Definition 3.5]{bernier} which allows more general resonant relations $\Omega_l (\k)$ and polynomial growth in the coefficients (instead of $\ell^{\infty}$).

\smallskip

Our first result is a characterization of the set of indices $\mathfrak{J}_{n,N}$ in \eqref{eq:complement of J-indices}.
Let 
\begin{equation} \label{def: normal form indices}
\cR_n = \bigcap_{l=1}^{n-1} \cR_n^l.
\end{equation}

\begin{theorem} \label{teorema multiindici rimasti}
Let $n \ge 3$, $l \in \{ 1, \ldots, n-1 \}$. There exists $N_0=N_0(n)$ such that for all $N\geq N_0$,
\begin{equation}\label{eq:multiindici rimasti}
\mathfrak{J}_{n,N} = \M_n \setminus \left ( \bigcup_{l = 1}^{n-1} \mathcal{J}_{n,l,N} \right ) = \cR_n \sqcup \mathfrak{C}_{n,N}
\end{equation}
where $\mathfrak{C}_{n,N} \subseteq \M_n^{(\mu_3 \ge N^{\frac{1}{2n}})} \cup \M_n^{(I \ge N^{\frac{1}{2n}})}$ and $\cR_n=\emptyset$ if $n$ is odd. Moreover, if $n$ is even and
$\k \in \cR_n$, then the corresponding monomial is integrable (cf. \Cref{def: functions in normal form}), i.e. $u^{\k} = I^{\k'}$ for some $\k' \in (\Z^*)^{\frac{n}{2}}$.
\end{theorem}

\Cref{teorema multiindici rimasti} directly follows from \Cref{cor: system} and \Cref{cor: azione grande o tre indici grandi} below.

\smallskip

Firstly,  we study the set $\cR_n$. By \eqref{def: indices zero momentum}, \eqref{def: resonant indices zero momentum} and \eqref{def: resonant relation},
\[
\k \in \cR_n \qquad \iff \qquad \Omega_l(\k)=0 \quad  \mbox{for all}\ l=0,1,\ldots ,n-1.
\]
Thanks to the following theorem, which is proven in \Cref{app: proof system}, we can conclude that if $\k \in \cR_n$, then the corresponding monomial is integrable (cf.~\Cref{def: functions in normal form}). See also \cite[Section~4.1]{feola} and \cite[Proposition~3.4]{feola2} for similar results in the context of the Degasperis–Procesi equation.

\begin{theorem} \label{cor: system}
Let $n\in\N$. Consider the system of equations
\[
\begin{cases}
k_1+k_2+\ldots+k_n &= 0 \\
k_1^3+k_2^3+\ldots+k_n^3 & =0 \\
& \vdots \\
k_1^{2n-1}+k_2^{2n-1}+\ldots+k_n^{2n-1} & =0,
\end{cases}
\]
with the additional assumption that $k_j \neq 0$ for all $j=1,\ldots, n$. Then, we have the following:
\begin{itemize}
    \item if $n$ is odd, there is no solution;
    \item if $n=2m$ is even, all  solutions $(k_1,k_2,\ldots,k_{2m})$ are paired, meaning that there exists a permutation $\sigma \in \mbox{Sym}(2m)$ such that $k_{\sigma(2j-1)}+k_{\sigma(2j)}=0$ for $j=1,\ldots,m$.
\end{itemize} 
\end{theorem} 

\smallskip

Next we define the set $\mathfrak{C}_{n,N}$ in \eqref{eq:multiindici rimasti} as
\begin{equation}
\mathfrak{C}_{n,N} =
\left ( \M_n^{(\mu_3 \ge N^{\frac{1}{2n}})} \cup \M_n^{(I \ge N^{\frac{1}{2n}})} \right ) \setminus \cR_n.
\end{equation}
\Cref{cor: azione grande o tre indici grandi} below, proved in \Cref{app: resti}, guarantees that $\mathfrak{J}_{n,N}\setminus \cR_n = \mathfrak{C}_{n,N}$, thereby proving \eqref{eq:multiindici rimasti}. 



\smallskip

\begin{prop} \label{cor: azione grande o tre indici grandi}
Let $n \ge 3$, $l \in \{ 1,\ldots,n-1 \}$ and $\k \in \mathcal{D}_n \setminus \mathcal{R}_n^{l}$. 
There exists $N_0=N_0(n)$ such that for all $N\geq N_0$, if
\[
\left \lvert \frac{k_1^{2l-1}}{k_1^{2l+1}+\ldots+k_n^{2l+1}} \right \rvert \ge N,
\]
then either there exists $\k' \in \M_{n-2}$ such that 
\begin{equation}
u^{\k} = |u_a|^2 u^{\k'} \quad \text{where} \quad a \ge N^{\frac{1}{2n}}
\label{eq: azione grande}
\end{equation}
or 
\begin{equation}
|k_3| \ge N^{\frac{1}{2n}}.
\label{eq: terzo indice grande}
\end{equation}
\end{prop}

\Cref{cor: azione grande o tre indici grandi} is a generalization of Lemma 3.7 and Corollary 3.8 in \cite{bernier} to the various dispersive relations in the KdV hierarchy.

\paragraph{The generators of the map $\Phi_n$.}

The Hamiltonian functions we use to generate symplectic transformations $\Phi_n$ in \eqref{eq:Phin} are of the form
\begin{equation}
G_n^{(l)} = \sum_{\bm{k} \in \mathcal{J}_{n,l,N}} \frac{\widetilde{c}_{\k}^{(l)}}{i \Omega_l(\bm{k})} u^{\bm{k}},
\label{eq: other lie transforms}
\end{equation}
for $n \ge 3$ and $l =1,\ldots,n-1.$ 
In \eqref{eq: other lie transforms} the coefficients $\widetilde{c}_{\k}^{(l)}$ are the ones appearing in 
\[
\widetilde{F}_n^{(l)} = \left ( F^{(l)} \circ \Phi_{3}\circ\ldots\circ \Phi_{n-1} \circ \Phi_n^{(1)} \circ \ldots \circ \Phi_n^{(l-1)} \right )_n, 
\]
which is the homogeneous term of degree $n$ in $F^{(l)}$ after composing with the previous transformations (cf.~\eqref{eq:Phin} and \eqref{vera Gn} below). 
By \Cref{lemma: facts about poisson bracket}-(ii) and \eqref{eq: taylor formula for Hnew integral remainder}, it is clear that $G_n^{(l)}$ is constructed to erase some monomials from $\widetilde{F}_n^{(l)}$. 
In fact, we will prove that this transformation actually removes some monomials from $\widetilde{F}_n^{(j)}$ for every $1\leq j\leq r-1$. 
This key observation allows us to put the first $r-1$ first integrals of the KdV hierarchy in normal form simultaneously (cf. \Cref{prop: justification normal form}).

The first result we need is the well-posedness of the Hamiltonian flow associated with $G_n^{(l)}$ \eqref{eq: other lie transforms}, $n \ge 3$ and $l \in \{1,\ldots,n-1\}$,
at least up to time 1 in a neighborhood of the origin of $\dot{H}^s$.
This result is a generalization of \cite[Lemma 3.6]{bernier} with some important differences: 
\begin{enumerate}
\item[(i)] the coefficients $g_{\k}$ in \eqref{eq:G} do not need to be bounded (in fact they can grow polynomially in $\k$, thanks to the functional space $Y_n^l$ cf. \eqref{def: functional space});
\item[(ii)] $\Omega_l(\k)$ in \eqref{eq:G} can be any of the resonant relations in the KdV hierarchy; and
\item[(iii)] we obtain precise bounds on the Lipschitz constant of the Hamiltonian flow associated with $G$.
\end{enumerate}

The latter point is crucial to \Cref{prop: properties of random fixed point}, which is at the heart of the main result in this manuscript.

\begin{lemma} \label{lemma: transformations preserve Hs}
Let $N, s\ge 1$, $n \ge 3$ and $l \in \{ 1,2,\ldots,n-1\}$.
If 
\begin{equation}\label{eq:G}
G = \sum_{\bm{k} \in \mathcal{J}_{n,l,N}} \frac{g_{\bm{k}}}{i \Omega_l(\bm{k})} u^{\bm{k}}
\end{equation}
with $\bm{g} \in Y_n^{l}$, then its vector field $X_G = \partial_x \nabla G$ maps $\dot{H}^s$ into itself and 
\[
\norm{X_G(u)}_{\dot{H}^s} \lesssim_{n,s} N \norm{\bm{g}}_{Y_n^{l}} \norm{u}_{\dot{H}^s}^{n-1}.
\]
As a consequence, there exists $\epsilon_0 \gtrsim_{n,s} [N \norm{\bm{g}}_{Y_n^{l}}]^{-\frac{1}{n-2}}$ such that for any $\epsilon \le \epsilon_0$ and any $\xi \in [0,2]$, $G$ generates a Hamiltonian flow $\Phi_G^{\xi} : B_s(0,\epsilon) \rightarrow \dot{H}^s$ which is close to the identity in the sense that
\begin{equation}
\norm{\Phi_G^{\xi}(u)-u}_{\dot{H}^s} \lesssim_{n,s} N \norm{\bm{g}}_{Y_n^{l}} \norm{u}_{\dot{H}^s}^{n-1}.
\label{eq: close to the identity 1}
\end{equation}
Finally, for $\xi \in [0,2]$, $\Phi_G^{\xi}$ is Lipschitz-continuous and
\begin{align}
\norm{\Phi_{G}^{\xi=1}(u_1)-\Phi_{G}^{\xi=1}(u_2)}_{\dot{H}^s} \le \norm{u_1-u_2}_{\dot{H}^s} \exp \left ( C(n,s) N \norm{\bm{g}}_{Y_n^l} (\norm{u_1}_{\dot{H}^s}+\norm{u_2}_{\dot{H}^s})^{n-2} \right ) 
\label{eq: lipschitz continuity of phi}.
\end{align}
\end{lemma}
\begin{proof}
By \Cref{lemma: convergence of hamiltonian function}, $G$ is a smooth function in $\dot{H}^s$, $s \ge 1$. Indeed, for $\k \in \mathcal{J}_{n,l,N}$, we have 
$\left \lvert \frac{g_{\k}}{\Omega_l(\k)} \right \rvert \le N \norm{\bm{g}}_{Y_n^l}$. 

\medskip

\noindent {\sc Vector Field:} Next we compute the vector field generated by $G$. Setting $\k =(\k',k_{n})$ with $\k'=(k_1,\ldots,k_{n-1})$, 
\[
\begin{split}
X_G(u)_{\tr} =&\ \{ u_{\tr} , G \} = \frac{\tr}{2\pi} \sum_{(\k',-{\tr}) \in \mathcal{J}_{n,l,N}} \left [  \frac{g(-\tr,k_1,\ldots,k_{n-1})}{\Omega_l(-\tr,k_1,\ldots,k_{n-1})}+ \ldots + \frac{g(k_1,\ldots,k_{n-1},-\tr)}{\Omega_l(k_1,\ldots,k_{n-1},-\tr)} \right ] u^{\bm{k'}}.
\end{split}
\]
Therefore
\[
\begin{split}
\norm{X_G(u)}_{\dot{H}^s}^2 \le&\ \frac{1}{(2\pi)^2} \sum_{\tr \in \Z^*} |\tr|^{2s+2} \bigg \lvert \sum_{(\bm{k}',-\tr) \in \J_{n,l,N}} \left [  \frac{g(-\tr,k_1,\ldots,k_{n-1})}{\Omega_l(-\tr,k_1,\ldots,k_{n-1})}+\ldots + \frac{g(k_1,\ldots,k_{n-1},-\tr)}{\Omega_l(k_1,\ldots,k_{n-1},-\tr)} \right ] u^{\bm{k'}} \bigg \rvert^2 \\
 &\ \hspace{-0.65cm} \stackrel{\eqref{def: functional space}}{\le} \frac{n^2}{(2\pi)^2} \norm{\bm{g}}_{Y_n^l}^2 \sum_{\tr \in \Z^*} |\tr|^{2s} \left ( \sum_{(\bm{k}',-\tr) \in \J_{n,l,N}} \left \lvert \frac{\max(|k_1|,\ldots,|k_{n-1}|,|\tr|)^{2l-1}}{\Omega_l((\bm{k}',-\tr))} \right \rvert \left \lvert u^{\k'} \right \rvert  \right )^2 \\
 &\ \hspace{-0.7cm} \stackrel{\eqref{def: set JnlN}}{\le} \frac{n^2 N^2 \norm{\bm{g}}_{Y_n^l}^2}{(2\pi)^2} \sum_{\tr \in \Z^*} |\tr|^{2s} \left \lvert \sum_{k_1+\ldots+k_{n-1}=\tr} |u_{k_1}|\ldots|u_{k_{n-1}}|  \right \rvert^2\\
 &\ \hspace{-0.7cm} \stackrel{\eqref{eq: jensen inequality}}{\le} \frac{n^{2+2s}N^2 \norm{\bm{g}}_{Y_n^l}^2}{(2\pi)^2} \sum_{\tr \in \Z^*} \left \lvert \sum_{k_1+\ldots+k_{n-1}=\tr} |k_1|^s |u_{k_1}|\ldots|u_{k_{n-1}}|  \right \rvert^2.
\end{split}
\]
By the Young convolution inequality \eqref{young convolution inequality} and the fact that $s > \frac{1}{2}$ ,
\begin{equation}\label{eq: est_vector field}
\norm{X_G(u)}_{\dot{H}^s}^2 \lesssim n^{2s+2} N^2 \norm{\bm{g}}_{Y_n^l}^2 \norm{u}_{\dot{H}^s}^2 \norm{u}_{\dot{\FL}^{0,1}}^{2(n-2)} \stackrel{\eqref{eq: embedding hs in fl01}} {\lesssim} c_0^{2(n-2)} N^2 n^{2s+2}  \norm{\bm{g}}_{Y_n^l}^2 \norm{u}_{\dot{H}^s}^{2n-2}.
\end{equation}
A similar argument yields
\begin{equation}
\norm{\mathrm{d} X_G(u)}_{\mathcal{L}(\dot{H}^s)}^2 \lesssim c_0^{2(n-2)} n^{2s+4} N^2 \norm{\bm{g}}_{Y_n^l}^2 \norm{u}_{\dot{H}^s}^{2n-4}.
\label{eq: stima sul differenziale}
\end{equation}
By the Cauchy-Lipschitz Theorem the flow $\Phi_G^{\xi}$ is locally well defined in $\dot{H}^s$.

\medskip

\noindent {\sc Close to the Identity:} To prove \eqref{eq: close to the identity 1} we use a bootstrap argument. 
Let 
\[
I = \left \{ \xi \in [0,2] \, | \, \sup_{\xi' \in [0,\xi]} \norm{\Phi_G^{\xi'}(u)}_{\dot{H}^s} \le 2\norm{u}_{\dot{H}^s} \right \}.
\]
This set is clearly closed and $0 \in I$. It suffices to show that it is open: if $\xi \in I$, then
\begin{equation} \label{eq: stima close to the identity}
\norm{\Phi_G^{\xi}(u)-u}_{\dot{H}^s} \le \int_0^{\xi} \norm{X_G(\Phi_G^{\xi'}(u))}_{\dot{H}^s} \, \mathrm{d} \xi' \stackrel{\eqref{eq: est_vector field}}{\lesssim_{n,s}} N \norm{\bm{g}}_{Y_n^l} \xi \sup_{\xi' \in [0,\xi]} \norm{\Phi_G^{\xi'}(u)}_{\dot{H}^s}^{n-1}  \lesssim_{n,s} N \norm{\bm{g}}_{Y_n^l} \norm{u}_{\dot{H}^s}^{n-1}.
\end{equation}
If we take 
\[
\norm{u}_{\dot{H}^s} \lesssim_{n,s} \left ( N \norm{\bm{g}}_{Y_n^l} \right )^{-\frac{1}{n-2}} = \epsilon_0,
\]
then we can arrange $\norm{\Phi_G^{\xi}(u)-u}_{\dot{H}^s}<\norm{u}_{\dot{H}^s}$.
Therefore, $\sup_{\xi' \in [0,\xi]} \norm{\Phi_G^{\xi'}(u)}_{\dot{H}^s} <2\norm{u}_{\dot{H}^s}$, which implies that $\xi+\delta\in I$ for $\delta>0$ sufficiently small -- thus $I$ is open. Therefore $I=[0,2]$ and \eqref{eq: stima close to the identity} is valid $\forall \xi \in [0,2]$, yielding \eqref{eq: close to the identity 1}. 

\medskip

\noindent {\sc Lipschitz Property:}
For $\xi \in [0,2]$,
\[
\begin{split}
\Phi_{G}^{\xi}(u_1)-\Phi_{G}^{\xi}(u_2) = u_1 - u_2 + \int_0^{\xi} [ X_G(\Phi_{G}^{\xi'}(u_1)) - X_G(\Phi_{G}^{\xi'}(u_2)) ] \, \mathrm{d} \xi'
\end{split}
\] 
and so
\[
\begin{split}
& \norm{\Phi_{G}^{\xi}(u_1)-\Phi_{G}^{\xi}(u_2)}_{\dot{H}^s} 
\!\le \norm{u_1-u_2}_{\dot{H}^s}\! +\! \int_0^{\xi}\!\! \sup_{\norm{w}_{\dot{H}^s} \le 2(\norm{u_1}_{\dot{H}^s}+\norm{u_2}_{\dot{H}^s})}\!\! \norm{\mathrm{d} X_G(w)}_{\mathcal{L}(\dot{H}^s)} \norm{\Phi_{G}^{\xi'}(u_1)-\Phi_{G}^{\xi'}(u_2)}_{\dot{H}^s} \! \mathrm{d} \xi'
\end{split}
\]  using the fact that $\norm{\Phi_{G}^{\xi}(u_j)}_{\dot{H}^s} \le 2 \norm{u_j}_{\dot{H}^s}$ for $\xi \in [0,2]$ and $j \in \{ 1,2 \}.$
The Gronwall inequality and \eqref{eq: stima sul differenziale} yield \eqref{eq: lipschitz continuity of phi}.
\end{proof}

\begin{rem} \label{rmk: reality condition}
If the coefficients $g_{\k}$ in \eqref{eq:G} satisfy the reality condition \eqref{def: reality condition} then, by \eqref{def: resonant relation}, $\dfrac{g_{\k}}{i \Omega_l(\k)}$ satisfies the reality condition and hence $G$ is real valued.
It is clear by the definition \eqref{intro: kdv hamiltonian in fourier} that if two formal polynomials $F$ and $G$ satisfy the reality condition, then their Poisson bracket satisfies it.
Since in this manuscript we will only use formal polynomials from the KdV hierarchy (see \Cref{lem:first_integrals}) and we will transform them using symplectic transformations generated by the flow of some real valued Hamiltonian $G$ \eqref{eq:G}, then by \eqref{eq: taylor formula for Hnew integral remainder} we immediately see that all the formal polynomials we consider satisfy the reality condition \eqref{def: reality condition}.
\end{rem}

Our next result allows us to control the Poisson bracket between a homogeneous  $m$-formal polynomial and a homogeneous $l$-formal polynomial supported on the set of indices \eqref{def: set JnlN}, such as the Hamiltonian $G_n^{(l)}$ in \eqref{eq: other lie transforms} (see \Cref{prop: justification normal form}).
Thanks to the choice of the set $\mathcal{J}_{n,l,N}$, we don't lose any derivatives, as anticipated by \Cref{rem: derivatives poisson bracket}.
The proof is a generalization of \cite[Lemma 3.9]{bernier}.

\begin{lemma} \label{lemma: stabilità parentesi poisson}
Let $N \ge 2$, $m \in \N$, $s \ge m$ and $n,r \ge 3$.
Let $l \in \{ 1,2,\ldots,n-1\}$. 
Consider $u \in \dot{H}^s(\T)$ and
\[
P(u) = \sum_{\k \in \mathcal{M}_r} c_{\k} u^{\k}, \qquad G(u) = \sum_{\k \in \mathcal{J}_{n,l,N}} \frac{g_{\k}}{i \Omega_l(\k)} u^{\k},
\]
with $\bm{g} \in Y_n^l$ and $\bm{c} \in Y_r^m$. 
Then
\[
\{ P,G \}(u) = \sum_{\bar{\k} \in \mathcal{M}_{n+r-2}} d_{\bar{\k}} u^{\bar{\k}}
\]
where $\bm{d} \in Y_{n+r-2}^m$ and 
\begin{equation}\label{d_est}
\norm{\bm{d}}_{Y_{n+r-2}^m} \le \frac{1}{2\pi} N n^{2m-1} r \norm{\bm{g}}_{Y_n^l} \norm{\bm{c}}_{Y_r^m}.
\end{equation}
\end{lemma}
\begin{proof}
We have that
\[
\begin{split}
\{ P,G \}(u) =&\ \sum_{j \in \Z^*} \frac{ij}{2\pi} \sum_{\alpha = 1}^r \sum_{\beta=1}^n \sum_{ \substack{ \k \in \mathcal{M}_r \\ \k' \in \mathcal{J}_{n,l,N} \\ k_{\alpha} = j = -k'_{\beta}}} \frac{c_{\k} g_{\k'}}{i \Omega_l(\k')} \frac{u^{\k} u^{\k'}}{u_{k_{\alpha}}u_{k_{\beta}'}} = - \frac{i}{2\pi} \sum_{\alpha = 1}^r \sum_{\beta=1}^n \sum_{ \substack{\k \in \mathcal{M}_r \\ \k' \in \mathcal{J}_{n,l,N} \\ k_{\alpha} = -k'_{\beta}}} k_{\beta}' \frac{c_{\k} g_{\k'}}{i \Omega_l(\k')} \frac{u^{\k} u^{\k'}}{u_{k_{\alpha}}u_{k_{\beta}'}}.
\end{split}
\]
As in \eqref{dk parentesi FG}, we can reparametrize the latter sum using a variable $\bar{\k}\in\mathcal{M}_{n+r-2}$
\[
\begin{split}
\{ P,G \}(u) =&\ - \frac{1}{2\pi} \sum_{ \substack{ \bar{\k} \in \mathcal{M}_{n+r-2} \\ (\bar{k}_r,\ldots,\bar{k}_{r+n-2},\sum_{j=1}^{r-1} \bar{k}_j) \in \J_{n,l,N} } } \sum_{\alpha = 1}^r \sum_{\beta=1}^n c(\bar{k}_1,\ldots,\bar{k}_{\alpha-1},\sum_{j=r}^{r+n-2} \bar{k}_j,\bar{k}_{\alpha},\ldots,\bar{k}_{r-1}) \\
& \cdot \big (\sum_{j=1}^{r-1} \bar{k}_j\big )\,\frac{g(\bar{k}_r,\ldots,\bar{k}_{r+\beta-2},\sum_{j=1}^{r-1} \bar{k}_j,\bar{k}_{r+\beta-1},\ldots,\bar{k}_{r+n-2})}{\Omega_l(\bar{k}_r,\ldots,\bar{k}_{r+\beta-2},\sum_{j=1}^{r-1} \bar{k}_j,\bar{k}_{r+\beta-1},\ldots,\bar{k}_{r+n-2})} u^{\bar{\k}}
=: \sum_{ \bar{\k} \in \mathcal{M}_{n+r-2} } d_{\bar{\k}} u^{\bar{\k}}
\end{split}
\]
renaming each time the indices. 
Finally 
\begin{equation}\label{d_est_c}
|d_{\bar{\k}}| \le \frac{1}{2\pi} Nn \norm{\bm{g}}_{Y_n^l} \sum_{\alpha = 1}^r | c(\bar{k}_1,\ldots,\bar{k}_{\alpha-1},\bar{k}_r+\ldots+\bar{k}_{r+n-2},\bar{k}_{\alpha},\ldots,\bar{k}_{r-1}) |
\end{equation}
by the definition of $\J_{n,l,N}$. 
Note that 
\begin{equation} \label{est_c}
\begin{split}
\frac{| c(\bar{k}_1,\ldots,\bar{k}_{\alpha-1},\bar{k}_r+\ldots+\bar{k}_{r+n-2},\bar{k}_{\alpha},\ldots,\bar{k}_{r-1}) |}{\max ( |\bar{k}_1|,\ldots,|\bar{k}_{r-1}|,|\bar{k}_r+\ldots+\bar{k}_{r+n-2}| )^{2m-2}} & \frac{\max ( |\bar{k}_1|,\ldots,|\bar{k}_{r-1}|,|\bar{k}_r+\ldots+\bar{k}_{r+n-2}| )^{2m-2}}{\max ( |\bar{k}_1|,\ldots,|\bar{k}_{r+n-2}| )^{2m-2} } \\
\stackrel{\eqref{eq: jensen inequality}}{\le}&\ (n-1)^{2m-2} \norm{\bm{c}}_{Y_r^m}.
\end{split}
\end{equation}
Bounds \eqref{d_est_c}-\eqref{est_c} yield \eqref{d_est}.
\end{proof}

\paragraph{Main normal form theorem for the KdV hierarchy.}
Now we are ready for the most important result of this section. 
We prove that \Cref{thm: normal form for ham} holds actually for all $F^{(j)}$ in the KdV hierarchy \eqref{eq: first integrals in fourier}, for $1 \le j \le r-1$.
In particular we provide a symplectic transformation $\Phi$, invertible in a neighborhood of the origin of $\dot{H}^s(\T)$, $s \ge 2r$, close to the identity (in the sense of \eqref{eq: close to the identity}), and we give a precise estimate of the coefficients of the transformed Hamiltonian.
The proof is a generalization of \cite[Proposition 3.10, Theorem 2]{bernier}.
The main difference is that we have to prove that the KdV hierarchy is transformed simultaneously, since we use it to construct each transformation (cf. \Cref{prop: magic formula}).
Since $F^{(j)}$ has not bounded coefficients, we have to work in our more general functional setting (cf. \eqref{def: functional space}).

\begin{theorem} \label{prop: justification normal form}
Let $r \ge 3$, $s \ge 2r$. There exists $N_0=N_0(r)$ and a constant $c(r,s)>0$ such that for all $N\geq N_0$ and all $\epsilon>0$ such that $\epsilon\leq c(r,s) N^{-1}$ the following holds.


There exists an invertible symplectic map $\Phi : B_s(0,\epsilon) \rightarrow B_s(0,2\epsilon)$ with inverse $\Phi^{-1} : B_s(0,\epsilon/2) \rightarrow B_s(0,\epsilon)$
which are close to the identity in the sense that
\begin{equation}
\norm{\Phi^{\pm 1}(u)-u}_{\dot{H}^s} \lesssim_{r,s} N \norm{u}_{\dot{H}^s}^2, 
\label{eq: close to the identity}
\end{equation}
and Lipschitz-continuous:
\begin{equation} \label{eq: lip_const}
\norm{\Phi^{\pm 1}(u)-\Phi^{\pm 1}(\underline{u})}_{\dot{H}^s} \le 2 \norm{u-\underline{u}}_{\dot{H}^s}.
\end{equation}
Moreover on $B_s(0,\epsilon)$ we have, for $j=1,\ldots, r-1$,
\begin{equation}
F^{(j)} \circ \Phi = \sum_{\alpha = 1}^{\lfloor \frac{r}{2} \rfloor } \widehat{F}_{2\alpha}^{(j)} + R^{(j),(\mu_3 \ge N^{1/2r})} + R^{(j),(I \ge N^{1/2r})} + R_{\ge r+1}^{(j)},
\label{hamiltoniana dopo tutte le trasformazioni}
\end{equation}
where
\begin{align}
\widehat{F}_2^{(j)}(u) &= \pi \sum_{k \in \Z^*} k^{2j} |u_k|^2,\tag*{(\ref{hamiltoniana dopo tutte le trasformazioni}a)} \label{parte F2}\\
\widehat{F}_{2\alpha}^{(j)}(u) &= \sum_{\k \in (\Z^*)^{\alpha}} b_{\k}^{(j)} I^{\k} \quad \text{where $I=(I_k)_{k \in \Z^*}=(|u_k|^2)_{k \in \Z^*}$ and $\frac{|b_{\k}^{(j)}|}{\max ( |k_1|,\ldots,|k_{\alpha}| )^{2j-2} } \lesssim_{\alpha} N^{2\alpha-3}$},\tag*{(\ref{hamiltoniana dopo tutte le trasformazioni}b)} \label{parte integrabile 2} \\
R^{(j),(\mu_3 \ge N^{1/2r})}(u) &= \sum_{ \substack{ \k \in \mathcal{D} \\ 3 \le \#\k \le r \\ |k_3| \ge N^{\frac{1}{2r}}} } c_{\k}^{(j)} u^{\k} \quad \text{with $\frac{|c_{\k}^{(j)}|}{\max ( |k_1|,\ldots,|k_{\#\k}| )^{2j-2} } \lesssim_{\#\k} N^{\#\k-3}$}, \tag*{(\ref{hamiltoniana dopo tutte le trasformazioni}c)$_r$} \label{terzo indice grande} \\
R^{(j),(I \ge N^{1/2r})}(u) &= \sum_{m = \lceil N^{\frac{1}{2r}} \rceil}^{+\infty} \, \, \sum_{ \substack{ \k \in \mathcal{D} \\ 1 \le \#\k \le r-2} } c_{m,\k}^{(j)} I_m u^{\k} \quad \text{with $\frac{|c_{m,\k}^{(j)}|}{\max ( |m|,|k_1|,\ldots,|k_{\#\k}| )^{2j-2} } \lesssim_{\#\k} N^{\#\k-1}$}, \tag*{(\ref{hamiltoniana dopo tutte le trasformazioni}d)$_r$}\label{azione grande} \\
R_{\ge r+1}^{(j)}(u) &= \sum_{ \substack{ \k \in \M \\ \#\k \ge r+1 } } c_{\k}^{(j)} u^{\k} \quad \text{where $\frac{|c_{\k}^{(j)}|}{\max ( |k_1|,\ldots,|k_{\#\k}| )^{2j-2} } \le \rho^{\#\k} N^{\#\k-3}$ for some $\rho \lesssim_r 1$}, \tag*{(\ref{hamiltoniana dopo tutte le trasformazioni}e)$_r$} \label{resto grande}
\end{align}
and all the Hamiltonians have real coefficients and satisfy \eqref{def: reality condition}.
\end{theorem}

\begin{rem} \label{rmk: improved normal form}
In \Cref{subsection: cases 3 and 4}, we actually prove the following stronger statement about the remainders \ref{terzo indice grande}--\ref{azione grande}:
 \begin{align}
    R^{(j),(\mu_3 \ge N^{1/2r})}(u) &= \sum_{ \substack{ \k \in \mathcal{D} \\ 5 \le \#\k \le r \\ |k_3| \ge N^{\frac{1}{2r}}} } c_{\k}^{(j)} u^{\k} \quad \text{with $\frac{|c_{\k}^{(j)}|}{\max ( |k_1|,\ldots,|k_{\#\k}| )^{2j-2} } \lesssim_{\#\k} N^{\#\k-3}$}, \tag*{(\ref{hamiltoniana dopo tutte le trasformazioni}f)$_r$}  \\
    R^{(j),(I \ge N^{1/2r})}(u) &= \sum_{m = \lfloor N^{\frac{1}{2r}} \rfloor + 1}^{+\infty} \, \, \sum_{ \substack{ \k \in \mathcal{D} \\ 3 \le \#\k \le r-2} } c_{m,\k}^{(j)} I_m u^{\k} \quad \text{with $\frac{|c_{m,\k}^{(j)}|}{\max ( |m|,|k_1|,\ldots,|k_{\#\k}| )^{2j-2} } \lesssim_{\#\k} N^{\#\k-1}$} \tag*{(\ref{hamiltoniana dopo tutte le trasformazioni}g)$_r$},
\end{align}
namely the fact that there are no terms of degree 3 and 4 in these remainders. 

In order to prove this, we enlarge $\J_{n,l,N}$ in \eqref{def: set JnlN} as follows:
\begin{equation}\label{def: set tildeJnlN}
\widetilde{\mathcal{J}}_{n,l,N} = 
\begin{cases}
\Big \{ \k\in \M_n \setminus \mathcal{R}_n^l \, \big | \, \dfrac{\norm{\k}_{\ell^{\infty}}^{2l-1}}{|\Omega_l(\k)|}  \le N \Big \}, & \text{if $n=3,4,$} \\ 
\vspace{-0.1cm} \\
\Big \{ \k \in \M_n \setminus \mathcal{R}_n^l \, \big | \, \dfrac{\norm{\k}_{\ell^{\infty}}^{2l-1}}{|\Omega_l(\k)|} \le N \Big \} \setminus \left (  \M_n^{(\mu_3 \ge N^{\frac{1}{2n}})} \cup \M_n^{(I \ge N^{\frac{1}{2n}})} \right ), & \text{if $n \ge 5$.}
\end{cases}
\end{equation}
With this definition, all results in \Cref{section: rigorous theory} continue to hold with $\widetilde{\mathcal{J}}$ instead of $\J$. The only proof that must be modified is that of \Cref{prop: magic formula} below, see \Cref{blabla} for additional details.

Then \Cref{lemma: orders 3 and 4}-(ii) and \Cref{prop: magic formula} show that the transformations $\Phi_3=\Phi_{G_3}$ and $\Phi_4=\Phi_{G_4}$, with $G_3$ and $G_4$ defined in \eqref{G3} and \eqref{G4}, remove \emph{all} terms of degree $3$ and $4$ from the KdV hierarchy. In particular, no such terms are present in the remainders \ref{terzo indice grande}--\ref{azione grande}.
\end{rem}

\subsection{Normal form for the KdV hierarchy} \label{subsection: proof of thm}

In this subsection we prove \Cref{prop: justification normal form}. We begin with some preliminary results. 

\paragraph{Convergence of the Taylor series \eqref{eq: taylor formula for Hnew integral remainder}:}
We now state the following proposition which justifies the convergence of the Taylor series \eqref{eq: taylor formula for Hnew integral remainder} under certain conditions on $\H$ and $G$.
The proof, which is an adaptation of \cite[Proposition 3.10]{bernier} can be found in \Cref{app: normal form}.
\begin{prop} \label{conv taylor series}
Let $j \in \N$, $s \ge j$, $N \ge 1$ and let $F^{(j)}$ be a $j$-formal polynomial
\begin{equation} \label{formula Fj prop}
F^{(j)} = F_2^{(j)} + \sum_{\beta \ge 3} F_{\beta}^{(j)} = \pi \sum_{k \in \Z^*} k^{2j} |u_k|^2 + \sum_{\beta \ge 3} \sum_{\k \in \M_{\beta}} c_{\k} u^{\k}, \qquad \norm{(c_{\k})_{\#\k = \beta}}_{Y^{j}_{\beta}} \le \rho^{\beta} N^{\beta-3},
\end{equation}
for some $\rho > 0$.
Consider 
\[
G_n^{(l)} = \sum_{\k \in \J_{n,l,N}} \frac{g_{\k}}{i \Omega_l(\k)} u^{\k}
\]
defined in \eqref{eq: other lie transforms} with $l \in \N$, $n \ge 3$,  and assume that the following conditions hold: 
\begin{align} 
\norm{\bm{g}}_{Y_n^{l}} & \le \rho^n N^{n-3}\label{eq: stima induttiva coeff g} \\ 
\{ F_2^{(j)} , G_n^{(l)} \} & = - \sum_{\k \in \J_{n,l,N}} c_{\k} u^{\k}. \label{hom_eq prop}
\end{align}

Then there exists a constant $c=c(n,s,\rho,j)>0$ such that for all $0<\epsilon\leq c\, N^{-1}$ the time-1 flow of $G_n^{(l)}$ is well-defined on $B_s(0,\epsilon)$ and invertible, in the sense that


\[
\Phi_{G_n^{(l)}} : B_s(0,\epsilon) \rightarrow B_s(0,2\epsilon), \qquad \Phi_{G_n^{(l)}}^{-1} = \Phi_{-G_n^{(l)}} : B_s(0,\epsilon/2) \rightarrow B_s(0,\epsilon), \qquad \Phi_{G_n^{(l)}} \circ \Phi_{G_n^{(l)}}^{-1} = \mathrm{Id}_{B_s(0,\epsilon/2)},
\]
and they both are close to the identity, i.e.
\begin{equation} \label{eq: close to id prop}
\norm{\Phi_{G_n^{(l)}}^{\pm 1}(u)-u}_{\dot{H}^s} \lesssim_{n,s} N \norm{u}_{\dot{H}^s}^2.
\end{equation}
Moreover on $B_s(0,\epsilon)$ we have that $F^{(j)} \circ \Phi_{G_n^{(l)}}$ is again a $j$-formal polynomial,
\begin{equation}\label{formula Fj trans prop}
F^{(j)} \circ \Phi_{G_n^{(l)}} = F_2^{(j)} + \sum_{\beta \ge 3} \sum_{\k \in \M_\beta} q_{\k} u^{\k}.
\end{equation}
Finally there exists $\widetilde{\rho}=\widetilde{\rho}(n,j,\rho)>0$ such that 
\begin{equation}\label{eq: stima coeff q}
    \norm{(q_{\k})_{\#\k = \beta}}_{Y_{\beta}^j} \le \widetilde{\rho}^{\beta} N^{\beta-3}
\end{equation}
and
\begin{equation} \label{eq: formula qn}
q_{\k} =  c_{\k} \quad \mbox{if} \quad \#\k<n,\qquad 
q_{\k} =  c_{\k} \delta(\k \in \M_n \setminus \mathcal{J}_{n,l,N}) \quad \text{for} \quad \#\k=n.
\end{equation}
\end{prop}

Proving that \eqref{eq: stima induttiva coeff g}  and especially \eqref{hom_eq prop} are satisfied for our choice of auxiliary Hamiltonians $G_n^{(l)}$ is one of the key results of this section.

\begin{rem}
As explained in \cite[~Remark 3.11]{bernier}, the exponent $\#\k-3$ is quite natural because of the loss of one derivative in the Poisson bracket. Indeed, by \Cref{lemma: stabilità parentesi poisson} we have
\[
\left \{ \sum_{\k \in \mathcal{M}_r} c_{\k} u^{\k} , \sum_{\k \in \mathcal{J}_{n,l,N}} \frac{\widetilde{c}_{\k}}{i \Omega_l(\k)} \right \} = \sum_{\k \in \mathcal{M}_{r+n-2}} d_{\k} u^{\k}
\]
with
\[
\norm{\bm{d}}_{Y_{r+n-2}^j} \lesssim N n^{2j-1} r \norm{\bm{c}}_{Y_{\#\k = r}^j} \norm{\widetilde{\bm{c}}}_{Y_{\#\k = n}^{l}} \stackrel{\eqref{formula Fj prop},\eqref{eq: stima induttiva coeff g}}{\lesssim_{n,j,r}} N^{1+n-3+r-3} \lesssim_{n,j,r} N^{(r+n-2)-3}. 
\]
\end{rem}
\begin{rem} \label{rem: easy computations}
Notice that whenever we use \Cref{lemma: transformations preserve Hs} in the proof of \Cref{conv taylor series}, we have to impose $\epsilon \lesssim_{n,s} [N \norm{\bm{g}}_{Y_n^l}]^{-\frac{1}{n-2}}$.
In principle one could fear that, since during the normal form procedure we compute a lot of Poisson brackets, then $\norm{\bm{g}}_{Y_n^l}$ may become huge in $N$. 
However, since we are able to prove the stability property \eqref{eq: stima induttiva coeff g}, then at each step we will always have to impose $\epsilon \lesssim_{n,s} N^{-1}$, a finite number of times. Therefore it is sufficient to choose the smallest constant among a finite number, see the proof of \Cref{prop: justification normal form} below.
\end{rem}

\paragraph{Simultaneous transformation for the hierarchy}
The next result concerns the Poisson bracket between certain classes of formal polynomials like \ref{parte integrabile 2}--\ref{azione grande}. This is useful because we are going to use the fact that the KdV hierarchy commutes \eqref{eq: first integral in involution}, even after some transformations \eqref{eq: hamiltonian flows preserve poisson structure}, but inductively the hierarchy has the form \eqref{hamiltoniana dopo tutte le trasformazioni}.

\begin{lemma} \label{lemma: parentesi tra resti è resto}
Consider $n_1,n_2 \ge 5$, $n_3,n_4 \ge 3$, $n_5 \ge 1$, $M \ge 1$ and let
\[
G_1(u) = \sum_{ \substack{ \k \in \mathcal{D}_{n_1} \\ |k_3| \ge M } } c_{\k}^{(1)} u^{\k}, \quad G_2(u) = \sum_{ \substack{ \k \in \mathcal{D}_{n_2} \\ |k_3| \ge M } } c_{\k}^{(2)} u^{\k}, \quad G_3(u) = \sum_{\tr = M}^{+\infty} \, \, \sum_{\k \in \mathcal{D}_{n_3}} c_{(\tr,-\tr,\k)}^{(3)} I_{\tr} u^{\k},
\]
\[
G_4(u) = \sum_{\tr = M}^{+\infty} \, \, \sum_{\k \in \mathcal{D}_{n_4}} c_{(\tr,-\tr,\k)}^{(4)} I_{\tr} u^{\k}, \quad G_5(u) = \sum_{\k \in (\Z^*)^{n_5}} c_{\k}^{(5)} I^{\k},
\]
with $\bm{c}^{(i)} \in Y_{n_i}^{j_i}$ for $i=1,2$, $\bm{c}^{(i)} \in Y_{n_i+2}^{j_i}$ for $i=3,4$ and 
\[
\sup_{\k \in (\Z^*)^{n_5}} \frac{|c_{\k}^{(5)}|}{\norm{\k}_{\ell^{\infty}}^{2j_5-2} } < +\infty,
\]
where $j_1,\ldots,j_5 \ge 1.$
Then
\[
\{ G_1 , G_2 \} = \sum_{ \substack{ \k \in \mathcal{D}_{n_1+n_2-2} \\ |k_3| \ge M } } c_{\k}^{(1,2)} u^{\k}, \quad \{ G_1 , G_3 \} = \sum_{ \substack{ \k \in \mathcal{D}_{n_1+n_3} \\ |k_3| \ge M } } c_{\k}^{(1,3)} u^{\k}, \quad \{ G_1 , G_5 \} = \sum_{ \substack{ \k \in \mathcal{D}_{n_1+2n_5-2} \\ |k_3| \ge M } } c_{\k}^{(1,5)} u^{\k},
\]
\[
\{ G_3 , G_4 \} = \sum_{\tr = M}^{+\infty} \, \, \sum_{\k \in \mathcal{D}_{n_3+n_4}} c_{(\tr,-\tr,\k)}^{(3,4)} I_{\tr} u^{\k}, \quad \{ G_3 , G_5 \} = \sum_{\tr = M}^{+\infty} \, \, \sum_{\k \in \mathcal{D}_{n_3+2n_5}} c_{(\tr,-\tr,\k)}^{(3,5)} I_{\tr} u^{\k},
\]
\end{lemma}
\begin{proof}
Poisson brackets between various $G_i$ are formal polynomials in view of 
\Cref{lemma: facts about poisson bracket}.

Let us prove the formula for $\{G_1,G_2\}$. Note that 
\[
\{G_1,G_2\}(u)= \sum_{ \substack{ \bar{\k} \in \mathcal{D}_{n_1+n_2-2}}} c_{\bar{\k} }^{(1,2)} u^{\bar{\k}} 
\]
where $c_{\bar{\k} }^{(1,2)}$ explicitly depends on $\bm{c}^{(1)}$ and $\bm{c}^{(2)}$, cf. \eqref{dk parentesi FG}. 
In particular, if $|\bar{k}_3|<M$, then $c_{\bar{\k} }^{(1,2)} =0$ by \eqref{dk parentesi FG} and the fact that $\bm{c}^{(j)}$, $j=1,2$, are supported  on $|k_3|\geq M$. The same argument works for the Poisson bracket $\{G_1,G_3\}$.

For $\{ G_3,G_4 \}$ it is sufficient to notice that, by the Leibniz rule and \Cref{lemma: facts about poisson bracket}-(iii),
\[
\{ I_{r_1} u^{\k_1} , I_{r_2} u^{\k_2} \} = I_{r_1} I_{r_2} \{ u^{\k_1} , u^{\k_2} \} + I_{r_1} u^{\k_2} \{ u^{\k_1} , I_{r_2} \} + u^{\k_1} I_{r_2} \{ I_{r_1} , u^{\k_2} \},
\]
which ensures that each monomial has at least one action with index $ \ge M.$

In order to justify the formula for $\{G_3,G_5\}$, it suffices to note that
\[
\{ I^{\k} , I_r u^{\k'} \} = I_r \{ I^{\k} , u^{\k'} \}.
\]

Finally, to justify the formula for $\{G_1,G_5\}$, note that, by the Leibniz rule, 
\[
\{ u^{\k}, I^{\k'}\} = \sum_{\alpha=1}^{\#\k} \sum_{\beta=1}^{\#\k'} \{ u_{k_{\alpha}}, I_{k'_{\beta}}\} \, \frac{u^{\k}}{u_{k_{\alpha}}} \, \frac{I^{\k'}}{I_{k'_{\beta}}}.
\]
Moreover, 
\[
\{ u_{k_{\alpha}} , I_{k'_{\beta}} \} = \frac{ik_{\alpha}}{2\pi} \frac{\pa |u_{k'_{\beta}}|^2}{\pa \overline{u_{k_{\alpha}}}} = \frac{i{k_\alpha}}{2\pi} u_{k_{\alpha}}( \delta(k'_{\beta}=k_{\alpha}) + \delta(k'_{\beta}=-k_{\alpha}) ).
\]
As a result, the number of indices which are greater or equal than $M$ (in absolute value) does not decrease.
\end{proof}

Our next result will be the key tool in order to inductively prove that the symplectic transformation $\Phi$ simultaneously puts the KdV hierarchy in normal form up to certain remainders. This result is ultimately based on \cite[Theorem~G.2]{kappeler}, although that theorem cannot be applied directly since our transformation does not satisfy its hypotheses. Indeed, our mapping $\Phi$ does not put the Hamiltonian in full Birkhoff normal form due to the truncation \eqref{def: set JnlN}. See also \cite[Lemma~4.4]{feola} and \cite[Proposition~3.6]{feola2} for similar results in the context of the Degasperis-Procesi equation. 

\begin{prop} \label{prop: magic formula}
Let $r \ge n \ge 3$, $s \ge 2 r$ and $j = 1,\ldots,r-1$.
Let
\begin{equation}\label{eq: pre magic formula}
F^{(j)} = \sum_{\alpha = 1}^{\lfloor \frac{n-1}{2} \rfloor } \widehat{F}_{2\alpha}^{(j)} + R^{(j),(\mu_3 \ge N^{\frac{1}{2(n-1)}})} + R^{(j),(I \ge N^{\frac{1}{2(n-1)}})} + R_{\ge n}^{(j)},
\end{equation}
where each summand on the right-hand side is of the form \ref{parte F2}--\ref{parte integrabile 2} and (\ref{hamiltoniana dopo tutte le trasformazioni}c)$_{n-1}$--(\ref{hamiltoniana dopo tutte le trasformazioni}e)$_{n-1}$.
Assume that
\begin{equation} \label{eq: kdv hierarchy commutes induction}
\{ F^{(j)}  , F^{(l)}  \} = 0 \qquad \mbox{for all} \quad j,l=1,\ldots, r-1.
\end{equation}

\smallskip

Then, recalling that $F_n^{(j)}(u) = \sum_{\k\in\mathcal{M}_n} c_{\k}^{(j)} u^{\k}$, we have that
\begin{equation}
\{ F_2^{(j)} , G_{n}^{(l)} \} = - \sum_{\k \in \mathcal{J}_{n,l,N}} c_{\k}^{(j)} u^{\k} \qquad \mbox{for any}\ j = 1,\ldots,r-1,\ \ l = 1, \ldots , n-1 ,
\label{claim}
\end{equation}
with $G_{n}^{(l)}$ defined in \eqref{eq: other lie transforms}.
\end{prop}
\begin{proof}
By \Cref{lemma: facts about poisson bracket}-(ii) we have 

\begin{equation}\label{pre pre formula magica}
\{ F_2^{(j)} , G_{n}^{(l)} \}  = - \sum_{\k \in \mathcal{J}_{n,l,N}} \frac{c_{\k}^{(l)}}{\Omega_l(\k)} \Omega_j(\k) u^{\k} = - \sum_{[\k] \in \mathcal{J}_{n,l,N}/\mbox{Sym}(n)} \frac{\Omega_j(\k)}{\Omega_l(\k)} \sum_{\j \sim \k} c_{\j}^{(l)} u^{\j}.
\end{equation}
Since $s \ge j+l+\frac{1}{2},$ we can apply \Cref{rmk: homogeneity argument} to \eqref{eq: kdv hierarchy commutes induction}, and isolating the homogeneous term of degree $n$ we obtain
\begin{equation}
\{ F_2^{(j)}, F_{n}^{(l)} \} + \{ F_3^{(j)}, F_{n-1}^{(l)} \} + \ldots + \{ F_{n-1}^{(j)}, F_3^{(l)} \} + \{ F_{n}^{(j)}, F_2^{(l)} \} = 0.
\label{poisson commute order r}
\end{equation}
Notice that
\[
\{ F_2^{(j)}, F_{n}^{(l)} \} + \{ F_{n}^{(j)}, F_2^{(l)} \} = i \sum_{\k \in \M_{n}} \left ( \Omega_l(\k) c_{\k}^{(j)} - \Omega_j(\k) c_{\k}^{(l)} \right ) u^{\k}.
\]
By \Cref{lemma: facts about poisson bracket}-(iii), \Cref{lemma: parentesi tra resti è resto} and \eqref{eq: pre magic formula}, \eqref{poisson commute order r} becomes
\[
\sum_{\k \in \M_{n} \setminus \left ( \M_{n}^{(\mu_3 \ge N^{\frac{1}{2n}})} \cup \M_{n}^{(I \ge N^{\frac{1}{2n}})} \right ) } i \left ( \Omega_l(\k) c_{\k}^{(j)} - \Omega_j(\k) c_{\k}^{(l)} \right ) u^{\k} + \sum_{\k \in \M_{n}^{(\mu_3 \ge N^{\frac{1}{2n}})} \cup \M_{n}^{(I \ge N^{\frac{1}{2n}})}} e_{\k}^{(j,l)} u^{\k} = 0.
\]
Using \Cref{rmk: homogeneity argument}, we obtain the formula
\begin{equation}
\sum_{\k \sim \j} \left [ \Omega_j(\k) c_{\k}^{(l)} - \Omega_l(\k) c_{\k}^{(j)} \right ] = 0 \qquad \forall \j \in \M_{n} \setminus \left ( \M_{n}^{(\mu_3 \ge N^{\frac{1}{2n}})} \cup \M_{n}^{(I \ge N^{\frac{1}{2n}})} \right ),
\label{magic formula}
\end{equation}
and in particular for all $\j \in \mathcal{J}_{n,l,N}.$
Therefore
\[
\sum_{\k \sim \j} c_{\k}^{(j)} = \frac{\Omega_j(\j)}{\Omega_l(\j)} \sum_{\k \sim \j} c_{\k}^{(l)} \qquad \qquad \forall \j \in \mathcal{J}_{n,l,N},
\]
which yields \eqref{claim} in view of \eqref{pre pre formula magica}.
\end{proof}

\begin{rem}
We have shown that if \eqref{hamiltoniana dopo tutte le trasformazioni} holds, then the Hamiltonian $G_n^{(l)}$ in \eqref{eq: other lie transforms} simultaneously satisfies $n-1$ homological equations, cf. \eqref{claim}. We note that this result requires a careful definition of $\mathcal{J}_{n,l,N}$, \eqref{def: set JnlN}, in order to avoid terms coming from interactions between the remainders, cf. \eqref{magic formula}. 
This is one key technical difference with respect to the approach in \cite{bernier}.
\end{rem}

\begin{rem} \label{blabla}
If instead of $\mathcal{J}_{n,l,N}$ we use $\widetilde{\mathcal{J}}_{n,l,N}$ in \Cref{rmk: improved normal form}, then \Cref{prop: magic formula} still holds.
Indeed \eqref{poisson commute order r} for $n=3$ immediately gives \eqref{magic formula} for all $\j \in \widetilde{\J}_{3,l,N}$, since only $\{F_2^{(j)},F_n^{(l)}\}$ and $\{F_n^{(j)},F_2^{(l)}\}$ appear.
The same is true for $n=4$, since \eqref{poisson commute order r} becomes 
\[
\{ F_2^{(j)}, F_{4}^{(l)} \} + \{ F_3^{(j)}, F_{3}^{(l)} \} + \{ F_{4}^{(j)}, F_2^{(l)} \} = 0,
\]
and \Cref{lemma: orders 3 and 4}-(ii) below proves that there are no terms of degree 3 in the KdV hierarchy (at least for $j,l \le r-1)$ after the tranformation $\Phi_3$.
\end{rem}

\paragraph{Proof of \Cref{prop: justification normal form}.}
By \Cref{lem:first_integrals}, the theorem is true for $r=2$, with $\Phi = \mbox{Id}$.
Next fix $r\geq 3$. We fix $r-1$ first integrals $F^{(j)}$ in \eqref{eq: first integrals in fourier}, for $j \in \{ 1, \ldots , r-1 \}$.
The proof is by induction on $n$, the degree up to which we put these first integrals are in normal form (up to remainders). In particular, we will induct from $n=2$ to $n=r$.
At the end of the $(n-1)$-th step, we will have carried out transformations $\Phi_1,\ldots,\Phi_{n-1}$ in such a way that
\begin{equation}\label{pre induction}
F^{(j)} \circ \Phi_2\circ\cdots\circ \Phi_{n-1}= \sum_{\alpha = 1}^{\lfloor \frac{n-1}{2} \rfloor } \widehat{F}_{2\alpha}^{(j)} + R^{(j),(\mu_3 \ge N^{\frac{1}{2(n-1)}})} + R^{(j),(I \ge N^{\frac{1}{2(n-1)}})} + R_{\ge n}^{(j)} ,\quad j=1,\ldots, r-1,
\end{equation}
satisfying \ref{parte F2}--\ref{parte integrabile 2} and (\ref{hamiltoniana dopo tutte le trasformazioni}c)$_{n-1}$--(\ref{hamiltoniana dopo tutte le trasformazioni}e)$_{n-1}$. 

\medskip 

As part of the $n$-th step we perform an additional transformation $\Phi_n$ and we inductively check the following facts:
\begin{enumerate}
    \item For all $\epsilon \le c(r,n,s) N^{-1}$, the symplectic transformation
    $
    \Phi_n : B_s(0,\epsilon) \rightarrow B_s(0,2\epsilon),
    $
    and its inverse 
    $
    \Phi_n^{-1} : B_s(0,\epsilon/2) \rightarrow B_s(0,\epsilon),
    $
    (cf. \eqref{eq:Phi_comp}), are well-defined and both are close to the identity (cf. \eqref{eq: close to the identity}).
    \item Setting
    \[
    \Phi_{\le n-1} = \Phi_2 \circ \ldots \circ \Phi_{n-1},
    \]
    we note that for $j,j_1,j_2 \in \{ 1,\ldots,r-1 \},$
    \[
    F^{(j)} \circ \Phi_{\le n-1} \circ \Phi_{n}, \qquad \text{and} \quad \{ F^{(j_1)} \circ \Phi_{\le n-1} \circ \Phi_{n} , F^{(j_2)} \circ \Phi_{\le n-1} \circ \Phi_{n} \}
    \]
    are respectively a $j$ and a $j_1+j_2+1$ formal polynomial well-defined on $B_s(0,\epsilon)$ and
    \[
    \{ F^{(j_1)} \circ \Phi_{\le n-1} \circ \Phi_{n} , F^{(j_2)} \circ \Phi_{\le n-1} \circ \Phi_{n} \} = 0.
    \]
    \item The transformed Hamiltonian can be written as
    \begin{equation}\label{post induction}
F^{(j)} \circ \Phi_2\circ\cdots\circ \Phi_{n}= \sum_{\alpha = 1}^{\lfloor \frac{n}{2} \rfloor } \widehat{F}_{2\alpha}^{(j)} + R^{(j),(\mu_3 \ge N^{\frac{1}{2n}})} + R^{(j),(I \ge N^{\frac{1}{2n}})} + R_{\ge n+1}^{(j)} ,\quad j=1,\ldots, r-1,
\end{equation}
and it satisfies \ref{parte F2}--\ref{parte integrabile 2} and (\ref{hamiltoniana dopo tutte le trasformazioni}c)$_n$--(\ref{hamiltoniana dopo tutte le trasformazioni}e)$_n$.
\end{enumerate}

We note that the remainders $R^{(j)}$ in \eqref{post induction} are not equal to those in \eqref{pre induction}, but they satisfy (\ref{hamiltoniana dopo tutte le trasformazioni}c)$_n$--(\ref{hamiltoniana dopo tutte le trasformazioni}e)$_n$ instead of (\ref{hamiltoniana dopo tutte le trasformazioni}c)$_{n-1}$--(\ref{hamiltoniana dopo tutte le trasformazioni}e)$_{n-1}$. This choice of notation, which will be kept throughout the proof, is made to avoid cumbersome indices.

Note that properties 1, 2, and 3 are true for $n=2$ since $\Phi_2 = \mbox{Id}$, which is the induction basis. Next we assume that properties 1, 2, and 3 hold for $n-1$ and for all $j \in \{ 1, \ldots , r-1 \}.$ 

\smallskip

The transformation $\Phi_n$ is itself the composition of $n-1$ symplectic transformations generated by 
\begin{equation} \label{vera Gn}
G_n^{(l)} = \sum_{\k \in \J_{n,l,N}} \frac{c_{\k}^{(l)}}{i \Omega_l(\k)} u^{\k} \quad \text{where} \quad \widetilde{F}_n^{(j)} = \left ( F^{(l)} \circ \Phi_1 \circ \ldots \circ \Phi_{n-1} \circ \Phi_{G_n^{(1)}} \circ \ldots \circ \Phi_{G_n^{(l-1)}} \right )_n = \sum_{\k \in \M_n} c_{\k}^{(l)} u^{\k}
\end{equation}
defined in \eqref{eq: other lie transforms} for $l = 1,\ldots,n-1$.
In order to prove points 1--3 above we proceed by induction on $l=1,\ldots,n-1$.
We inductively check the following facts:
\begin{enumerate}
    \item[(i)] For all $\epsilon \le c(l,n,r,s)N^{-1}$, the symplectic transformation 
    $\Phi_{G_n^{(l)}} : B_s(0,\epsilon) \rightarrow B_s(0,2\epsilon),$
    and its inverse 
    $
    \Phi_{-G_n^{(l)}} : B_s(0,\epsilon/2) \rightarrow B_s(0,\epsilon),
    $
    (cf. \eqref{eq:Phin}), are well-defined and both are close to the identity (cf. \eqref{eq: close to the identity}).
    \item[(ii)] Setting
    \[
    \Phi_{\mathrm{old}} = \Phi_2 \circ \ldots \circ \Phi_{n-1} \circ \Phi_{G_n^{(1)}} \circ \ldots \circ \Phi_{G_n^{(l-1)}},
    \]
    we note that for $j,j_1,j_2 \in \{ 1,\ldots,r-1 \},$
    \[
    F^{(j)} \circ \Phi_{\mathrm{old}} \circ \Phi_{G_n^{(l)}}, \qquad \text{and} \quad \{ F^{(j_1)} \circ \Phi_{\mathrm{old}} \circ \Phi_{G_n^{(l)}} , F^{(j_2)} \circ \Phi_{\mathrm{old}} \circ \Phi_{G_n^{(l)}} \}
    \]
    are respectively a $j$ and a $j_1+j_2+1$ formal polynomial well-defined on $B_s(0,\epsilon)$ and
    \[
    \{ F^{(j_1)} \circ \Phi_{\mathrm{old}} \circ \Phi_{G_n^{(l)}} , F^{(j_2)} \circ \Phi_{\mathrm{old}} \circ \Phi_{G_n^{(l)}} \} = 0.
    \]
    \item[(iii)] The transformed Hamiltonian can be written as
\begin{equation}\label{post induction on l}
F^{(j)} \circ \Phi_{\mathrm{old}} \circ \Phi_{G_n^{(l)}} = \sum_{\alpha = 1}^{\lfloor \frac{n-1}{2} \rfloor } \widehat{F}_{2\alpha}^{(j)} + R^{(j),(\mu_3 \ge N^{\frac{1}{2(n-1)}})} + R^{(j),(I \ge N^{\frac{1}{2(n-1)}})} + R_{\ge n}^{(j)} ,\quad j=1,\ldots, r-1,
\end{equation}
and it satisfies \ref{parte F2}--\ref{parte integrabile 2} and (\ref{hamiltoniana dopo tutte le trasformazioni}c)$_{n-1}$--(\ref{hamiltoniana dopo tutte le trasformazioni}e)$_{n-1}$.
Moreover the homogeneous part of degree $n$ in \eqref{post induction on l} changes in the following way:
\[
(F^{(j)} \circ \Phi_{\mathrm{old}})_n = \sum_{\k \in \M_n} c_{\k}^{(j)} u^{\k} \implies \left ( F^{(j)} \circ \Phi_{\mathrm{old}} \circ \Phi_{G_n^{(l)}} \right )_n = \sum_{\k \in \left ( \M_n \smallsetminus \mathcal{J}_{n,l,N} \right )} c_{\k}^{(j)} u^{\k}.
\]
\end{enumerate}

It is clear that properties (i)--(iii) are satisfied before carrying out any transformation at order $n$, i.e. by taking $\Phi_{G_n^{(0)}}=\mbox{Id}$ after defining $G_n^{(0)}=0$.
Indeed, \eqref{pre induction} implies \eqref{post induction on l} for $l=0$ by the induction hypothesis. 

\medskip

Let us prove points (i)--(iii) by induction on $l$. 
They easily follow by \Cref{conv taylor series} and \Cref{prop: magic formula}: let us check that the hypotheses are satisfied.
Notice that \eqref{formula Fj prop}-\eqref{eq: stima induttiva coeff g} follow by the inductive hypotheses, in particular by \eqref{post induction on l}.
On the other hand, \eqref{hom_eq prop} follows by the definition of $G_n^{(l)}$ and by \Cref{prop: magic formula}, whose hypotheses are again satisfied by the inductive hypotheses: in particular \eqref{eq: pre magic formula} follows by \eqref{post induction on l} and \eqref{eq: kdv hierarchy commutes induction} follows by point (ii) at the previous inductive step.
Now we can say that

\smallskip

\noindent $\bullet$ Point (i) follows by \Cref{conv taylor series}, in particular \eqref{eq: close to id prop}.  

\noindent $\bullet$ Point (ii) follows by \Cref{conv taylor series}.
We take $c(l,n,r,s) \leq \frac{c(l-1,n,r,s)}{2}$, so that the composition $\Phi_{\mathrm{old}} \circ \Phi_{G_n^{(l)}}$ is well-defined and so that we can apply \Cref{conv taylor series}.
The fact that the transformed integrals commute follows by \eqref{eq: first integral in involution} and the fact that each transformation is symplectic.


\noindent $\bullet$ Point (iii) easily follows from \Cref{conv taylor series}, in particular \eqref{eq: formula qn}.

\medskip

Finally we complete the induction on $n$ proving points 1--3.

\smallskip

\noindent $\bullet$ Point 1: we take $c(n,r,s)=\min_{l=1,...,n-1} c(l,n,r,s)$ so that the composition $\Phi_{G_n^{(1)}} \circ \ldots \circ \Phi_{G_n^{(n-1)}}$ is well-defined.
A priori the codomain of $\Phi_n$ is not $B_s(0,2\epsilon)$.
However, notice that by the triangle inequality
\begin{equation} \label{real codomain}
\norm{\Phi_n(u)-u}_{\dot{H}^s} \le \sum_{i=1}^{n-1} \norm{\Phi_{G_n^{(i)}} \left ( \Phi_{G_n^{(i+1)}} \circ \ldots \circ \Phi_{G_n^{(n-1)}}(u) \right )  - \Phi_{G_n^{(i+1)}} \circ \ldots \circ \Phi_{G_n^{(n-1)}}(u)}_{\dot{H}^s} \lesssim_{n,s} N \norm{u}_{\dot{H}^s}^2
\end{equation}
which proves that $\Phi_n$ is close to the identity in the sense of \eqref{eq: close to the identity} and also that the codomain is correct, provided that one chooses $c(n,r,s)$ sufficiently small.
The same is true for the inverse $\Phi_n^{-1}$.

\noindent $\bullet$ Point 2: the fact that the composition is well-defined follows again by choosing $c(n,r,s)\leq c(n-1,r,s)/2$, while the other facts follow from \Cref{conv taylor series} and \eqref{eq: first integral in involution} as in point (ii).

\noindent $\bullet$ Point 3: follows by \eqref{eq: formula qn} and \Cref{teorema multiindici rimasti}.

\medskip

This concludes the induction in $n$. Finally, we fix $c(r,s)=\min_{n=2,\ldots,r} c(n,r,s)$ and by an argument similar to \eqref{real codomain}, up to shrinking $c(r,s)$ again we can conclude that $\Phi : B_s(0,\epsilon) \rightarrow B_s(0,2\epsilon)$ for all $\epsilon \le c(r,s) N^{-1}$.
The same is true for the inverse $\Phi^{-1}$. 

\medskip

{\sc Proof of \eqref{eq: lip_const}.} By the contruction of $\Phi^{\pm 1}$ as composition of the maps $\Phi_{\pm G_n^{(l)}}$, it follows that
\[
\begin{split}
\norm{\Phi^{\pm 1}(u)-\Phi^{\pm 1}(\underline{u})}_{\dot{H}^s} & \stackrel{\eqref{eq: lipschitz continuity of phi}}{\le} \exp \left ( C(r,s) \sum_{n=3}^r \norm{\bm{g}}_{Y_n^l} N (\norm{u}_{\dot{H}^s}+\norm{\underline{u}}_{\dot{H}^s})^{n-2} \right ) \norm{u-\underline{u}}_{\dot{H}^s} \\
& \stackrel{(\ref{hamiltoniana dopo tutte le trasformazioni}e)_n,\eqref{eq: other lie transforms}}{\le} \exp \left ( C(r,s) \sum_{n=3}^r N^{n-2} (\norm{u}_{\dot{H}^s}+\norm{\underline{u}}_{\dot{H}^s})^{n-2} \right ) \norm{u-\underline{u}}_{\dot{H}^s} \le 2 \norm{u-\underline{u}}_{\dot{H}^s}
\end{split}
\]
using that $\epsilon N\leq c(r,s)$ with $c(r,s)$ is sufficiently small. \hfill \qedsymbol

\begin{rem}
We highlight that, in the second inductive procedure in the proof of \Cref{prop: justification normal form} (points (i)--(iii)), the order in which we apply the transformations $\Phi_{G_n^{(l)}}$, $l=1,\ldots,n-1$, can be chosen arbitrarily.

\smallskip

Moreover, by slightly modifying \Cref{cor: system}, one can see that at the $n$-th step it is not necessary to choose $l \in \{ 1,\ldots,n-1 \}$, but it is enough to choose $l \in \{ p, p + a, \ldots, p + (n-2) a \}$ for some $p,a \in \N.$
\end{rem}

\begin{rem} We highlight that, in \Cref{prop: justification normal form}, we simultaneously put $r-1$ first integrals in normal form to order $r$ (up to remainder terms). Our proof can be adapted to simultaneously put $m$ first integrals in normal form to order $r$, provided $m\geq r-1$ and $s\geq m+r$. We chose $m=r-1$ in order to minimize the regularity $s$.
\end{rem}

\begin{rem}\label{rk: low regularity}
    Let us highlight that each of the mappings composing $\Phi$ can be defined on $B_{\sigma}(0,\epsilon)$  for any $1\leq \sigma\leq s$, cf. \Cref{lemma: transformations preserve Hs} and \Cref{prop: magic formula}. In particular, we can choose $\sigma$ sufficiently large so that the composition $F^{(j)}\circ\Phi$ and the Poisson brackets $\{ F^{(j)}\circ\Phi, F^{(l)}\circ\Phi\}$ are well-defined, e.g. \eqref{pre induction}. One can thus obtain the algebraic identity \eqref{magic formula}
    by defining $\Phi$ on a sufficiently regular domain. This allows us to derive \eqref{claim}
    in the minimal regularity necessary to define $F^{(j)}$. In particular, this allows us to enunciate \Cref{thm: normal form for ham} for $s\geq 1$, as explained in \Cref{rk: megaremark sotto 210}.
\end{rem}

\subsection{The maps \texorpdfstring{$\Phi_3$ and $\Phi_4$}.} \label{subsection: cases 3 and 4}

In this section we slightly modify the maps $\Phi_3$ and $\Phi_4$ following \Cref{rmk: improved normal form}. 
In particular, we define $\Phi_n = \Phi_{G_n}$ with $G_n$, $n=3,4$, defined in \eqref{G3} and \eqref{G4}, as opposed to the previously defined $G_n^{(l)}$ for $l \leq n-1$. 
In particular, this will allow us to explicitly compute $\widehat{\H}_4$ in \eqref{parte quartica - thm ham} and to prove that there are no terms of degree 3 and 4 in the remainders of the transformed Hamiltonian \eqref{hamiltoniana dopo tutte le trasformazioni - thm ham}. 

First, we prove that third and fourth order resonances are trivial:



\begin{lemma} \label{lemma: orders 3 and 4}
Let $\mathcal{R}_n^1$ be the set of $n$-resonant tuples defined in \eqref{def: resonant indices zero momentum}. 
Then

\smallskip

\noindent (i) $\mathcal{R}_3^1 = \emptyset$ and  
\[
\mathcal{R}_4^1 = \left \{ (k_1,k_2,k_3,k_4)\in\M_4\mid (k_1+k_2)(k_1+k_3)(k_1+k_4) = 0\right \}.
\]
Note that indices in $\mathcal{R}_4^1$ correspond to integrable terms (cf. \Cref{def: functions in normal form}).

\smallskip

\noindent (ii) For $N\geq 1$, $\widetilde{\mathcal{J}}_{3,1,N} = \M_3$ and $\widetilde{\mathcal{J}}_{4,1,N} = \mathcal{M}_4 \setminus \mathcal{R}_4^1$ (cf. \Cref{rmk: improved normal form}).
\end{lemma}
\begin{proof}
\noindent (i) Note that for $\k \in \M_3$,
\begin{equation}\label{R3}
k_1+k_2+k_3=0 \quad \implies \quad \Omega_1(\k)=k_1^3+k_2^3+k_3^3 = 3k_1k_2k_3,
\end{equation}
and for $\k \in \M_4$,
\begin{equation} \label{R4}
k_1+k_2+k_3+k_4=0 \quad \implies \Omega_1(\k)=k_1^3+k_2^3+k_3^3+k_4^3 = -3(k_1+k_2)(k_1+k_3)(k_2+k_3).
\end{equation}

\noindent (ii) The fact that $\widetilde{\J}_{3,1,N} = \M_3$ follows immediately from the definition of $\widetilde{\J}_{3,1,N}$ (cf. \eqref{def: set tildeJnlN}) and \eqref{R3}.

To prove $\widetilde{\J}_{4,1,N} = \mathcal{M}_4 \setminus \mathcal{R}_4^1$, we distinguish two cases. Since the four indices cannot all have the same sign:
\begin{itemize}
    \item If three of them have the same sign, then the other one, say $k_1$, is maximal, i.e. $\max{(|k_1|,|k_2|,|k_3|,|k_4|)}=|k_1|$. Since $k_1 = -k_2-k_3-k_4$, then for one index, say $k_2$, we have that $|k_2| \le \frac{|k_1|}{3}$. Therefore $|k_1+k_2| \ge \frac{2}{3}|k_1|$ and the thesis follows from \eqref{R4}.
 \item If two indices have the same sign, by symmetry we can assume that $\max{(|k_1|,|k_2|,|k_3|,|k_4|)}=|k_1|$ and that $k_1$ and $k_2$ have the same sign. Therefore $\frac{|k_1|}{|k_1+k_2|} \le 1$ and \eqref{R4}  concludes the proof. 
 \end{itemize}
\end{proof}

We define the transformations $\Phi_3$ and $\Phi_4$ in \eqref{eq:Phi_comp} as time-1 flows of Hamiltonians $G_3$ and $G_4$ which we construct next. 
Using \eqref{R3}, we define
\begin{equation}\label{G3}
G_3 = i \frac{\pi}{3} \sum_{\k \in \M_3} \frac{u^{\k}}{k_1^3 +k_2^3 +k_3^3} = i \frac{\pi}{9} \sum_{\k \in \M_3} \frac{u^{\k}}{k_1k_2k_3}
\end{equation}
which exactly solves the homological equation:
\[
\widehat{\H}_3 = (\H \circ\Phi_3)_3 \stackrel{\eqref{eq: taylor formula for Hnew integral remainder}}{=} \H_3 + \{ \H_2 , G_3 \} = 0
\]
yielding
\[
\H\circ\Phi_3 = (\H\circ\Phi_3)_2 +  \widehat{\H}_3+ \widetilde{\H}_4 + \ldots
\]
where
\[
\widetilde{\H}_4 = (\H\circ\Phi_3)_4 \stackrel{\eqref{eq: taylor formula for Hnew integral remainder}}{=} \frac{1}{2} \{ \{ \H_2,G_3 \},G_3 \} + \{ \H_3,G_3 \} = \frac{1}{2} \{ \H_3,G_3 \} = -\frac{\pi}{12} \sum_{\substack{ \k \in \M_4 \\ k_2+k_3 \neq 0}}  \frac{u^{\k}}{k_2 k_3}.
\]
Then one has to solve the homological equation
\[
\{ \H_2 , G_4 \} + \widetilde{\H}_4 = \widehat{\H}_4
\]
and find 
\begin{equation}\label{G4}
G_4 = -\frac{\pi}{12} \sum_{\k \in \M_4 \smallsetminus \mathcal{R}_4^1}  \frac{u^{\k}}{i k_2 k_3 \Omega_1(\k)}, \qquad \widehat{\H}_4 = -\frac{\pi}{12} \sum_{k \in \Z^*} \frac{|u_k|^4}{k^2}.
\end{equation}
These formulae are the result of straightforward computations which we briefly describe in the following lemma (which also proves \eqref{parte quartica - thm ham}).

\begin{lemma} \label{lemma: h4}
Given the Hamiltonian $\H \circ \Phi_{G_3}$, we have that
\[
\widetilde{\H}_4 = \frac{1}{2} \{ \H_3,G_3 \} = -\frac{\pi}{12} \sum_{\substack{ \k \in \M_4 \\ k_2+k_3 \neq 0}}  \frac{u^{\k}}{k_2 k_3}
\]
and choosing $G_4$ as in \eqref{G4} it follows
\[
\widehat{\H}_4 = -\frac{\pi}{12} \sum_{k \in \Z^*} \frac{|u_k|^4}{k^2}.
\]
\end{lemma}
\begin{proof}
By the linearity of the Poisson bracket we obtain
\[
\begin{split}
\widetilde{\H}_4 & = \frac{1}{2} \left \{ - \frac{\pi}{3} \sum_{\k \in \M_3} u^{\k} , i \frac{\pi}{9} \sum_{\k' \in \M_3} \frac{u^{\k'}}{k_1'k_2'k_3'} \right \} = - \frac{i\pi^2}{54} \sum_{\k,\k' \in \M_3} \frac{1}{k_1'k_2'k_3'} \{ u^{\k}, u^{\k'} \},
\end{split}
\]
where
\[
\begin{split}
\{ u^{\k}, u^{\k'} \} = \frac{i}{2 \pi} \left[  \right. &k_1 u_{k_2}u_{k_3} \left ( \delta(k_1=-k_1')u_{k_2'}u_{k_3'} + \delta(k_1=-k_2')u_{k_1'}u_{k_3'}  + \delta(k_1=-k_3')u_{k_1'}u_{k_2'} \right ) \\
+ &k_2u_{k_1}u_{k_3} \left ( \delta(k_2=-k_1')u_{k_2'}u_{k_3'} + \delta(k_2=-k_2')u_{k_1'}u_{k_3'} + \delta(k_2=-k_3')u_{k_1'}u_{k_2'} \right ) \\
+ &k_3u_{k_1}u_{k_2} \left.  \left ( \delta(k_3=-k_1')u_{k_2'}u_{k_3'} + \delta(k_3=-k_2')u_{k_1'}u_{k_3'} + \delta(k_3=-k_3')u_{k_1'}u_{k_2'} \right ) \right].
\end{split}
\]
By symmetry,
\[
\begin{split}
\widetilde{\H}_4 = - \frac{i\pi^2}{54} \cdot \frac{i}{2 \pi} \cdot 9 \cdot  \sum_{\substack{k_1+k_2+k_3=0 \\ -k_1+k_2'+k_3'=0}} \frac{k_1u_{k_2}u_{k_3}u_{k_2'}u_{k_3'}}{-k_1k_2'k_3'} 
\end{split}
\]
which yields the thesis upon renaming the indices.
Then
\[
\begin{split}
\{ \H_2,G_4 \} + \widetilde{\H}_4 =&\ -\frac{\pi}{12} \sum_{\substack{ \k \in \mathcal{R}_4^1 \\ k_2+k_3 \neq 0}}  \frac{u^{\k}}{k_2k_3} \\
=&\ - \frac{\pi}{12} \left( \sum_{k \in \Z^*} \frac{u_{-k}u_{k}u_{k}u_{-k}}{k^2} - \sum_{\substack{k \in \Z^* \\ j \in \Z \smallsetminus \{ 0, \pm k \} }} \frac{u_{j}u_{-j}u_{k}u_{-k}}{jk} - \sum_{\substack{k \in \Z^* \\ j \in \Z \smallsetminus \{ 0, \pm k \} }} \frac{u_{j}u_{k}u_{-j}u_{-k}}{jk}\right)
\end{split}
\]
using \Cref{lemma: orders 3 and 4}. 
Finally the last two sums vanish by symmetry.
\end{proof}

\medskip

\section{Dynamics} \label{subsection: dynamics}

In this section we exploit the results provided in \Cref{section: rigorous theory} to study the dynamics of the KdV equation. We note that global well-posedness of the KdV equation in $L^2 (\T)$ 
    is well-known (for instance \cite[Chapter 3]{erdogan}).
The following result allows us to approximate the dynamics over arbitrarily long (polynomial) timescales:

\begin{theorem}[Approximate dynamics for long timescales] \label{teorema: approssimazione con fasi lineari in t - con eps}
Let $r \ge 3$. 
Fix $s \ge 12 r^2$, then there exists $\epsilon_1 \lesssim_{r,s} 1$ such that for all $0<\epsilon\leq \epsilon_1$ the following holds. Let $u(0) \in \dot{H}^s$ such that 
$\norm{u(0)}_{\dot{H}^s} \le \epsilon$
and let $u(t)$ be the unique global solution to the KdV equation with initial datum $u(0)$, then for all $t \in [0,T]$ with
\begin{equation}
T = \epsilon^{-\frac{r}{5}}
\label{eq: timescales}
\end{equation}
we have that
\begin{equation}\label{eq: approx fasi lineari}
\norm{u(t)}_{\dot{H}^s} \le 2 \epsilon, \qquad \norm{u_j(t)-e^{i t \theta_j(0)}u_j(0)}_{\dot{H}^1} \le \epsilon^{\frac{8}{5}}, 
\end{equation}
where 
\begin{equation}
\theta_j(t) = j^3 - \frac{1}{6j} |\Phi^{-1}(u(t))_j|^2 + j \sum_{n=3}^{\lfloor \frac{r}{2} \rfloor} \, \, \sum_{\k \in (\Z^*)^{n-1}} b_{(j,\k)} |\Phi^{-1}(u(t)))^{\k}|^2, \qquad \lvert b_{(j,\k)} \rvert \lesssim_{r,s} \epsilon^{-\frac{2\#\k-1}{5}},
\label{def: phases 1}
\end{equation}
where $\Phi$ was defined in \Cref{thm: normal form for ham}.
\end{theorem}

The theorem is proved in \Cref{subsection: proof of thm dynamics}.

\begin{rem} \label{rem dyn 1}
The choice $s \ge 12 r^2$ is technical.
Choosing a large $N$ in \Cref{thm: normal form for ham} allows us to better control the remainders \eqref{terzo indice grande - thm ham}--\eqref{azione grande - thm ham}, but this gives rise to worse approximations of the dynamics on account of \eqref{eq: close to the identity - thm ham}. We exploit the additional regularity (with respect to \Cref{thm: normal form for ham}) in order to obtain better bounds on the remainder terms \eqref{terzo indice grande - thm ham}--\eqref{azione grande - thm ham} without making $N$ too large.
\end{rem}

\begin{rem} \label{rem dyn 2}
We do not try to optimize the relationship between $r$, $s$ and $T$. For instance, by taking $N\sim \epsilon^{0-}$, then one may potentially reach timescales $T\approx \epsilon^{-r+}$ at the cost of greatly increasing $s$ (as a function of $r$).
\end{rem}

\subsection{Control of the remainders}

In this subsection we control the remainders in \eqref{hamiltoniana dopo tutte le trasformazioni - thm ham}.
The following is a result analogous to \cite[Proposition 3.3]{bernier}, which allows us to estimate the Poisson bracket between the $\dot{H}^s$-norm squared and specific Hamiltonian functions which are like the remainders of the normal form \eqref{terzo indice grande - thm ham}--\eqref{resto grande - thm ham}.
The reason why we need such a result is that, in general, the vector field $X_R = \partial_x \nabla R$ doesn't map $\dot{H}^s$ into itself.
However we can bound the growth of the $\dot{H}^s$-norm along the flow of these remainders. A similar strategy was used in \cite{yuanzhang} in the context of the derivative NLS equation.

\,

\begin{prop} \label{prop: controllo resti}
Let $s \ge 2$, $M > 0$ and $n \ge 3$. 

\medskip

\noindent (i) Let $R = \sum_{\beta \ge r+1} \sum_{\k \in \M_\beta} b_{\k} u^{\k}$, with $r \ge 1$, be a 1-formal polynomial (cf. \Cref{def:formal_pol}) with $|b_{\k}| \le M^{\#\k}$ for all $\k \in \M_n$. 
Then
\[
\left \lvert \left \{ \norm{\cdot}_{\dot{H}^s}^2 , R \right \} (u) \right \rvert \lesssim_{s,r} \left ( M \norm{\cdot}_{\dot{H}^s} \right )^{r+1} \quad \text{for} \quad \norm{u}_{\dot{H}^s} \le \left ( 2 c_0 M \right )^{-1}
\]
where $c_0$ is a universal constant.

\medskip

\noindent (ii) Let $P_n$ be a homogeneous 1-formal polynomial of the form 
\begin{equation}\label{hom_pol}
P_n(u) = \sum_{ \substack{ \k \in \M_n \\ \mu_3(\k) \ge M } } b_{\k} u^{\k},
\end{equation}
where $\mu_3(\k)$ denotes the third largest number among $|k_1|,\ldots,|k_n|$ and $\bm{b}$ are some bounded coefficients. Then
\begin{equation}\label{hom_pol_bound}
\left \lvert \left \{ \norm{\cdot}_{\dot{H}^s}^2 , P_n \right \} (u) \right \rvert \le (2s+1) c_0^n n^{s+4} M^{-s+2} \norm{\bm{b}}_{\ell^{\infty}} \norm{u}_{\dot{H}^s}^n.
\end{equation}

\medskip

\noindent (iii) Let $P_{n+2}$ be a homogeneous 1-formal polynomial of degree $n+2$ of the form
\[
P_{n+2}(u) = \sum_{ \substack{ \k \in \M_n \\ j \ge M } } b_{(j,-j,\k)} I_j u^{\k},
\]
where $\bm{b}$ are some bounded coefficients. Then
\[
\left \lvert \left \{ \norm{\cdot}_{\dot{H}^s}^2 , P_{n+2} \right \} (u) \right \rvert \lesssim_{s,n} M^{-2s} \norm{\bm{b}}_{\ell^{\infty}} \norm{u}_{\dot{H}^s}^{n+2}.
\]
\end{prop}
\begin{proof}
(ii) Let $P_n$ be the homogeneous polynomial in \eqref{hom_pol}. By \Cref{lemma: facts about poisson bracket}-(ii), 
\begin{align}
\left \lvert \left \{ \norm{\cdot}_{\dot{H}^s}^2 , P_n \right \} (u) \right \rvert &\lesssim \norm{\bm{b}}_{\ell^{\infty}} \sum_{ \substack{ \k \in \M_n \\ \mu_3(\k) \ge M } } \left \lvert k_1^{2s+1}+k_2^{2s+1}+\ldots+k_n^{2s+1} \right \rvert \lvert u^{\k} \rvert \nonumber \\
&\lesssim n^3 \norm{\bm{b}}_{\ell^{\infty}} \sum_{ \substack{ \k \in \M_n \\ |k_1| \ge |k_2| \ge |k_3| \ge M,|k_4|,\ldots,|k_n| } } \left ( \left \lvert k_1^{2s+1}+k_2^{2s+1} \right \rvert +(n-2)|k_3|^{2s+1} \right ) \lvert u^{\k} \rvert \label{eq:use_young}
\end{align}
where we reordered the first three indices and used the triangle inequality.
Notice that by \eqref{young convolution inequality}, \eqref{eq: embedding hs in fl01} and the fact that $n \ge 3$ we have that
\begin{equation}\label{eq:young_conseq}
\sum_{\k \in \M_n} |k_1|^s |k_2|^s |k_3|^{s-1} \lvert u^{\k} \rvert \le \norm{|\pa_x|^s u}_{\dot{L}^2}^2 \norm{|\pa_x|^{s-1} u}_{\dot{\FL}^{0,1}} \norm{u}_{\dot{\FL}^{0,1}}^{n-3} \le c_0^{n-2} \norm{u}_{\dot{H}^s}^n
\end{equation}
using the same argument as in \Cref{lemma: convergence of hamiltonian function}.
Since
\[
|k_3|^{2s+1} \le |k_3|^{s-1} |k_2|^s |k_1|^2 \le M^{-s+2} |k_1|^s|k_2|^s|k_3|^{s-1}
\]
we have that the second summand in \eqref{eq:use_young} is bounded by
\[
\begin{split}
\sum_{ \substack{ \k \in \M_n \\ |k_1| \ge |k_2| \ge |k_3| \ge M,|k_4|,\ldots,|k_n| } } (n-2)|k_3|^{2s+1} \lvert u^{\k} \rvert \lesssim n c_0^{n-2} M^{-s+2} \norm{u}_{\dot{H}^s}^n.
\end{split}
\]

In order to bound the first summand in \eqref{eq:use_young}, we distinguish two cases:

\medskip

\noindent (a) If $k_1k_2 > 0$, then by the zero momentum condition $\k \in \M_n$ we obtain
\[
|k_1|,|k_2| \le n|k_3| \quad \text{and} \quad |k_1| \le n|k_2|,
\]
and so
\begin{equation}\label{casea}
\left \lvert k_1|k_1|^{2s} + k_2|k_2|^{2s} \right \rvert \le 2n^{s+1}|k_3|^{s-1}|k_1|^s|k_2|^2 \le 2n^{s+1} M^{-s+2} |k_1|^s|k_2|^s|k_3|^{s-1}.
\end{equation}

\medskip 

\noindent (b) If $k_1k_2 < 0$, since $\partial_x (x|x|^{2s}) = (2s+1)|x|^{2s}$, by the mean value theorem we obtain
\begin{equation}\label{caseb}
\begin{split}
\left \lvert k_1|k_1|^{2s} + k_2|k_2|^{2s} \right \rvert &\le (2s+1) \left \lvert | k_1 | - | k_2 | \right \rvert |k_1|^{2s} \le (2s+1) |k_1|^s n^s |k_2|^s n |k_3| \\
&\le (2s+1) n^{s+1} M^{-s+2} |k_1|^s |k_2|^s |k_3|^{s-1}.
\end{split}
\end{equation}

\medskip

We are ready to bound first summand in \eqref{eq:use_young}
\[
\sum_{ \substack{ \k \in \M_n \\ |k_1| \ge |k_2| \ge |k_3| \ge M,|k_4|,\ldots,|k_n| } } \left \lvert k_1^{2l+1}+k_2^{2l+1} \right \rvert \lvert u^{\k} \rvert \stackrel{\eqref{eq:young_conseq}-\eqref{caseb}}{\le} (2s+1) n^{s+1} c_0^{n-2} M^{-s+2} \norm{u}_{\dot{H}^s}^n
\]
which concludes the proof of point (ii).

\medskip

\noindent (iii) First, we note that $\sum_{j \ge M} I_j \le M^{-2s} \norm{u}_{\dot{H}^s}^2$.
Moreover, since the actions commute with the $\dot{H}^s$-norm, we have
\[
\begin{split}
\left \lvert \left \{ \norm{\cdot}_{\dot{H}^s}^2 , P_{n+2} \right \} (u) \right \rvert =&\ \sum_{j \ge M} I_j \bigg \lvert \bigg \{ \norm{u}_{\dot{H}^s}^2 ,\sum_{\k \in \M_n} b_{(j,-j,\k)}  u^{\k} \bigg \} \bigg \rvert 
\stackrel{\eqref{hom_pol_bound}}{\le} (2s+1) \norm{\bm{b}}_{\ell^{\infty}} n^{s+4} c_0^n \norm{u}_{\dot{H}^s}^{n}\, \sum_{j\geq M} I_j\\
& \le (2s+1) M^{-2s} \norm{\bm{b}}_{\ell^{\infty}} n^{s+4} c_0^n \norm{u}_{\dot{H}^s}^{n+2}
\end{split}
\]
where we used \eqref{hom_pol_bound} with $M=1$.

\medskip 

\noindent (i) Using \eqref{hom_pol_bound} with $M=1$, together with the bound $|b_{\k}| \le M^{\#\k}$, we find
\[
\left \lvert \left \{ \norm{\cdot}_{\dot{H}^s}^2 , R \right \} (u) \right \rvert \lesssim_s \sum_{n \ge r+1} n^{s+4} \left ( c_0 M \norm{u}_{\dot{H}^s} \right )^n\lesssim_{s,r} \left ( c_0 M \norm{u}_{\dot{H}^s} \right )^{r+1}
\]
provided $\norm{u}_{\dot{H}^s} \le \left ( 2 c_0 M \right )^{-1}$.
\end{proof}

\subsection{A priori bounds in the new variables}
We are ready to prove some estimates on the dynamics in the new variables $v$.
We will exploit the special structure of the Hamiltonian \eqref{hamiltoniana dopo tutte le trasformazioni} and \Cref{prop: controllo resti}, where we study exactly how the remainders \ref{azione grande}--\ref{resto grande} have an impact on the $\dot{H}^s$-norm of the solution.

\begin{theorem}[A priori bounds in the new variables] \label{thm: dynamics in the new variables}
Let $r \ge 3,$ $s \ge 2r$ and $N \ge N_0(r)$ (see \Cref{thm: normal form for ham}).
Let $v(t)$ be a solution to the Hamiltonian system associated with \eqref{hamiltoniana dopo tutte le trasformazioni - thm ham}.
There exists positive constants $c(r,s), C(r,s)$ such that, provided
\begin{equation}
\sup_{\tau \in [0,t]} \norm{v(\tau)}_{\dot{H}^s} \le c(r,s) N^{-1},
\label{eq: hp dynam new var}
\end{equation}
then
\begin{equation}
\begin{split}
\norm{v(t)}_{\dot{H}^s}^2 \le \norm{v(0)}_{\dot{H}^s}^2 + t \, C(r,s) \left [ N^{-\frac{s}{2r}} + \left ( N \norm{v}_{L^{\infty}([0,t],\dot{H}^s)} \right )^{r+1}  \right ]. 
\end{split}
\label{eq: controllo norma sobolev}
\end{equation}
Moreover, if for $j \in \Z^*$ we define the real-valued functions
\begin{equation}
\theta_j(t) = \frac{j}{\pi} \partial_{J_j} \left [ \H_2(J(t)) + \widehat{\H}_4(J(t)) + \sum_{n=3}^{\lfloor \frac{r}{2} \rfloor} \widehat{\H}_{2n}(J(t)) \right], \quad \text{where $J_j(t) = |v_j(t)|^2,$}
\label{def: phases 2}
\end{equation}
with $\widehat{\mathcal{H}}_{2n}$ in \Cref{thm: normal form for ham}, then 
\begin{equation}
\begin{split}
\norm{v_j(t)-e^{i t \theta_j(0)}v_j(0)}_{\dot{H}^1} & \le C(r,s)\ t \left ( 1 + t \norm{v(0)}_{\dot{H}^s} \right ) \, \left [ N^{-\frac{s}{2r}} + \left ( N \norm{v}_{L^{\infty}([0,t],\dot{H}^s)} \right )^r\ \right ].
\end{split}
\label{eq: controllo dinamica nuove variabili}
\end{equation}
\end{theorem}
\begin{proof} Thanks to \eqref{eq: hp dynam new var} at $t=0$, the fact that $\norm{\Phi (v)}_{\dot{H}^s}\leq 2 \norm{v}_{\dot{H}^s}$ (cf. \Cref{thm: normal form for ham}), and the local well-posedness of the KdV equation in $\dot{H}^s$, we have that the Hamiltonian system associated with \eqref{hamiltoniana dopo tutte le trasformazioni - thm ham} is locally well-posed and its solution is precisely $v(t) = \Phi^{-1}(u(t))$, cf. \cite{arnold, fasano}. 

\medskip

\noindent {\sc Estimate \eqref{eq: controllo norma sobolev}.}
Since $v$ is the solution to the Hamiltonian system associated with \eqref{hamiltoniana dopo tutte le trasformazioni - thm ham},
\begin{equation}\label{eq:Hs_variation}
\norm{v(t)}_{\dot{H}^s}^2 = \norm{v(0)}_{\dot{H}^s}^2 + \int_0^t \{ \norm{\cdot}_{\dot{H}^s}^2 , \H \circ \Phi \}(v(\tau)) \, \mathrm{d} \tau.
\end{equation}
By \Cref{lemma: facts about poisson bracket}-(iii) and \Cref{thm: normal form for ham}, we have that
\[
\{ \norm{\cdot}_{\dot{H}^s}^2 , \H \circ \Phi \} = \{ \norm{\cdot}_{\dot{H}^s}^2 , R^{(\mu_3 \ge N^{1/2r})} \} + \{ \norm{\cdot}_{\dot{H}^s}^2 , R^{(I \ge N^{1/2r})} \} + \{ \norm{\cdot}_{\dot{H}^s}^2 , R_{\ge r+1} \}.
\]
Now we estimate each term separately.
By \eqref{resto grande - thm ham}, \Cref{prop: controllo resti}-(i) applied with $M = \rho N$, and  \eqref{eq: hp dynam new var},
\[
\left \lvert \{ \norm{\cdot}_{\dot{H}^s}^2 , R_{\ge r+1} \}(v(t)) \right \rvert \lesssim_{r,s} \left ( \rho N \norm{v(t)}_{\dot{H}^s} \right)^{r+1} \lesssim_{r,s} \left ( N \norm{v(t)}_{\dot{H}^s} \right)^{r+1}.
\]
By \eqref{terzo indice grande - thm ham} and \Cref{prop: controllo resti}-(ii) applied with $M = N^{\frac{1}{2r}}$, we have
\[
\left \lvert \{ \norm{\cdot}_{\dot{H}^s}^2 , R^{(\mu_3 \ge N^{1/2r})} \}(v(t)) \right \rvert \lesssim_{r,s} N^{-\frac{s}{2r}} \sum_{n=5}^r 
\left ( N \norm{v(t)}_{\dot{H}^s} \right )^n \stackrel{\eqref{eq: hp dynam new var}}{\lesssim_{r,s}} N^{-\frac{s}{2r}}.
\]
Finally, by \eqref{azione grande - thm ham} and \Cref{prop: controllo resti}-(iii) applied with $M = N^{\frac{1}{2r}}$ we have
\[
\left \lvert \{ \norm{\cdot}_{\dot{H}^s}^2 , R^{(I \ge N^{1/2r})} \}(v(t)) \right \rvert \lesssim_{r,s} N^{-\frac{s}{r}} \sum_{n=5}^{r} \left ( N \norm{v(t)}_{\dot{H}^s} \right )^n \stackrel{\eqref{eq: hp dynam new var}}{\lesssim_{r,s}} N^{-\frac{s}{r}}.
\]
Using these estimates on \eqref{eq:Hs_variation}, we obtain \eqref{eq: controllo norma sobolev}.

\medskip

\noindent {\sc Estimate \eqref{eq: controllo dinamica nuove variabili}.} By the triangle inequality, we write
\begin{equation}\label{eq:step0_approx}
 \norm{v_j(t)-e^{i t \theta_j(0)}v_j(0)}_{\dot{H}^1} \leq \norm{v_j(t) - e^{i \int_0^t \theta_j(\tau) \, \mathrm{d} \tau}v_j(0)}_{\dot{H}^1} + \norm{e^{i \int_0^t \theta_j(\tau) \, \mathrm{d} \tau}v_j(0) - e^{i t \theta_j(0)}v_j(0)}_{\dot{H}^1} \ .
\end{equation}

\medskip

\noindent $\bullet$ \underline{First term in \eqref{eq:step0_approx}}. Note that $v_j(t)$ satisfies
\[
\partial_t v_j(t) = i \theta_j(t) v_j(t) + (\partial_x \nabla R)_j(v(t))
\]
where $\theta_j(t)\in\R$ is defined in \eqref{def: phases 2} and $R = R_{\ge r+1} + R^{(\mu_3 \ge N^{1/2r})} + R^{(I \ge N^{1/2r})}$, following \Cref{thm: normal form for ham}. Integrating in time, we find
\begin{equation}\label{eq:step1_approx}
\norm{v_j(t) - e^{i \int_0^t \theta_j(\tau) \, \mathrm{d} \tau}v_j(0)}_{\dot{H}^1} \le t \sup_{t' \in [0,t]} \norm{\nabla R(v(t'))}_{\dot{H}^2}
\end{equation}
We estimate each remainder term separately, but we bound $R^{(\mu_3 \ge N^{1/2r})}$ and $R^{(I \ge N^{1/2r})}$ simultaneously. Setting
$R^{(\mu_2 \ge N^{1/2r})}:= R^{(\mu_3 \ge N^{1/2r})}+ R^{(I \ge N^{1/2r})}$, \eqref{terzo indice grande - thm ham}--\eqref{azione grande - thm ham} yield
\begin{equation}\label{eq:R_mu2}
\left \lvert l^2 \partial_{u_{-l}} R^{(\mu_2 \ge N^{1/2r})}(v(t)) \right \rvert 
\lesssim_{r,s} \sum_{n=5}^r N^{n-3} \sum_{ \substack{ (\k,-l) \in \M_n \\ |k_1| \ge \ldots \ge |k_{n-1}| \\ |k_1| \ge N^{\frac{1}{2r}} } } |l|^2 |v_{k_1}(t)| \ldots |v_{k_{n-1}}(t)|
\end{equation}
and therefore
\[
\begin{split}
\norm{\partial_x \nabla R^{(\mu_2 \ge N^{1/2r})}(v(t))}_{\dot{H}^1} &  \stackrel{\eqref{eq:R_mu2}, \eqref{young convolution inequality}}{\lesssim_{r,s}} \sum_{n=5}^r N^{n-3} \norm{v(t)}_{\dot{H}_{|j| \ge N^{1/2r}}^2} \norm{v(t)}_{\dot{\FL}^{0,1}}^{n-2}  \\
&\lesssim_{r,s} \sum_{n=5}^r N^{n-3} N^{\frac{-s+2}{2r}} \norm{v(t)}_{\dot{H}^s}^{n-1}
\stackrel{\eqref{eq: hp dynam new var}}{\lesssim_{r,s}} N^{-\frac{s}{2r}}.
\end{split}
\]
Finally, we bound the last remainder using \eqref{resto grande - thm ham},
\[
\begin{split}
\norm{\partial_x \nabla R_{\ge r+1}(v(t))}_{\dot{H}^1} & \le  \sum_{n = r+1}^{+\infty} \rho^n N^{n-3} \norm{\sum_{(\k,-l) \in \M_{n}} |l|^2 |v_{k_1}(t)| \ldots |v_{k_{n-1}}(t)|}_{\dot{\ell}_l^2} \\
& \stackrel{\eqref{young convolution inequality}}{\lesssim_{r,s}} N^{-2} \sum_{n = r+1}^{+\infty} n \left ( \rho N c_0 \norm{v(t)}_{\dot{H}^s} \right )^{n-1} \lesssim_{r,s} \left ( N \norm{v(t)}_{\dot{H}^s} \right )^r
\end{split}
\]
provided that $\rho N c_0 \norm{v(t)}_{\dot{H}^s} < 1$, which follows by \eqref{eq: hp dynam new var} by choosing $c(r,s)\leq (\rho c_0)^{-1}$. Combining the remainder bounds with \eqref{eq:step1_approx} yields:
\begin{equation} \label{eq:step2_approx}
\begin{split}
\norm{v_j(t) - e^{i \int_0^t \theta_j(\tau) \, \mathrm{d} \tau}v_j(0)}_{\dot{H}^1} \le t \, C(r,s) \left [ N^{-\frac{s}{2r}} + \left ( N \norm{v}_{L^{\infty}([0,t],\dot{H}^s)} \right )^r \right ].
\end{split}
\end{equation}

\medskip

\noindent $\bullet$ \underline{Second term in \eqref{eq:step0_approx}}. First, we bound:
\begin{equation}\label{eq:step3_approx}
\norm{e^{i \int_0^t \theta_j(\tau) \, \mathrm{d} \tau}v_j(0) - e^{i t \theta_j(0)}v_j(0)}_{\dot{H}^1}^2 \leq \sum_{j\in\Z^{*}} |j|^2 |v_j(0)|^2\, \Big | 1- e^{i\int_0^t [\theta(\tau)-\theta(0)]\, \mathrm{d}\tau}\Big |^2 \ .
\end{equation}
Next we study the phase difference, which,
by \Cref{thm: normal form for ham}, is 
\[
\begin{split}
\int_0^t [\theta (\tau)-\theta (0)]\, \mathrm{d}\tau = \! \! \int_0^t \!\!\! \bigg ( \frac{|v_j(0)|^2-|v_j(\tau)|^2}{6j} \! + \! j \sum_{n=3}^{\lfloor \frac{r}{2} \rfloor} \sum_{\k \in (\Z^*)^{n-1}} \! \!\!\!\!\! \left [ b_{(j,k_1,\ldots,k_{n-1})} + \ldots + b_{(k_1,\ldots,k_{n-1},j)} \right ] \! \left [ J(\tau)^{\k} - J(0)^{\k} \right ] \! \bigg ) \, \frac{\mathrm{d}\tau}{\pi},
\end{split}
\]
hence by \eqref{parte quartica - thm ham}--\eqref{parte integrabile 2 - thm ham},
\[
\begin{split}
\left \lvert  \int_0^t [\theta (\tau)-\theta (0)]\, \mathrm{d}\tau \right \rvert \lesssim_r \int_0^t \left [ \frac{1}{|j|} \left \lvert |v_j(\tau)|^2-|v_j(0)|^2 \right \rvert + |j| \sum_{n=3}^{\lfloor \frac{r}{2} \rfloor} N^{2n-3} \, \, \sum_{\k \in (\Z^*)^{n-1}} \left \lvert J(\tau)^{\k} - J(0)^{\k} \right \rvert \right ] \, \mathrm{d} \tau.
\end{split}
\]
By the reverse triangle inequality
\[
\left \lvert |v_j(\tau)|^2-|v_j(0)|^2 \right \rvert = \left ( |v_j(\tau)|+|v_j(0)| \right ) \big \lvert |v_j(\tau)|-|v_j(0)| \big \rvert \le \left ( |v_j(\tau)|+|v_j(0)| \right ) \left \lvert v_j(\tau) - e^{i \int_0^t \theta_j(\tau) \, \mathrm{d}\tau} v_j(0) \right \rvert\ ,
\]
and the Cauchy-Schwartz inequality
\begin{equation} \label{eq: stima utile in prob}
\begin{split}
\sum_{\k \in (\Z^*)^{n-1}} \left \lvert J(\tau)^{\k} - J(0)^{\k} \right \rvert &\lesssim_r \sup_{\tau' \in [0,\tau]} \norm{v(\tau')}_{\dot{L}^2}^{2n-4} \sum_{k \in \Z^*} \left ( |v_{k}(\tau)|+|v_{k}(0)| \right ) \big | |v_{k}(\tau)|-|v_{k}(0)| \big | \\
&\lesssim_r \sup_{\tau' \in [0,\tau]} \norm{v(\tau')}_{\dot{L}^2}^{2n-3} \norm{v_j(\tau) - e^{i \int_0^t \theta_j(\tau) \, \mathrm{d}\tau} v_j(0)}_{\dot{L^2}},
\end{split}
\end{equation}
we have that
\[
\begin{split}
 \left \lvert  \int_0^t [\theta (\tau)-\theta (0)]\, \mathrm{d}\tau \right\rvert  & \stackrel{\eqref{eq: hp dynam new var}}{\lesssim_r} \int_0^t \left [ \frac{1}{|j|} \left ( |v_j(\tau)|+|v_j(0)| \right ) \left \lvert v_j(\tau) - e^{i \int_0^t \theta_j(\tau) \, \mathrm{d}\tau} v_j(0) \right \rvert + |j| \norm{v_k(\tau) - e^{i \int_0^t \theta_k(\tau) \, \mathrm{d}\tau} v_k(0)}_{\dot{L}^2} \right ] \, \mathrm{d} \tau\\
& \lesssim_r 
t \, |j|\,\sup_{\tau\in [0,t]} \norm{v_k(\tau) - e^{i \int_0^{\tau} \theta_k(\tau') \, \mathrm{d}\tau'} v_k(0)}_{\dot{\ell}^2_k}.
\end{split}
\]

Plugging this bound into \eqref{eq:step3_approx}, we obtain:
\[
\norm{e^{i \int_0^t \theta_j(\tau) \, \mathrm{d} \tau}v_j(0) - e^{i t \theta_j(0)}v_j(0)}_{\dot{H}^1}  \lesssim_r t \!\! \sup_{\tau \in [0,t]} \norm{v_j(\tau) - e^{i \phi_j(\tau)} v_j(0)}_{\dot{H}^s} \!\! \norm{v(0)}_{\dot{H}^s}.
\]
Using \eqref{eq:step2_approx} and \eqref{eq:step0_approx}, we obtain the desired bound \eqref{eq: controllo dinamica nuove variabili}.
\end{proof}

\begin{rem} \label{rem dyn 3}
   It is possible to obtain bounds on the $H^{s-1}$-norm of the approximation instead of the $H^1$-norm in \eqref{eq: controllo dinamica nuove variabili}. To do so, it suffices to include in $\theta (t)$ certain terms from $\pa_x \nabla R^{(I \ge N^{1/2r})}$ coming from differentiating the action, as is done in \cite[Section~7.1.3]{bernier}. However, in our case it suffices to use a norm which controls $L^{\infty}_x$, cf. \Cref{thm: almost global ldp intro}.




It is also possible to obtain approximations over longer timescales (potentially up to \(T = \epsilon^{-r+}\) in \eqref{eq: timescales}) by using nonlinear phases \(\int_0^t \theta_j(\tau)\,\mathrm{d}\tau\) instead of the linear phases \(t\theta_j(0)\). These extended timescales rely on the improved bound \eqref{eq:step2_approx}, in contrast with the quadratic-in-time bound \eqref{eq: controllo dinamica nuove variabili}. This strategy is used in \cite{bertimasperograndestaffilani} for the Gravity Water Waves system (with \(r=3\)).

However, the linear approximation \(t\theta_j(0)\) has the key advantage of depending only on the initial datum, which is essential for propagating statistical information over long times. More precisely, the size of the quasi-synchronization event \eqref{def: intorno} is inversely proportional to the Lipschitz constant of the map sending the initial data to the phases \(t\theta_j(0)\), and this constant grows only linearly in time. By contrast, using the phases \(\int_0^t \theta_j(\tau)\,\mathrm{d}\tau\) may lead to exponential growth of this Lipschitz constant\footnote{These bounds would further deteriorate if one used the phases in \cite[Section~7.1.3]{bernier}, which include non-integrable terms.}, as suggested by standard Gr\"onwall estimates. If sharp, such growth would obstruct the proof of \Cref{thm: almost global ldp - lower bound}.
\end{rem}

\subsection{Approximate dynamics for long timescales} \label{subsection: proof of thm dynamics}



In this subsection we prove \Cref{teorema: approssimazione con fasi lineari in t - con eps}.
We want to exploit the normal form provided in \Cref{thm: normal form for ham}. 
For fixed $r$ and $s$, the theorem allows us to do the normal form transformation in any ball of radius $\epsilon\leq c(r,s) N^{-1}$. In particular, we may fix
\begin{equation}
N = \epsilon^{-\frac{1}{5}},
\label{eq: choice of N}
\end{equation}
provided that $\epsilon\leq \min\{N_0(r)^{-5}, c(r,s)^{\frac54}\}$.
Set $v(0) = \Phi^{-1}(u(0))$ to be the initial datum in the new coordinates. Then
\begin{equation}\label{eq:bootstrap0}
\norm{v(0)}_{\dot{H}^s} \le \norm{u(0)}_{\dot{H}^s} + \norm{u(0)-v(0)}_{\dot{H}^s} \stackrel{\eqref{eq: close to the identity - thm ham}}{\le} \epsilon + C(r,s) \epsilon^{\frac{9}{5}} \le \frac{4}{3} \epsilon
\end{equation}
provided that $\epsilon \lesssim_{r,s} 1$ is sufficiently small.
We proceed with a bootstrap argument.
Let
\[
\begin{split}
I = \bigg \{ t \in [0,T] \, | \, &\sup_{t' \in [0,t]} \norm{u(t')}_{\dot{H}^s} \le 2 \epsilon \bigg \},
\end{split}
\]
where $T=\epsilon^{-\frac{r}{5}}$ as given in \eqref{eq: timescales}. Clearly, $0 \in I$ and $I$ is closed. To show that $I$ is open, fix $t \in I$.
Since $\norm{u(t)}_{\dot{H}^s} \le 2 \epsilon$, we have that
\[
\begin{split}
\norm{v(t)}_{\dot{H}^s} &\le \norm{u(t)}_{\dot{H}^s} + \norm{u(t)-v(t)}_{\dot{H}^s} \stackrel{\eqref{eq: close to the identity - thm ham}}{\le} 2 \epsilon + C(r,s) \epsilon^{-\frac{1}{5}} \norm{u(t)}_{\dot{H}^s}^2 \le C(r,s) \epsilon \ll \epsilon^{\frac{1}{5}} = N^{-1}.
\end{split}
\]
Since $v(t)$ satisfies \eqref{eq: hp dynam new var}, we can apply \eqref{eq: controllo norma sobolev}:
\[
\begin{split}
\norm{v(t)}_{\dot{H}^s}^2 &\le \norm{v(0)}_{\dot{H}^s}^2 + C(r,s) \epsilon^{-\frac{r}{5}} \left [ \epsilon^{\frac{s}{10r}} + \epsilon^{\frac{4}{5}(r+1)} \right] \le \norm{v(0)}_{\dot{H}^s}^2 + C(r,s)\epsilon^{\frac52}
\end{split}
\]
provided that $s \ge 12 r^2$, which implies that
\begin{equation}\label{eq:bootstrap1}
\norm{v(t)}_{\dot{H}^s} \le \norm{v(0)}_{\dot{H}^s} + \frac{1}{3}\epsilon.
\end{equation}
Therefore we have that
\[
\begin{split}
\norm{u(t)}_{\dot{H}^s} & \le \norm{v(t)}_{\dot{H}^s} + \norm{u(t)-v(t)}_{\dot{H}^s} \stackrel{\eqref{eq:bootstrap1},\eqref{eq: close to the identity - thm ham}}{\le} \norm{v(0)}_{\dot{H}^s} + \frac{1}{3}\epsilon + C(r,s) \epsilon^{-\frac{1}{5}} \norm{u(t)}_{\dot{H}^s}^2 \\
& \stackrel{\eqref{eq:bootstrap0}}{\le} \frac{5}{3}\epsilon + C(r,s) \epsilon^{\frac{9}{5}} < 2 \epsilon
\end{split}
\]
provided that $\epsilon \lesssim_{r,s} 1$ is sufficiently small.
This shows that $I=[0,T]$, and thus yields \eqref{eq: controllo norma sobolev} for $t \in [0,T]$.

Finally, let $\theta_j$ be the phases in \eqref{def: phases 2}. Using \eqref{eq: choice of N} and \eqref{parte integrabile 2 - thm ham}, we obtain \eqref{def: phases 1}.
In order to prove \eqref{eq: approx fasi lineari}, we insert \eqref{eq: timescales} and \eqref{eq: choice of N} into \eqref{eq: controllo dinamica nuove variabili}:
\begin{equation}\label{eq:bootstrap2}
\begin{split}
\norm{v_j(t)-e^{i t \theta_j(0)}v_j(0)}_{\dot{H}^1} & \le C(r,s) \epsilon^{-\frac{r}{5}} \left ( 1 + \epsilon^{-\frac{r}{5}+1} \right ) \left [ \epsilon^{\frac{s}{10r}} + \epsilon^{\frac45 r} \right] \le \epsilon^2
\end{split}
\end{equation}
provided that $r \ge 3$, $s \ge 12r^2$ and that $\epsilon \lesssim_{r,s} 1$ is sufficiently small. 
Finally, by the triangle inequality,
\[
\begin{split}
\norm{u_j(t)-e^{i t \theta_j(0)}u_j(0)}_{\dot{H}^1} &\le \norm{u(t)-v(t)}_{\dot{H}^s} + \norm{v_j(t)-e^{i t \theta_j(0)}v_j(0)}_{\dot{H}^1} + \norm{u(0)-v(0)}_{\dot{H}^s} \\
& \stackrel{\eqref{eq: close to the identity - thm ham},\eqref{eq:bootstrap2}}{\le} C(r,s) \left [ \epsilon^{-\frac{1}{5}} \norm{u(t)}_{\dot{H}^s}^2 + \epsilon^{2} \right ] \le \epsilon^{\frac{8}{5}}.
\end{split}
\]

\section{Large deviations principle} \label{subsection: probability}

In this section, we prove an almost global large deviations principle for the KdV equation \eqref{intro: kdv} with random initial data \eqref{intro: random data} satisfying the more flexible condition \eqref{def: coeff ck} below, instead of \eqref{intro: coefficients}.


\begin{theorem}[Almost global LDP for KdV] \label{thm: almost global ldp intro}
Let $n \in \N$, $\lambda > 0$ and $\delta \in (0,1)$. Let $u_{\varepsilon}^{\omega}(t,x)$ be the (a.s.\ global) solution to \eqref{intro: kdv} with initial datum \eqref{intro: random data} satisfying: \begin{equation} \label{def: coeff ck}
0\leq c_k = c_{-k}, \quad \mbox{and} \quad \sum_{k \in \Z^*} |k|^{2\sigma} |c_k|^2 < + \infty \quad \mbox{for some}\  \sigma \ge s = \max \left \{  12 \left \lceil \frac{7n}{1-\delta} \right \rceil^2 , \frac{1}{\delta} \right \}.
\end{equation} 
Then we have that
\begin{equation} \label{eq: limit ldp}
\lim_{\varepsilon\to 0^{+}} \sup_{|t|\leq \varepsilon^{-n}} \left \lvert \varepsilon^{2\delta} \log \P \left( \sup_{x \in \T} u_{\varepsilon}^{\omega}(t,x) \ge \lambda \varepsilon^{1-\delta} \right) + \frac{\lambda^2}{4 \sum_{k\in\N} c_k^2} \right\rvert = 0.
\end{equation}
\end{theorem}

\medskip

Note the main theorem \Cref{thm:main_intro} follows from \Cref{thm: almost global ldp intro}. The proof of \Cref{thm: almost global ldp intro} is divided into a matching upper bound (\Cref{thm: almost global ldp}) and lower bound (\Cref{thm: almost global ldp - lower bound}). 

\medskip

The upper bound follows from the approximate solution \eqref{eq: approx fasi lineari} constructed in \Cref{teorema: approssimazione con fasi lineari in t - con eps}, where the moduli of the Fourier coefficients are conserved in time. As a result, its $\FL^{0,1}$-norm, which controls the $L^{\infty}$-norm, is constant in time. This allows us to propagate upper bounds from the initial data to any polynomial time.

The proof of the lower bound is more involved and requires the machinery introduced in \cite{bertimasperograndestaffilani}. The idea is to study the probability that the $L^{\infty}$-norm of the solution is close to its $\FL^{0,1}$-norm, which occurs when many phases are \emph{quasi-synchronized}. This quasi-synchronization event contains the nonlinear structure of the problem, as the phases of the approximate solution are highly nonlinear functions of the initial data (via the normal form transformation), cf. \eqref{def: phases 1}. Proving that the probability of this event is sufficiently large over arbitrarily long polynomial timescales is the main challenge in the proof of the lower bound.

\smallskip

Our next result essentially follows from the proof of \Cref{thm: almost global ldp intro}, and it shows that the dispersive focusing mechanism is dominant in the formation of extreme waves for the KdV equation.
Under the same assumptions as in \Cref{thm: almost global ldp intro}, consider the a.s. global solution $u_{\varepsilon}^{\omega}(t,x)$ in \eqref{intro: exact solution} to the KdV equation \eqref{intro: kdv} with initial datum \eqref{intro: random data} satisfying \eqref{def: coeff ck}.
Recall also the set $\mathfrak{P}(\tN,\delta,\varepsilon)$ in \eqref{intro: fasi piccole}, i.e. the event that the first $\tN$ phases are very close to zero. 

\begin{theorem}[Dispersive Focusing] \label{thm: dispersive focusing}
Under the same assumptions of \Cref{thm: almost global ldp intro} we have 
\begin{equation} \label{eq: dispersive focusing}
\lim_{\varepsilon\to 0^{+}} \sup_{|t|\leq \varepsilon^{-n}} \left \lvert \varepsilon^{2\delta} \log \P \left( \left \{ \sup_{x \in \T} u_{\varepsilon}^{\omega}(t,x) \ge \lambda \varepsilon^{1-\delta} \right \} \cap \mathfrak{P}(\tN,\delta,\varepsilon) \right) + \frac{\lambda^2}{4 \sum_{k\in\N} c_k^2} \right \rvert =0.
\end{equation}
\end{theorem}
\noindent This result will be proved in \Cref{subsection: dispersive focusing}.

\paragraph{Preliminary results.} 
Given $\varepsilon > 0$, we study the KdV equation with random initial data:
\begin{equation}
\begin{dcases}
u_t + u_{xxx} + uu_x = 0 \\
u_{\varepsilon}^{\omega}(0,x) = \varepsilon \sum_{k \in \Z^*} c_k \eta_k^{\omega} e^{ikx} = \varepsilon \widetilde{u}_0^{\omega}(x),
\end{dcases}
\label{def: pdc con random iv (preliminari)}
\end{equation}
where the coefficients $(c_k)_{k \in \Z^*}$ satisfy \eqref{def: coeff ck}. 
The sequence $(\eta_k^{\omega})_{k \in \N}$ is made of i.i.d standard complex Gaussian random variables (cf.~\eqref{intro: complex gaussian}) on a complete probability space $(\mathbb{\Omega}, \mathcal{F}, \P)$. 
Moreover we fix $\eta_{-k}^{\omega} = \overline{\eta_k^{\omega}}$ for all $k\in\N$, which together with the fact that $c_k=c_{-k}$ in \eqref{def: coeff ck}, guarantees that the random initial datum is real valued.

It is well known \cite{rayleigh} that $\eta_k^{\omega} = R_k^{\omega} e^{i \phi_k^{\omega}}$ where $R_k^{\omega} \ge 0$ is a Rayleigh r.v. $\sim \mathtt{R} \left ( \frac{1}{\sqrt{2}} \right )$ and $\phi_k^{\omega} \in [0,2\pi)$ is uniformly distributed.
Furthermore $(R_k^{\omega})_{k \in \N}$ and $(\phi_k^{\omega})_{k \in \N}$ are independent.
Therefore we can write
\begin{equation} \label{eq: random data}
\begin{split}
u_{\varepsilon}^{\omega}(0,x) =&\ \varepsilon \sum_{k=1}^{+\infty} c_k R_k^{\omega} \left ( e^{i \phi_k^{\omega} + ikx} + e^{-i \phi_k^{\omega} -ikx} \right ) = 2 \varepsilon \sum_{k=1}^{+\infty} c_k R_k^{\omega} \cos(kx + \phi_k^{\omega}).
\end{split}
\end{equation}

One can easily check that for any fixed $x \in \T$, $u_{\varepsilon}^{\omega}(0,x)$ is itself a centered real Gaussian random variable and that $u_{\varepsilon}^{\omega}(0) \in \dot{H}^{s}(\T)$ for $s$ in \eqref{def: coeff ck} almost surely, since 
\begin{equation} \label{eq: estimates random initial datum}
\mathbb{E}\norm{u_{\varepsilon}^{\omega}(0)}_{\dot{H}^s}^2 = 2 \varepsilon^2 \sum_{k = 1}^{+\infty} |k|^{2s}\, c_k^2, \qquad  \mathbb{E}\norm{u_{\varepsilon}^{\omega}(0)}_{\FL^{0,1}} = \sqrt{\pi}\, \varepsilon\, \sum_{k = 1}^{+\infty} c_k.
\end{equation}

\medskip 

We start with a large deviation result for the initial datum $\norm{u_{\varepsilon}^{\omega}(0)}_{\FL^{0,1}}$, 
which was proved in \cite[Section~2]{grande} for $c_k =\O( e^{-b|k|})$, $b>0$, and in \cite[Proposition~3.1]{liang} for $(c_k)_{k\in\Z}\in \ell^1$.

\begin{prop} \label{prop: grande}
Let $\lambda> 0$ and consider a sequence $(c_k)_{k \in \N}\in\ell^1$. Consider also a sequence
$(R_k^{\omega})_{k \in \N}$ of i.i.d. Rayleigh r.v. $\sim \mathtt{R} \left ( \frac{1}{\sqrt{2}} \right )$ on the same probability space $(\mathbb{\Omega}, \mathcal{F}, \P)$.
Then
\begin{equation} \label{eq: grande}
\lim_{\varepsilon \to 0^+} \varepsilon^2 \log \P \left ( \varepsilon \sum_{k \in \N} c_k R_k^{\omega} \ge \lambda \right ) = - \frac{\lambda^2}{\sum_{k=1}^{+\infty} c_k^2}.
\end{equation}
\end{prop}
\begin{rem}
   We note that the Fourier coefficients of $u_{\varepsilon}^{\omega}(0)$ are not independent, since $\eta_k = \overline{\eta_{-k}}$. As a result, whenever we use \Cref{prop: grande}, we will have $\norm{u_{\varepsilon}^{\omega}(0)}_{\FL^{0,1}} = \sum_{k\in\Z^*} c_k R_k^{\omega} = 2 \sum_{k\in\N} c_k R_k^{\omega}$, since $c_k=c_{-k}$. This explains the factor $4$ appearing in \eqref{eq: limit ldp}.
\end{rem}

This result can be easily extended to the linear flow of \eqref{def: pdc con random iv (preliminari)} which remains Gaussian. However, the nonlinear dynamics are only well-approximated by the linear flow for timescales $|t|\ll \varepsilon^{-2(1-\delta)}$. In order to go beyond these timescales, we use the approximation developed in \Cref{teorema: approssimazione con fasi lineari in t - con eps}:
\begin{equation}\label{eq: random approximate solution}
u_{\mathrm{app},\varepsilon,r}^{\omega}(t,x) = 2 \, \varepsilon \sum_{j \in \N} c_j R_j^{\omega} \cos\left( t \theta_j(u_{\varepsilon}^{\omega}(0)) + \phi_j^{\omega} + jx\right), \qquad r \ge \frac{7n}{1-\delta}
\end{equation}
where the phases $\theta_j$ are defined in \eqref{def: phases 1} and the choice of $r$ will be justified in \Cref{rem: r n delta e tutti gli amici}.
We highlight that, for $\sigma \ge 0$ and $p \in [1,+\infty]$,
\begin{equation}\label{eq:conserved_app_norm}
\norm{u_{\mathrm{app},\varepsilon,r}^{\omega}(t)}_{\dot{\FL}^{\sigma ,p}} = \norm{u_{\varepsilon}^{\omega}(0)}_{\dot{\FL}^{\sigma ,p}} = \varepsilon \norm{\widetilde{u}_0^{\omega}}_{\dot{\FL}^{\sigma,p}}.
\end{equation}

\subsection{Upper bound} \label{subsection: upper bound}

We are now ready to prove the upper bound of the almost global LDP for the KdV equation.

\begin{theorem} \label{thm: almost global ldp}
Consider the Cauchy problem with random initial data \eqref{def: pdc con random iv (preliminari)}.
Let $u_{\varepsilon}^{\omega}(t,x)$ be the corresponding unique a.s.\ global solution.
Fix $n \in \N$, $\delta \in (0,1)$, $\lambda > 0$. Then
\begin{equation}\label{eq:limsup}
\limsup_{\varepsilon \to 0^{+}} \sup_{|t|\leq \varepsilon^{-n}} \varepsilon^{2\delta} \log \P \left( \sup_{x \in \T} u_{\varepsilon}^{\omega}(t,x) \ge \lambda \varepsilon^{1-\delta} \right) \le - \frac{\lambda^2}{4 \sum_{k=1}^{+\infty} c_k^2}.
\end{equation}
\end{theorem}
\begin{proof}
Consider the approximation $u_{\mathrm{app},\varepsilon,r}^{\omega}$ in \eqref{eq: random approximate solution} and define:
\[
\mathcal A_{\varepsilon}  = \left \{ \omega \in \mathbb{\Omega} \, | \,  \sup_{x\in\T} u_{\varepsilon}^{\omega}(t,x) \ge \lambda \varepsilon^{1-\delta} \right \}, \quad
\mathcal B_{\varepsilon}  = \left \{ \omega \in \mathbb{\Omega} \, | \,  \norm{u_{\varepsilon}^{\omega}(t) - u_{\mathrm{app},\varepsilon,r}^{\omega}(t)}_{L_x^{\infty}} \ge \varepsilon^{\frac{8}{7} (1-\delta)} \right \}.
\]
First, we write
\begin{equation}\label{eq:intersection0}
\P(\A_{\varepsilon}) = \P(\A_{\varepsilon}\cap \B_{\varepsilon})+\P(\A_{\varepsilon}\cap \B_{\varepsilon}^c).
\end{equation}
By the reverse triangle inequality, we have that
\[
\begin{split}
\P( \A_{\varepsilon} \cap \B_{\varepsilon}^c) 
& \le \P \left( \left \{  \sup_{x\in\T} u_{\varepsilon}^{\omega}(t,x) \ge \lambda \varepsilon^{1-\delta} \right \} \cap \left \{  \sup_{x\in\T} u_{\varepsilon}^{\omega}(t,x)- \norm{u_{\mathrm{app},\varepsilon,r}^{\omega}(t)}_{L_x^{\infty}} < \varepsilon^{\frac{8}{7}(1-\delta)} \right \} \right) \\
& \le \P \left( \norm{u_{\mathrm{app},\varepsilon,r}^{\omega}(t)}_{L_x^{\infty}} \ge \lambda \varepsilon^{1-\delta} - \varepsilon^{\frac{8}{7}(1-\delta)} \right)\stackrel{\eqref{eq:conserved_app_norm}}{\leq} \P \left( \norm{u_{\varepsilon}^{\omega}(0)}_{\FL^{0,1}} \ge \lambda \varepsilon^{1-\delta} - \varepsilon^{\frac{8}{7}(1-\delta)} \right).
\end{split}
\]
In particular, \Cref{prop: grande} and the fact that $\delta<1$ imply that
\begin{equation}\label{eq:intersection1}
    \log \P( \A_{\varepsilon} \cap \B_{\varepsilon}^c) = -\frac{\lambda^2 \varepsilon^{-2\delta}}{4\sum_{k\in\N} c_k^2} + o(\varepsilon^{-2\delta}) \qquad \mbox{as}\ \varepsilon\rightarrow 0.
\end{equation}

Next we bound the second term in \eqref{eq:intersection0}. 
By \Cref{teorema: approssimazione con fasi lineari in t - con eps} applied with $\epsilon=\varepsilon^{\frac57 (1-\delta)}$, and the embeddings $L^{\infty}(\T) \supseteq H^s(\T) \supseteq \FL^{s,1}(\T)$ for $s > \frac{1}{2}$,
we have that $\B_{\varepsilon} \subset \{ \omega\in\mathbb{\Omega} \mid \norm{u^{\omega}_{\varepsilon}(0)}_{\FL^{s,1}}\geq \varepsilon^{\frac57 (1-\delta)}\}$. As a result, 
\begin{equation}\label{eq:intersection2}
\begin{split}
\log \P(\A_{\varepsilon} \cap \B_{\varepsilon}) \le \log \P(\B_{\varepsilon}) \leq \log \P\left(\norm{u^{\omega}_{\varepsilon}(0)}_{\FL^{s,1}}\geq \varepsilon^{\frac57 (1-\delta)}\right) \leq - \frac{ \varepsilon^{-\frac47 - \frac{10}{7} \delta}}{8\sum_{k\in\N} k^{2s} c_k^2}
\end{split}
\end{equation}
provided $\varepsilon$ is small enough, where in the last inequality we used \eqref{eq: grande}, with $c_k \mapsto k^s c_k$ (cf. \eqref{def: norm flsp}). Note that when using \Cref{teorema: approssimazione con fasi lineari in t - con eps}, we must impose condition \eqref{def: coeff ck} on the regularity $s$ in terms of $\delta$ and $n$, which we explain in detail in \Cref{rem: r n delta e tutti gli amici}.

\smallskip

Finally, in view of  \eqref{eq:intersection1}--\eqref{eq:intersection2}, for sufficiently small $\varepsilon$ we can arrange $\P(\A_{\varepsilon} \cap \B_{\varepsilon}^c) \leq \P(\A_{\varepsilon} \cap \B_{\varepsilon})$. Then \eqref{eq:intersection0} and \eqref{eq:intersection1} yield \eqref{eq:limsup}, since the errors are uniform in time.
\end{proof}

\begin{rem}\label{rem: r n delta e tutti gli amici}
We highlight that, when using \Cref{teorema: approssimazione con fasi lineari in t - con eps} with $\epsilon=\varepsilon^{\frac57 (1-\delta)}$, we can reach timescales $T=( \varepsilon^{\frac57 (1-\delta)})^{-\frac{r}{5}}$. Imposing that such times are $T\geq \varepsilon^{-n}$ leads to fixing some $r\geq \frac{7n}{1-\delta}$. Moreover, \Cref{teorema: approssimazione con fasi lineari in t - con eps} requires $s\geq 12 r^2$, which explains \eqref{def: coeff ck}.
\end{rem}

\subsection{Lower bound} \label{subsection: the lower bound}

Controlling the sup-norm of the solution via the $\FL^{0,1}$-bound allowed us to exploit the quasi-conservation of the Fourier moduli in time, which was the key to deriving the upper bound \Cref{thm: almost global ldp}. In order to obtain a lower bound, however, we cannot disregard the dynamics of the phases of the Fourier coefficients. 

In view of \eqref{eq: random approximate solution}, these nonlinear phases are given by 
$\phi_j^{\omega} + t\,\theta_j (u_{\varepsilon}^{\omega}(0))$. While $\phi_j^{\omega}$ is uniformly distributed in $[0,2\pi)$, the nonlinear phases are not on account of the nontrivial way in which $\theta_j (u_{\varepsilon}^{\omega}(0))$ depend on the initial datum (and thus on the initial phases). Roughly speaking, the key idea is to study the set of quasi-synchronized phases (cf. \eqref{eq: stability})
\begin{equation}\label{eq:quasi-synch intro}
\{ \omega\in\mathbb{\Omega} \mid |\phi_j^{\omega} + t\,\theta_j (u_{\varepsilon}^{\omega}(0))|\ll 1 \quad \mbox{for many}\ j\}.
\end{equation}
Within this set, we can replace the $L^{\infty}$-norm by the $\FL^{0,1}$-norm, which allows us to exploit the quasi-invariance of Fourier moduli once again. 
One of the main challenges is therefore to show that the probability of the quasi-synchronization set does not become too small 
over very long timescales $|t| \le \varepsilon^{-n}$, $n\in\N$. In order to prove this property we start by constructing phases which are precisely synchronized. 
To do so, we think of the initial datum as a (deterministic) function of a finite (but growing) number of phases, cf. \eqref{def: partially rid}.

\paragraph{Partially randomized initial datum.}
Let $M\in\N$ to be fixed later, and define
\begin{equation}
\Pi_{> M} : \dot{L}^2(\T,\R) \longrightarrow \dot{L}^2(\T,\R), \qquad \Pi_{> M} \left ( \sum_{j \in \Z^*} f_j e^{ijx} \right ) = \sum_{|j| > M} f_j e^{ijx}.
\end{equation}
For $\vec{\phi} \in \R^M$ consider the \textit{partially randomized initial datum} (cf. \eqref{eq: random data})
\begin{equation} \label{def: partially rid}
\widetilde{\tu}_0^{\omega}(x ; \vec{\phi} ) = 2 \sum_{j = 1}^M c_j R_j^{\omega} \cos(jx + \phi_j) + \underbrace{ 2 \sum_{j=M+1}^{+\infty} c_j R_j^{\omega} \cos(jx + \phi_j^{\omega}) }_{ \Pi_{>M} \widetilde{\tu}_0^{\omega}(x) }
\end{equation}
where $(\phi_j^{\omega})_{j > M}$ are i.i.d. uniformly distributed r.v. in $[0,2\pi)$, $(R_j^{\omega})_{j \in \N}$ are i.i.d. Rayleigh r.v. $\sim \mathtt{R} \left ( \frac{1}{\sqrt{2}} \right )$, and $(\phi_j^{\omega})_{j > M}$ and $(R_j^{\omega})_{j \in \N}$ are independent.

\begin{rem}
If $\vec{\phi}^{\omega} = (\phi_j^{\omega})_{j=1}^M$ are i.i.d. uniformly distributed r.v. in $[0,2\pi)$ and independent of $(R_j^{\omega})_{j \in N}$ and $(\phi_j^{\omega})_{j > M}$, then our initial datum coincides with \eqref{def: pdc con random iv (preliminari)}, namely
    \[
    \widetilde{\tu}_0^{\omega}(x ; \vec{\phi}^{\omega} ) = \widetilde{u}_0^{\omega}(x).
    \]
\end{rem}

For $\varepsilon > 0$, $\delta \in (0,1)$, and $s$ in \eqref{def: coeff ck}, and recalling \eqref{def: pdc con random iv (preliminari)}, we define the event
\begin{equation} \label{def: event small size initial datum}
\begin{split}
\mathcal{B}(\delta,\varepsilon,s) & = \left \{ \omega \in \mathbb{\Omega} \ \Big | \ \varepsilon \norm{\widetilde{u}^{\omega}_0}_{\dot{H}^s} = \varepsilon \norm{\widetilde{\tu}^{\omega}_0}_{\dot{H}^s} \le \varepsilon^{\frac57 (1-\delta)} \right \}.
\end{split}
\end{equation}
Notice that this definition does not depend on $M$, as the $H^s$-norm does not depend on the phases.

We also introduce the $\sigma$-algebra $\mathcal{G}$ generated by the random moduli and the high-Fourier modes, i.e. 
    \begin{equation} \label{def: sigma algebra senza fasi basse}
    \mathcal{G} = \sigma ( (R_j^{\omega})_{j \in \N} , (\phi_j^{\omega})_{j > M} ). 
    \end{equation}
Note that the event $\mathcal{B}(\delta,\varepsilon,s)$  depends only on the random variables $(R_j^{\omega})_{j \in \N}$. As a result, $\mathcal{B}(\delta,\varepsilon,s)$ is $\mathcal{G}$-measurable. It will also be convenient to define  the restriction of $\mathcal{G}$ to the event $\mathcal{B}$ denoted by 
\begin{equation}\label{def:tildeG}
    \widetilde{\mathcal{G}} = \{ G \cap \mathcal{B}(\delta,\varepsilon,s) \, \lvert \, G \in \mathcal{G} \}\subseteq \mathcal{G}.
\end{equation}
Finally, we introduce the \textit{restricted probability space} $(\mathcal{B}(\delta,\varepsilon,s) , \widetilde{\mathcal{G}} , \widetilde{\P} )$, where $\widetilde{\P}(G) = \P(G)$ for all $G \in \widetilde{\mathcal{G}}$. Without loss of generality, we can assume that $(\mathbb{\Omega},\mathcal{G},\mathbb{P})$ and $(\mathcal{B}(\delta,\varepsilon,s), \widetilde{\mathcal{G}},\widetilde{\P})$ are complete and independent of the initial uniformly distributed phases $\vec{\phi}^{\omega}$.

\paragraph{Random fixed point.}

For any $\omega\in\mathcal{B}(\delta,\varepsilon,s)$, we may write the phases in \eqref{eq: random approximate solution} explicitly (cf. \eqref{def: phases 1} and recall $\epsilon=\varepsilon^{\frac57 (1-\delta)}$) as:
\begin{equation}\label{eq: approx_phases random}
\theta_j(u_{\varepsilon}^{\omega}(0)) = j^3 - \frac{1}{6j} |\Phi^{-1}(u_{\varepsilon}^{\omega}(0))_j|^2 + j \sum_{m=3}^{\lfloor \frac{r}{2} \rfloor} \, \, \sum_{\k \in (\Z^*)^{m-1}} b_{(j,\k)} |\Phi^{-1}(u_{\varepsilon}^{\omega}(0))^{\k}|^2,  \qquad |b_{(j,\k)}| \lesssim_r \varepsilon^{-\frac{2\#k-3}{7}(1-\delta)},
\end{equation}
where $\Phi^{-1}$ are defined in \Cref{thm: normal form for ham}, and $r \ge \frac{7n}{1-\delta}$. This choice of $r$ is made so that $u_{\mathrm{app},\varepsilon,r}^{\omega}$  in \eqref{eq: random approximate solution} is a good approximation of the KdV dynamics under long timescales, as explained in \Cref{rem: r n delta e tutti gli amici}.

\medskip

We thus define the random operator
\begin{equation} \label{def: random phase nonlinear operator}
\begin{split}
& \mathcal{T} : \R^M \times \mathcal{B}(\delta,\varepsilon,s) \rightarrow \R^M,  \\
& \mathcal{T}_j(\vec{\phi},\omega) = -t \theta_j(\varepsilon\, \widetilde{\tu}_0^{\omega}(\cdot ; \vec{\phi})), \qquad j=1,\ldots,M.
\end{split}
\end{equation}

This operator allows us to rewrite the problem of finding a point in \eqref{eq:quasi-synch intro} where phases exactly synchronize into a fixed point problem for the mapping $\mathcal{T}$. Below we will prove that for each $\omega\in\mathcal{B}(\delta,\varepsilon,s)$ there exists 
$\vec{\phi}^{*,\omega} \in \R^M$ such that 
\[
t \theta_j(\varepsilon \widetilde{\tu}_0^{\omega}(\cdot ; \vec{\phi}^{*,\omega})) + \vec{\phi}^{*,\omega} = 0 \qquad \forall \, \, 1 \le j \le M.
\]
The key difficulty is proving that $\vec{\phi}^{*,\omega}$ is $\widetilde{\mathcal{G}}$-measurable. We prove this via a random Brouwer fixed point theorem \cite{fixed point}, in the spirit of classical measurable selection theorems such as Kuratowski–Ryll-Nardzewski, see also \cite[Chapter~3.2]{measurable selection}.
The following proposition is an adaptation of \cite[Proposition 5.3]{bertimasperograndestaffilani}. 

\begin{prop}[Random fixed point] \label{prop: properties of random fixed point}
Let $M,n \in \N$, $\delta \in (0,1)$, 
and $s$ in \eqref{def: coeff ck}.
There exists $\varepsilon_0 = \varepsilon_0(n,\delta) > 0$ such that for any $\varepsilon \in (0,\varepsilon_0)$ and for any $|t| \le \varepsilon^{-n}$, the random operator $\mathcal{T}$ defined in \eqref{def: random phase nonlinear operator} satisfies the following properties:

\medskip




\smallskip

\noindent (i) {\sc Fixed point:} There exists a random variable $\vec{\phi}^{*,\omega} = ( \phi_j^{*,\omega} )_{j=1}^M$ in $(\mathbb{\Omega},\mathcal{G},\mathbb{P})$ such that for every $\omega\in \mathcal{B}(\delta,\varepsilon,s)$
\[
\mathcal{T} ( \vec{\phi}^{*,\omega} , \omega ) = \vec{\phi}^{*,\omega}.
\]

\smallskip

\noindent (ii) {\sc Neighborhood of the fixed point:} If $\vec{\phi}^{\omega} = (\phi_j^{\omega})_{j=1}^M$ are i.i.d uniformly distributed r.v. in $[0,2\pi)$ and independent of $\mathcal{G}$ in \eqref{def: sigma algebra senza fasi basse}, $\beta \in (0,\pi)$ and we define
\begin{equation} \label{def: intorno}
\mathcal{N}(\beta) = \{ \omega \in \mathbb{\Omega} \, \big \lvert \, | \phi_j^{\omega} - \phi_j^{*,\omega} | < \beta \quad \forall j = 1,\ldots,M \},
\end{equation}
then $\mathcal{N}(\beta)$ is $\mathcal{F}$-measurable and $\P(\mathcal{N}(\beta)) = \left ( \frac{\beta}{\pi} \right )^M$.

\smallskip

\noindent (iii) {\sc Stability property:} For any $\omega \in \mathcal{N}(\beta) \cap \mathcal{B}(\delta,\varepsilon,s)$ we have
\begin{equation} \label{eq: stability}
| \phi_j^{\omega} - \mathcal{T} ( \phi_j^{\omega} , \omega ) | \le (1+M|t|)\beta, \qquad j = 1,\ldots,M.
\end{equation}

\smallskip

\noindent (iv) {\sc Factorization property:} For any measurable set $\A \in \mathcal{G}$,
\begin{equation} \label{eq: factorization}
\P (\A \cap \mathcal{N}(\beta)) = \left ( \frac{\beta}{\pi} \right )^M \P ( \A ).
\end{equation}
\end{prop}

It is worth to notice that the factorization property \eqref{eq: factorization} holds even if $\mathcal{N}(\beta)$ is not independent of $\mathcal{G}$.
The proof of \Cref{prop: properties of random fixed point} is postponed to \Cref{subsection: proof prop fixed point}. 
In particular the proof of \eqref{eq: factorization} relies on \Cref{lemma: properties of conditional expectation} below.
Finally we prove the lower bound of the large deviations principle which, together with \Cref{thm: almost global ldp}, yields \Cref{thm: almost global ldp intro}.

\begin{theorem} \label{thm: almost global ldp - lower bound}
Consider the Cauchy problem \eqref{def: pdc con random iv (preliminari)} with random initial data satisfying \eqref{def: coeff ck}, which a.s. admits a unique global solution $u_{\varepsilon}^{\omega}(t,x)$.
Fix $n \in \N$, $\delta \in (0,1)$, $\lambda > 0$. 
Then 
\begin{equation}\label{eq: almost global ldp - lower bound}
\liminf_{\varepsilon \to 0^{+}}  \inf_{|t|\leq \varepsilon^{-n}} \varepsilon^{2\delta} \log \P \left( \sup_{x \in \T} u_{\varepsilon}^{\omega}(t,x) \ge \lambda \varepsilon^{1-\delta} \right) + \frac{\lambda^2}{4 \sum_{k=1}^{+\infty} c_k^2} \geq 0.
\end{equation}
\end{theorem}
\begin{proof}
Recall $\B(\delta,\varepsilon,s)$ in \eqref{def: event small size initial datum} and the approximate solution $u_{\mathrm{app},\varepsilon,r}^{\omega}$ in \eqref{eq: random approximate solution}.
By the reverse triangle inequality, the embedding $\dot{H}^s(\T) \subseteq L^{\infty}(\T)$, and \Cref{teorema: approssimazione con fasi lineari in t - con eps} applied with $\epsilon=\varepsilon^{\frac57 (1-\delta)}$, we have
\begin{align}
\P \left( \sup_{x \in \T} u_{\varepsilon}^{\omega}(t,x) \ge \lambda \varepsilon^{1-\delta} \right) \ge \P & \left( \left \{ \sup_{x \in \T} u_{\mathrm{app},\varepsilon,r}^{\omega}(t,x) \ge \lambda \varepsilon^{1-\delta} + C_s \norm{u_{\varepsilon}^{\omega}(t)-u_{\mathrm{app},\varepsilon,r}^{\omega}(t)}_{\dot{H}^1} \right \} \cap \B(\delta,\varepsilon,s) \right) \nonumber \\
\stackrel{\eqref{def: event small size initial datum},\eqref{eq: approx fasi lineari}}{\ge} & \P  \left( \left \{ \sup_{x \in \T} u_{\mathrm{app},\varepsilon,r}^{\omega}(t,x) \ge \lambda \varepsilon^{1-\delta} + C_s \varepsilon^{\frac{8}{7}(1-\delta)} \right \} \cap \B(\delta,\varepsilon,s) \right ) \label{for dispersive focusing} \\
\ge \P & \left( \left \{ u_{\mathrm{app},\varepsilon,r}^{\omega}(t,0) \ge \lambda \varepsilon^{1-\delta} + C_s \varepsilon^{\frac{8}{7}(1-\delta)} \right \} \cap \B(\delta,\varepsilon,s) \cap \mathcal{N}(\beta)  \right ), \label{third inequality}
\end{align}
where $\mathcal{N}(\beta)$ was defined in \eqref{def: intorno}. For $\omega \in \B(\delta,\varepsilon,s) \cap \mathcal{N}(\beta)$, and the estimate $\cos y\geq 1-y^2/2$, we have
\[
\begin{split}
u_{\mathrm{app},\varepsilon,r}^{\omega}(t,0) & \stackrel{\eqref{eq: random approximate solution},\eqref{def: random phase nonlinear operator}}{\ge} 2 \varepsilon \sum_{k=1}^M c_k R_k^{\omega} \cos(\phi_k^{\omega} - \mathcal{T}_k(\vec{\phi}^{\omega},\omega))  - \left \lvert \varepsilon \sum_{|k| > M} c_k R_k^{\omega} \cos(\phi_k^{\omega} + t \theta_k^{\omega}) \right \rvert \\
& \stackrel{\eqref{eq: stability}}{\ge} 2 \varepsilon \left ( 1 - \frac{(1 + M |t|)^2\,\beta^2}{2} \right ) \sum_{k=1}^M c_k R_k^{\omega} - C \norm{u_{\varepsilon}^{\omega}(0)}_{H_{|k| \ge M}^1} \\
& \stackrel{\eqref{def: event small size initial datum}}{\ge} 2 \varepsilon \left ( 1 - M^2\varepsilon^{-2n}\beta^2 \right ) \sum_{k=1}^M c_k R_k^{\omega} - C M^{1-s} \varepsilon^{\frac{5}{7} (1-\delta)}.
\end{split}
\]
We fix $\beta= M^{-1}\, \varepsilon^{n+\frac12}$, and continue with the lower bound:
\[
\begin{split}
\eqref{third inequality} & \ge \P \left( \underbrace{\left \{ \sum_{k=1}^M c_k R_k^{\omega} \ge \frac{ \lambda \varepsilon^{1-\delta} + C_s \varepsilon^{\frac{8}{7}(1-\delta)} + C M^{1-s} \varepsilon^{\frac{5}{7}(1-\delta)} }{ 2\varepsilon\,(1 - \varepsilon)} \right \} \cap \B(\delta,\varepsilon,s)}_{\mathcal{A} \ \mathrm{in\ \eqref{eq: factorization}}} \cap \, \, \mathcal{N}(\beta) \right ) \\
& \stackrel{\eqref{eq: factorization}}{\ge} \left ( \frac{\beta}{\pi} \right )^M \P \left( \left \{ \sum_{k=1}^M c_k R_k^{\omega} \ge \frac{ \lambda \varepsilon^{1-\delta} + C_s \varepsilon^{\frac{8}{7}(1-\delta)} + C M^{1-s} \varepsilon^{\frac{5}{7}(1-\delta)} }{2 \varepsilon \left ( 1 - \varepsilon\right )} \right \} \cap \B(\delta,\varepsilon,s) \right ) \\
& \stackrel{\eqref{def: event small size initial datum}}{\ge} \left ( \frac{\beta}{\pi} \right )^M \left[ \P \left( \sum_{k=1}^M c_k R_k^{\omega} \ge \frac{ \lambda \varepsilon^{1-\delta} + C_s \varepsilon^{\frac{8}{7}(1-\delta)} + C M^{1-s} \varepsilon^{\frac{5}{7}(1-\delta)} }{2 \varepsilon \left ( 1 - \varepsilon\right )} \right ) - \P \left ( \norm{\tilde{u}_0^{\omega}}_{\dot{H^s}} > \varepsilon^{-\frac{2}{7}-\frac{5}{7}\delta} \right ) \right].
\end{split}
\]
Let us choose $M = \lfloor \varepsilon^{-\delta} \rfloor$. As a result of this choice, we have that
\begin{equation}\label{eq : pre liminf}
\begin{split}
\varepsilon^{2\delta} & \log \P \bigg ( \sup_{x \in \T} u_{\varepsilon}^{\omega}(t,x) \ge \lambda \varepsilon^{1-\delta} \bigg ) \ge M \varepsilon^{2\delta} \log \left ( \frac{\beta}{\pi} \right ) \\
& + \varepsilon^{2\delta} \log \left[ \P \left( \sum_{k=1}^M c_k R_k^{\omega} \ge \frac{ \lambda \varepsilon^{1-\delta} + C_s \varepsilon^{\frac{8}{7}(1-\delta)} + C M^{1-s} \varepsilon^{\frac{5}{7}(1-\delta)} }{2 \varepsilon \left ( 1 - \varepsilon\right )} \right ) - \P \left ( \norm{\tilde{u}_0^{\omega}}_{\dot{H^s}} > \varepsilon^{-\frac{2}{7}-\frac{5}{7}\delta} \right ) \right],
\end{split}
\end{equation}
with
\begin{equation}\label{eq: quasisynch zero}
0\geq \lim_{\varepsilon\to 0^{+}} M \varepsilon^{2\delta} \log \left ( \frac{\beta}{\pi} \right ) \ge \lim_{\varepsilon\to 0^{+}} \varepsilon^{\delta} \log \left( \frac{\varepsilon^{n+\frac{1}{2}-\delta}}{\pi} \right) = 0.
\end{equation}
Recall that $s \geq \frac{1}{\delta}>\frac{2+5\delta}{7\delta}$ by \eqref{def: coeff ck}. This choice of $s$ ensures that $M^{1-s}\varepsilon^{\frac57 (1-\delta)} \ll \varepsilon^{1-\delta}$. As a result,
\[
\frac{ \lambda \varepsilon^{1-\delta} + C_s \varepsilon^{\frac{8}{7}(1-\delta)} + C M^{1-s} \varepsilon^{\frac{5}{7}(1-\delta)} }{2 \varepsilon \left ( 1 - \varepsilon\right )} = 
\lambda \varepsilon^{-\delta} + o(\varepsilon^{-\delta})
\]
Finally, for any $S \in \N$ fixed, \eqref{eq : pre liminf} yields
\[
\begin{split}
& \liminf_{\varepsilon \to 0^+} \inf_{|t|\leq \varepsilon^{-n}} \varepsilon^{2\delta} \log \P \bigg ( \sup_{x \in \T} u_{\varepsilon}^{\omega}(t,x) \ge \lambda \varepsilon^{1-\delta} \bigg )\\
& \stackrel{\eqref{eq: quasisynch zero}}{\ge}  \liminf_{\varepsilon \to 0^+} \inf_{|t|\leq \varepsilon^{-n}} \varepsilon^{2\delta} \log \left [ \P \left ( 2 \sum_{k=1}^S c_k R_k^{\omega} \ge \lambda \varepsilon^{-\delta} + o(\varepsilon^{-\delta}) \right ) - \P \left ( \norm{\tilde{u}_0^{\omega}}_{\dot{H^s}} > \varepsilon^{-\frac{2}{7}-\frac{5}{7}\delta} \right ) \right ]  \stackrel{\eqref{eq: grande},\eqref{eq:intersection2}}{=}  - \frac{\lambda^2}{4 \sum_{k=1}^S c_k^2},
\end{split}
\]
and we obtain \eqref{eq: almost global ldp - lower bound} by letting $S \to +\infty.$
\end{proof}

\smallskip

\subsection{Dispersive focusing} \label{subsection: dispersive focusing}

In this subsection we prove \Cref{thm: dispersive focusing}.
Notice that the upper bound for the $\limsup$ in \eqref{eq: dispersive focusing} follows immediately by \eqref{eq: limit ldp}.
Let us prove the lower bound.
First let $\kappa = \min \left \{ \frac{\delta}{2} , \frac{1-\delta}{14} \right \}$ and consider
\begin{equation} \label{eq: moduli grandi}
\mathfrak{R}(\tN, \kappa, \varepsilon) = \left \{ \omega \in \mathbb{\Omega} \, | \, R_j^{\omega} \ge \varepsilon^{-\delta + \kappa}, \, j=1,\ldots,\tN \right \}.
\end{equation}
We claim that
\begin{itemize}
\item[(i)] Outside the set $\mathfrak{R}(\tN,\kappa,\varepsilon)$, extreme waves are less likely to appear,
\begin{equation} \label{disp foc 1}
 \limsup_{\varepsilon \to 0^+} \sup_{|t|\leq \varepsilon^{-n}} \varepsilon^{2\delta} \log \P \left( \left \{ \sup_{x \in \T} u_{\mathrm{app},\varepsilon}^{\omega}(t,x) \ge \lambda \varepsilon^{1-\delta} + o(\varepsilon^{1-\delta}) \right \} \cap \mathfrak{R}(\tN,\kappa,\varepsilon)^c \right ) < - \frac{\lambda^2}{4 \sum_{k \in \N} c_k^2};
\end{equation}
    \item[(ii)] For $\varepsilon \lesssim_{\tN,\delta} 1$ sufficiently small,
\begin{equation} \label{inclusion}
\mathcal{N}(\beta) \cap \mathfrak{R}(\tN, \kappa, \varepsilon) \cap \B(\delta,\varepsilon,s) \subseteq \mathfrak{P}(\tN,\delta,\varepsilon) \cap \B(\delta,\varepsilon,s),
\end{equation}
where $\mathcal{N}(\beta)$ is defined in \eqref{def: intorno}, with $M = \lfloor \varepsilon^{-\delta} \rfloor$ and $\beta=M^{-1} \varepsilon^{n+\frac12}$ as in \Cref{thm: almost global ldp - lower bound}, and $\B(\delta,\varepsilon,s)$ is defined in \eqref{def: event small size initial datum}.
\end{itemize}
Assume for the moment these two claims, and recall the approximate solution $u_{\mathrm{app},\varepsilon,r}^{\omega}$ in \eqref{eq: random approximate solution}. Then 
\[
\begin{split}
& \P \left( \left \{ \sup_{x \in \T} u_{\varepsilon}^{\omega}(t,x) \ge \lambda \varepsilon^{1-\delta} \right \} \cap \mathfrak{P}(\tN,\delta,\varepsilon) \right) \\
& \stackrel{\eqref{for dispersive focusing}}{\ge} \P \left( \left \{ u_{\mathrm{app},\varepsilon,r}^{\omega}(t,0) \ge \lambda \varepsilon^{1-\delta} + o(\varepsilon^{1-\delta}) \right \} \cap \B(\delta,\varepsilon,s) \cap \mathfrak{P}(\tN,\delta,\varepsilon)  \right ) \\
& \stackrel{\eqref{inclusion}}{\ge} \P \left( \left \{ u_{\mathrm{app},\varepsilon,r}^{\omega}(t,0) \ge \lambda \varepsilon^{1-\delta} + o(\varepsilon^{1-\delta}) \right \} \cap \mathcal{N}(\beta) \cap \mathfrak{R}(\tN, \kappa, \varepsilon) \cap \B(\delta,\varepsilon,s)  \right ) \\
& \ge \P \left( \left \{ u_{\mathrm{app},\varepsilon,r}^{\omega}(t,0) \ge \lambda \varepsilon^{1-\delta} + o(\varepsilon^{1-\delta}) \right \} \cap \mathcal{N}(\beta) \cap \B(\delta,\varepsilon,s)  \right ) - \P \left( \left \{ u_{\mathrm{app},\varepsilon,r}^{\omega}(t,0) \ge \lambda \varepsilon^{1-\delta} + o(\varepsilon^{1-\delta}) \right \} \cap \mathfrak{R}(\tN, \kappa, \varepsilon)^c  \right ) \\
& \stackrel{\eqref{disp foc 1}}{\ge} \exp \left ( -\frac{\lambda^2 \varepsilon^{-2\delta}}{4 \sum_{k \in \N} c_k^2} + o(\varepsilon^{-2\delta}) \right )
\end{split}
\]
proceeding as in \Cref{thm: almost global ldp - lower bound}. As a result, we need only prove the two claims (i) and (ii).

\medskip

\noindent {\sc Claim }(i). It is sufficient to notice that
\[
\begin{split}
& \P \left( \left \{ \sup_{x \in \T} u_{\mathrm{app},\varepsilon}^{\omega}(t,x) \ge \lambda \varepsilon^{1-\delta} + o(\varepsilon^{1-\delta}) \right \} \cap \mathfrak{R}(\tN,\kappa,\varepsilon)^c \right) \\
\stackrel{\eqref{eq: moduli grandi}}{\le} & \sum_{k=1}^{\tN} \P \left( \left \{ \sup_{x \in \T} u_{\mathrm{app},\varepsilon}^{\omega}(t,x) \ge \lambda \varepsilon^{1-\delta} + o(\varepsilon^{1-\delta}) \right \} \cap \left \{ R_k^{\omega} < \varepsilon^{-\delta + \kappa} \right \} \right) \\
\stackrel{\eqref{eq: random approximate solution}}{\le} & \sum_{k=1}^{\tN} \P \left( \left \{ \sum_{j \in \Z^*} c_j R_j^{\omega} \ge \lambda \varepsilon^{-\delta} + o(\varepsilon^{-\delta}) \right \} \cap \left \{ R_k^{\omega} < \varepsilon^{-\delta + \kappa} \right \} \right) \le \sum_{k=1}^{\tN} \P \left( \sum_{j \in \Z^* \smallsetminus \{ \pm k \} } c_j R_j^{\omega} \ge \lambda \varepsilon^{-\delta} + o(\varepsilon^{-\delta}) \right) \\
\stackrel{\eqref{eq: grande}}{\le} & \tN \exp \left ( -\frac{\lambda^2 \varepsilon^{-2\delta}}{4 \sum_{k \in \N} c_k^2 - 4 \min \{ c_k^2,  \, \, k=1,\ldots,\tN \}} + o(\varepsilon^{-2\delta}) \right )
\end{split}
\]
and \eqref{disp foc 1} follows by taking logarithms, taking the supremum in $|t|\leq \varepsilon^{-n}$ and taking $\varepsilon\rightarrow 0^{+}$.

\medskip

\noindent {\sc Claim }(ii). Let $\omega \in \mathcal{N}(\beta) \cap \mathfrak{R}(\tN, \kappa, \varepsilon) \cap \B(\delta,\varepsilon,s)$.
We claim that $\omega \in \mathfrak{P}(\tN,\delta,\varepsilon)$, cf. \eqref{intro: fasi piccole}.
By the triangle inequality we obtain (cf. \eqref{eq: random approximate solution})
\begin{equation}\label{disp foc 2}
\begin{split}
\varepsilon \sum_{j \in \Z^*} c_j R_j^{\omega} \left \lvert e^{i \psi_j^{\omega}(t)} - e^{i \phi_j^{\omega} + i t \theta_j(u_0^{\omega})} \right \rvert &\le \varepsilon \sum_{j \in \Z^*} \left \lvert  c_j R_j^{\omega} - |u_j(t)| \right \rvert + \varepsilon \sum_{j \in \Z^*} \left \lvert c_j R_j^{\omega} e^{i \phi_j^{\omega} + i t \theta_j(u_0^{\omega})} - u_j(t) \right \rvert \stackrel{\eqref{eq: approx fasi lineari}}{\le} 2 \varepsilon^{\frac{8}{7}(1-\delta)}
\end{split}
\end{equation}
using \Cref{teorema: approssimazione con fasi lineari in t - con eps} with $\epsilon = \varepsilon^{\frac{5}{7}(1-\delta)}$, cf. \eqref{def: event small size initial datum}.
Within $\mathcal{N}(\beta)$ we have 
\begin{equation} \label{stima fasi approx}
\left \lvert \phi_j^{\omega} + t \theta_j(u_0^{\omega}) \right \rvert \stackrel{\eqref{eq: stability}}{\le} (1+M \varepsilon^{-n}) \beta \le 2 \varepsilon^{\frac12},
\end{equation}
since $M = \lfloor \varepsilon^{-\delta} \rfloor$ and $\beta=M^{-1} \varepsilon^{n+\frac12}$. Using the fact that $c_j\neq 0$ in \eqref{intro: fasi piccole},
\[
\left \lvert e^{i \psi_j^{\omega}(t)} - e^{i \phi_j^{\omega} + i t \theta_j(u_0^{\omega})} \right \rvert \stackrel{\eqref{eq: moduli grandi}, \eqref{disp foc 2}}{\lesssim_{\tN}} \varepsilon^{\frac{8}{7}(1-\delta)} \varepsilon^{-1+\delta-\kappa} \le \varepsilon^{\frac{1-\delta}{14}}
\]
and by the triangle inequality
\[
\left \lvert e^{i \psi_j^{\omega}(t)} - 1 \right \rvert \stackrel{\eqref{stima fasi approx}}{\le} \varepsilon^{\frac{1-\delta}{14}} + \varepsilon^{\frac12} \le \frac12 \varepsilon^{\frac{1-\delta}{15}}
\]
provided that $\varepsilon \lesssim_{\tN,\delta} 1$ is sufficiently small. 
Finally, we use the inequality $\frac{|x|}{2} \le |\sin(x)| \le |e^{ix}-1|$ for $|x| \ll 1$. 

\hfill \qedsymbol

\subsection{Random Brouwer fixed point} \label{subsection: proof prop fixed point}

\begin{lemma}[Properties of conditional expectation] \label{lemma: properties of conditional expectation}
Let $(\mathbb{\Omega}, \mathcal{F},\P)$ be a probability space.
Let $\mathcal{G} \subseteq \mathcal{F}$ be a sub $\sigma$-algebra.

\noindent 1) If $X \in L^1(\mathbb{\Omega},\mathcal{F},\P)$, then
\begin{equation} \label{tower property}
\E \left [ \E \left [ X \lvert \mathcal{G} \right ] \right ] = \E \left [ X \right ].
\end{equation}
2) If $X \in L^1(\mathbb{\Omega},\mathcal{F},\P)$ and $Y \in L^{\infty}(\mathbb{\Omega},\mathcal{F},\P)$ is $\mathcal{G}$-measurable, then
\begin{equation} \label{take out what is known}
\E \left [ XY \lvert \mathcal{G} \right ] = Y \E \left [ X \lvert \mathcal{G} \right ].
\end{equation}
3) If $X : \mathbb{\Omega} \rightarrow \R^n$ is a $\mathcal{G}$-measurable random variable, $Y : \mathbb{\Omega} \rightarrow \R^m$ is a random variable independent of $\mathcal{G}$ 
and $\Psi : \R^n \times \R^m \rightarrow \R$ is a Borel function such that $\Psi(X,Y)$ and $\Psi(x,Y)$ are integrable ($\forall x \in \R^n$), then
\begin{equation} \label{eq: resnick}
\E \left [ \Psi(X,Y) \lvert \mathcal{G} \right ] = \E \left [ \Psi(x,Y) \right ]_{\lvert x=X}.
\end{equation}
\end{lemma}
\begin{proof}
See \cite[Chapter 10]{resnick}.
\end{proof}

For the existence of the fixed point we use \Cref{thm: random fixed point} below, whose proof may be found in \cite[Theorem~10]{fixed point}.

\begin{theorem}[Random Brouwer fixed point] \label{thm: random fixed point}
Let $(\mathbb{\Omega} , \mathcal{F} , \P )$ be a complete probability space and $K$ be a convex and compact subset of $\R^M$.
Let $\mathcal{T} : K \times \mathbb{\Omega}  \longrightarrow K$ be such that:
\begin{enumerate}
    \item for any fixed $\omega \in \mathbb{\Omega}$, the map $\mathcal{T}(\cdot,\omega) : K \longrightarrow K$ is continuous;
    \item for any $\phi \in K$, the map $\mathcal{T}(\phi,\cdot) : \mathbb{\Omega} \longrightarrow K$ is measurable.
\end{enumerate}
Then there exists a random variable $\phi^{\omega}$ such that $\mathcal{T}(\phi^{\omega},\omega) = \phi^{\omega}$ almost surely in $\mathbb{\Omega}.$
\end{theorem}

\medskip

\noindent {\bf Proof of \Cref{prop: properties of random fixed point}.} We divide the proof in several points as in the statement.

\medskip 

\noindent (i) {\sc Fixed point.} In order to use \Cref{thm: random fixed point}, we prove the following two properties:

\begin{itemize}
\item[(i.1)] \underline{Lipschitzianity with respect to $\vec{\phi}$:} For any $\omega \in \mathcal{B}(\delta,\varepsilon,s)$, $\vec{\phi},\underline{\vec{\phi}} \in \R^M$, and $j = 1, \ldots, M$,
\begin{equation} \label{lipschitz}
\lvert \mathcal{T}_j(\vec{\phi},\omega) - \mathcal{T}_j(\underline{\vec{\phi}},\omega) \rvert \le M |t| \varepsilon^{\frac{10}{7}(1-\delta)} \, \norm{\vec{\phi}-\underline{\vec{\phi}}}_{\ell^{\infty}}.
\end{equation}

\noindent \emph{Proof.} 
A straightforward computation as in \Cref{thm: dynamics in the new variables} yields, using \eqref{eq: approx_phases random} and \eqref{def: random phase nonlinear operator},
\begin{equation} \label{eq: lip estimate}
\begin{split}
\lvert \mathcal{T}_j(\vec{\phi},\omega) & - \mathcal{T}_j(\underline{\vec{\phi}},\omega) \rvert \le \frac{|t|}{6|j|} \left \lvert |\Phi^{-1}(\varepsilon \widetilde{\tu}_0^{\omega}(\cdot ; \vec{\phi}))_j|^2 - |\Phi^{-1}(\varepsilon \widetilde{\tu}_0^{\omega}(\cdot ; \underline{\vec{\phi}}))_j|^2 \right \rvert \\
& + C(n,\delta) |t j| \sum_{m=3}^{\lfloor \frac{r}{2} \rfloor} \varepsilon^{-\frac{2m-3}{7}(1-\delta)} \varepsilon^{\frac{10}{7} (m-2) (1-\delta)} \sum_{k \in \Z^*} \left \lvert |\Phi^{-1}(\varepsilon \widetilde{\tu}_0^{\omega}(\cdot ; \vec{\phi}))_k|^2 - |\Phi^{-1}(\varepsilon \widetilde{\tu}_0^{\omega}(\cdot ; \underline{\vec{\phi}}))_k|^2 \right \rvert 
\end{split}
\end{equation}
where we used the fact that $\omega \in \mathcal{B}(\delta,\varepsilon,s)$ (cf. \eqref{def: event small size initial datum}) and that $\norm{\Phi^{-1}(u)}_{\dot{H}^s}\leq 2 \norm{u}_{\dot{H}^s}$.

By the reverse triangle inequality and \eqref{eq: lip_const}, we have
\begin{equation} \label{eq: stima fasi}
\begin{split}
& \sum_{k \in \Z^*} \left \lvert |\Phi^{-1}(\varepsilon \widetilde{\tu}_0^{\omega}(\cdot ; \vec{\phi}))_k|^2 - |\Phi^{-1}(\varepsilon \widetilde{\tu}_0^{\omega}(\cdot ; \underline{\vec{\phi}}))_k|^2 \right \rvert \le 4\varepsilon^{\frac{12}{7} - \frac57 \delta} \norm{\widetilde{\tu}_0^{\omega}(\cdot ; \vec{\phi}) - \widetilde{\tu}_0^{\omega}(\cdot ; \underline{\vec{\phi}})}_{\dot{H}^s}.
\end{split}
\end{equation}
Inserting \eqref{eq: stima fasi} in \eqref{eq: lip estimate}, using that $\omega\in \mathcal{B}(\delta,\varepsilon,s)$, provided that $\varepsilon \lesssim_{n,\delta} 1$ is sufficiently small, we obtain
\[
\begin{split}
\lvert \mathcal{T}_j(\vec{\phi},\omega) - \mathcal{T}_j(\underline{\vec{\phi}},\omega) \rvert & \le M |t| \varepsilon^{\frac{12}{7} - \frac57 \delta} \, \norm{\widetilde{\tu}_0^{\omega}(\cdot ; \vec{\phi}) - \widetilde{\tu}_0^{\omega}(\cdot ; \underline{\vec{\phi}})}_{\dot{H}^s} \le M |t| \varepsilon^{\frac{12}{7} - \frac57 \delta} \,\norm{\widetilde{\tu}_0^{\omega}(\cdot ; \vec{\phi})}_{\dot{H}^s} \norm{\vec{\phi} - \underline{\vec{\phi}} }_{\ell^{\infty}} \\
& \le M |t| \varepsilon^{\frac{10}{7} (1-\delta)} \,  \norm{\vec{\phi} - \underline{\vec{\phi}} }_{\ell^{\infty}}.
\end{split}
\]

\item[(i.2)] \underline{Measurability with respect to $\omega$:} For any $\vec{\phi} \in \R^M$, the map
\[
\mathcal{T}(\vec{\phi},\cdot) : ( \mathcal{B}(\delta,\varepsilon,s) , \widetilde{\mathcal{G}} ) \longrightarrow ( \R^M , \mathfrak{B}(\R^M))
\]
is measurable, where $\mathfrak{B}(\R^M)$ denotes the Borel $\sigma$-algebra of $\R^M$.

\noindent \emph{Proof.} Define the deterministic set
\[
\widetilde{\mathcal{B}}(\delta,\varepsilon,s,M) = \left \{ v = (\vec{r},\widetilde{u}_{0,M}) \in \R^M \times \dot{H}^s \, \big \lvert \, \Pi_{\le M} \widetilde{u}_{0,M} = 0, \, \, \norm{v}_{\widetilde{\mathcal{B}}} \le \varepsilon^{-\frac27-\frac57 \delta} \right \},
\]
where
\[
\norm{v}_{\widetilde{\mathcal{B}}} := \left ( 2 \sum_{j=1}^M j^{2s} c_j^2 r_j^2 + \norm{\widetilde{u}_{0,M}}_{\dot{H}^s}^2 \right )^{\frac{1}{2}} .
\]
We also introduce the deterministic operator
\begin{equation} \label{def: deterministic phase nonlinear operator}
\begin{split}
& \widetilde{\mathcal{T}}_j : \R^M \times \widetilde{\mathcal{B}}(\delta,\varepsilon,s,M) \rightarrow \R,  \qquad j=1,\ldots,M,\\
& \widetilde{\mathcal{T}}_j(\vec{\phi},v) = -t \theta_j(\widetilde{u}_0(v,\vec{\phi})), \qquad  \quad \widetilde{u}_0(v,\vec{\phi}) = 2 \sum_{j=1}^M c_j r_j \cos(jx+\phi_j) + \widetilde{u}_{0,M}(x),
\end{split}
\end{equation}
and the random map
\begin{equation} \label{def: iota}
\iota : \mathcal{B}(\delta,\varepsilon,s) \subseteq \mathbb{\Omega} \longrightarrow \widetilde{\mathcal{B}}(\delta,\varepsilon,s,M) \subseteq \R^M \times \dot{H}^s, \quad \iota(\omega) = \left ( (R_j^{\omega})_{j = 1}^M , \Pi_{> M} \widetilde{u}_0^{\omega} )  \right ).
\end{equation}

Then we have the following facts.
\begin{itemize}
    \item The random map $\iota$ defined in \eqref{def: iota} is measurable.
    Indeed it is weakly measurable, in the sense that given a functional $f$ in the dual of $\R^M \times \dot{H}^s$, then $f \circ \iota$ is measurable. 
    Since $\R^M \times \dot{H}^s$ is a separable Banach space, by the Pettis measurability theorem \cite{pettis}, $\iota$ is also strongly measurable.
    \item The deterministic map $\widetilde{\mathcal{T}}$ is Lipschitz-continuous. 
    Indeed, proceeding as in \eqref{eq: lip estimate}--\eqref{eq: stima fasi}, we find
    \[
    \lvert \widetilde{\mathcal{T}}_j(\vec{\phi},v) - \widetilde{\mathcal{T}}_j(\underline{\vec{\phi}},\underline{v}) \rvert \le M\, |t| \varepsilon^{\frac{12}{7} - \frac57 \delta}\, \left ( \norm{v-\underline{v}}_{\widetilde{\mathcal{B}}} + \varepsilon^{-\frac27-\frac57 \delta} \norm{\vec{\phi} - \underline{\vec{\phi}} }_{\ell^{\infty}} \right ).
    \]
    \item The map $\mathcal{T}$ is given by the composition $\mathcal{T}(\vec{\phi},\omega) = \widetilde{\mathcal{T}}(\vec{\phi},\iota(\omega))$.
\end{itemize}
The measurability of $\mathcal{T}(\vec{\phi},\cdot)$ follows from the fact that it is the composition the measurable map $\iota$ and the continuous map $\widetilde{\mathcal{T}}(\vec{\phi},\cdot)$.
\end{itemize}

\noindent In order to set up the random fixed point argument, we first note that, for $\varepsilon\lesssim_{n,\delta} 1$ sufficiently small,
\begin{equation}\label{eq: trivial bound}
\begin{split}
\lvert \mathcal{T}_j(\vec{\phi},\omega) \rvert \stackrel{\eqref{eq: approx_phases random}}{\le}&\ |t||j|^3 + \frac{|t|}{6|j|} |\Phi^{-1}(\varepsilon \widetilde{\tu}_0^{\omega}(\cdot ; \vec{\phi}))_j|^2 + C(n,\delta) |t j| \sum_{m=3}^{\lfloor \frac{r}{2} \rfloor} \varepsilon^{-\frac{2m-3}{7}(1-\delta)} \norm{\Phi^{-1}(\varepsilon \widetilde{\tu}_0^{\omega}(\cdot ; \vec{\phi}))}_{\dot{H}^s}^{2m-2} \\
\le&\ M^3 |t| + \frac{|t|}{6|j|} \varepsilon^{\frac{10}{7}(1-\delta)} + M |t| C(n,\delta) \sum_{m=3}^{\lfloor \frac{r}{2} \rfloor} \varepsilon^{-\frac{2m-3}{7}(1-\delta)} (2 \varepsilon^{\frac57(1-\delta)})^{2m-2} \le 2 M^3 |t|.
\end{split}
\end{equation}

Using (i.1) and (i.2), we can apply \Cref{thm: random fixed point} on the convex and compact subset $K= [-2M^3 |t|, 2M^3 |t|]^M \subseteq \R^M$. 
In particular, note that $\mathcal{T}(\cdot,\omega) : K \longrightarrow K$, using \eqref{eq: trivial bound} and the fact that $\mathcal{T}$ is $2\pi$-periodic in $\vec{\phi}$. We note that the fixed point $\vec{\phi}^{*,\omega}$ is defined in the restricted probability space $(\mathcal{B}(\delta,\varepsilon,s), \widetilde{\mathcal{G}},\widetilde{\P})$ and can be extended to the probability space $(\mathbb{\Omega},\mathcal{G},\mathbb{P})$ by setting $\vec{\phi}^{*,\omega}=0$ for $\omega\in \mathbb{\Omega}\setminus \mathcal{B}(\delta,\varepsilon,s)$.

\medskip

\noindent (ii) {\sc Neighborhood of the fixed point:} The measurability of $\mathcal{N}(\beta)$ follows from the fact that $\omega \in \mathbb{\Omega} \mapsto (\vec{\phi}^{\omega} , \vec{\phi}^{*,\omega})$ is measurable and $(a,b) \in \R^M \times \R^M \mapsto \norm{a-b}_{\ell^{\infty}}$ is continuous. 
Its probability is a consequence of \eqref{eq: factorization} with $\A = \mathbb{\Omega}.$

\medskip

\noindent (iii) {\sc Stability property:} By step (i),
$\mathcal{T}(\vec{\phi}^{*,\omega},\omega) - \vec{\phi}^{*,\omega} = 0$. For any $\omega\in \mathcal{N}(\beta)$,
\[
\left \lvert \phi_j - \mathcal{T}_j(\vec{\phi},\omega) \right \rvert  \le \left \lvert \phi_j - \phi_j^{*,\omega} \right \rvert + \left \lvert \mathcal{T}_j(\vec{\phi}^{*,\omega},\omega) - \mathcal{T}_j(\vec{\phi},\omega) \right \rvert  \stackrel{\eqref{lipschitz},\eqref{def: intorno}}{\le} \left ( 1 + M |t| \varepsilon^{\frac{10}{7}(1-\delta)}  \right ) \beta .
\]

\medskip

\noindent (iv) {\sc Factorization property:} Notice that for any $\mathcal{G}$-measurable event $\A$:
\[
\P (\A \cap \mathcal{N}(\beta) ) = \E \left [ \mathbbm{1}_{\A} \mathbbm{1}_{\mathcal{N}(\beta)} \right ] \stackrel{\eqref{tower property}}{=} \E \left [ \E \left [ \mathbbm{1}_{\A} \mathbbm{1}_{\mathcal{N}(\beta)} \lvert \mathcal{G} \right ] \right ] \stackrel{\eqref{take out what is known}}{=} \E \left [ \mathbbm{1}_{\A} \E \left [  \mathbbm{1}_{\mathcal{N}(\beta)} \lvert \mathcal{G} \right ] \right ]
\]

Now we apply \eqref{eq: resnick} with $X = \vec{\phi}^{*,\omega}$ (which is $\mathcal{G}$-measurable by point (i)), $Y = \vec{\phi}^{\omega}$ (which is $\mathcal{G}$-independent) and $\Psi(x,y) = \mathbbm{1}_{\mathcal{N}(\beta ; x)}(y)$, where
\[
\mathcal{N}(\beta ; x) = \{ y \in \R^M \, \lvert \, \norm{x-y}_{\ell^{\infty}} < \beta \}.
\]
We find that
\[
\P (\A \cap \mathcal{N}(\beta) ) = \E \left [ \mathbbm{1}_{\A} \E \left [  \mathbbm{1}_{\mathcal{N}(\beta ; x)}(\vec{\phi}^{\omega}) \right ]_{\lvert x = \vec{\phi}^{*,\omega}} \right ] = \left ( \frac{\beta}{\pi} \right )^M \P (\A).
\]
$\hfill \qedsymbol$

\,

\appendix

\section{Technical results}

\subsection{The KdV hierarchy} \label{proof lemma first integrals}

In this appendix we prove \Cref{lem:first_integrals} characterizing the KdV hierarchy.
We first introduce the notion of \textit{rank} of a first integral given in \cite{miura}.
Given a monomial $P$ in $u$ and its derivatives, we define
\begin{equation}\label{def_rank}
P(u) = \prod_{k=0}^r (\pa_x^k u)^{a_k},\qquad 
\mathrm{rank}(P) = \sum_{k = 0}^r \left ( 1+\frac{k}{2} \right ) a_k,
\end{equation}
where $k$ is the order of derivation and $a_k$ the corresponding exponent. Moreover, $A= \sum_{k=0}^r a_k$ is the \emph{degree} of $P(u)$ as a monomial in $u$ and $B=\sum_{k=1}^{r} k\, a_k$ is the total number of derivatives in $P(u)$.

\begin{lemma} \label{lemma: scaling}
For $j=1,\ldots,n$ let $m_j \in \N$, $s > \max_{j=1,\ldots,n} m_j + \frac{1}{2}$ and
\[
P_j(u) = \prod_{i=0}^{m_j} (\pa_x^i u)^{a_i^{(j)}} \quad A_j = \sum_{i=0}^{m_j} a_i^{(j)}, \quad B_j = \sum_{i=1}^{m_j} i a_i^{(j)}.
\]
If $u \in \dot{H}^s(\T)$ then each $P_j(u)$ is a continuous function of $x \in \T$ and if $\sum_{j=1}^n P_j(u) = 0$ on $\dot{H}^s$, then for any fixed $A,B \in \R,$
\[
\sum_{ \substack{ j \in \{ 1,\ldots,n \} \\ A_j = A, \, B_j = B } } P_j(u) = 0 \qquad \text{on $\dot{H}^s(\T).$}
\]
\end{lemma}
\begin{proof}
The continuity (hence the boundedness) of $P_j(u)$ follows by the Sobolev embedding and the algebra property for $\dot{H}^s$ with $s > \frac12$.
For any $\lambda > 0$ we have
\[
0 = \sum_{j=1}^n P_j(\lambda u) = \sum_{j=1}^n \lambda^{A_j} P_j(u).
\]
Without loss of generality we can assume $A_1 \ge A_2,\ldots,A_n$ and so
\[
0 = \sum_{j=1}^n \lambda^{A_j-A_1} P_j(u).
\]
Taking $\lambda \to \infty$ we obtain
\[
\sum_{ \substack{ j \in \{ 1,\ldots,n \} \\ A_j = A} } P_j(u) = 0.
\]
We may thus assume that all the $P_j$'s have the same degree $A$ from now on. Let $m \in \N$ and consider $u_m(x)=u(mx)$.
We obtain
\[
0 = \sum_{j=1}^n P_j(u_m) = \left ( \sum_{j=1}^n m^{B_j} P_j(u) \right )(mx).
\]
Again, without loss of generality, we can assume $B_1 \ge B_2,\ldots,B_n$. Then for any $x\in\T$,
\begin{equation}\label{m_limit}
\lim_{m \to +\infty} \left ( \sum_{ \substack{ j \in \{1,\ldots,n\} \\ B_j = B_1 } } P_j(u) \right ) \left (mx\right ) = 0.
\end{equation}

In particular, setting $x=\frac{2\pi p}{q}$, with $p,q\in\N$, $p<q$, then the sequence $\{mx\}_{m\in\N}$ takes $q$ values. 
Taking the subsequence $m\equiv 1$ (mod $q$) in \eqref{m_limit}, we find that $\sum_{ \substack{ j \in \{1,\ldots,n\} \\ B_j = B_1 } } P_j(u)$ vanishes on $2\pi\Q\cap \T$, and thus in $\T$ by continuity.
\end{proof}

\begin{proof}[Proof of \Cref{lem:first_integrals}]
Consider a monomial $P$ in \eqref{def_rank}.
The following properties are easy to prove: (i) the operator $\pa_x$ increases the rank by $\frac{1}{2}$, (ii) integration by parts leaves the rank constant, and (iii) $\nabla_{L^2}$ decreases the rank by $1$.

\medskip

\noindent{\sc Rank of monomials in $F^{(j)}$.}
As proved in \cite{lax,magri}, the first integrals $F^{(j)}$ in \eqref{eq: pre first integrals} satisfy the following recursive formula:
\begin{equation}\label{eq:recursiveF}
J_1 \nabla_{L^2} F^{(j+1)} = J_2 \nabla_{L^2} F^{(j)}, \qquad J_1 = \pa_x, \quad J_2 =  -\left( \partial_{xxx} + \frac{2}{3} u \partial_x + \frac{1}{3} u_x \right).
\end{equation}
Using properties (i)-(iii), we want to show that each monomial in $F^{(j)}$ has rank $j+2$, $j\in\N$. Indeed, $F^{(1)}=\H$ in \eqref{intro: kdv hamiltonian in fourier} only has monomials of rank 3 and the only one of degree $2$ up to a constant is $\frac{1}{2} \int_{\T} u_x^2 \, \mathrm{d}x$. 
Assuming by induction that each monomial in $F^{(j)}$ has rank $j+2$ and that the only monomial of degree $2$ up to a constant is $\frac{1}{2} \int_{\T} (\pa_x^j u)^2 \, \mathrm{d}x$, we show that monomials in $F^{(j+1)}$ must have rank $j+3$ and that the only monomial of degree 2 up to a constant is $\frac{1}{2} \int_{\T} (\pa_x^{j+1} u)^2 \, \mathrm{d}x$.
By properties (i)-(iii) above, all non-zero monomials in $J_2 \nabla_{L^2} F^{(j)}$ have rank $j+2+\frac{1}{2}$, and the only monomial of degree 1 is, up to a constant, $\pa_x^{2j+3} u$. 
Since the latter is the only monomial of degree 1, \Cref{lemma: scaling} implies that $J_2 \nabla_{L^2} F^{(j)}$ is not identically zero.
If any monomials in $J_1 \nabla_{L^2} F^{(j+1)}$ had rank different from $j+2+\frac{1}{2}$, \eqref{eq:recursiveF} would yield an equality between sums of monomials of different ranks, which,  by \Cref{lemma: scaling}, implies that the sum of any terms of rank different from  $j+2+\frac{1}{2}$ vanishes. 
As a result, monomials in $J_1 \nabla_{L^2} F^{(j+1)}$ must have rank $j+2+\frac{1}{2}$, which, given that $J_1 \nabla_{L^2}$ decreases the rank by $\frac12$, must come from monomials of $F^{(j+1)}$ of rank $j+3$.
Moreover, the only monomial of degree 2 in $F^{(j+1)}$ is, up to a constant, $\frac{1}{2} \int_{\T} (\pa_x^{j+1} u)^2 \, \mathrm{d}x$.

\medskip

\noindent {\sc Structure of $F^{(j)}$.}
Each monomial $P$ in $F^{(j)}$, $j\in\N$, has rank
\[
j+2 = \sum_{k=0}^r \left ( 1+\frac{k}{2} \right )a_k.
\]
If we isolate the term of degree $n=\sum_{k=0}^r a_k$, we have
\[
\sum_{k=0}^r k a_k = 2(j+2-n),
\]
which proves that we cannot have terms of degree $n> j+2$.
Moreover, if $n=2$, by integrating by parts we can always assume $a_k=0$ for $k \neq j$ and $a_j=2$.
From now on assume $n \ge 3$.
If $u$ appears in a monomial of $F^{(j)}(u)$ of degree $n$, with $r \geq j$ derivatives, then it has to appear with degree 1, because 
\[
r a_r \leq \sum_{k=0}^r k a_k = 2(j+2-n) \quad \implies \quad  a_r \leq \frac{2(j+2-n)}{r}< 2.
\]
Suppose that there exists at least one nonzero $a_r$ for $r\geq j$. 
Then we have shown that $a_r=1$ and 
\[
\sum_{k=1}^{r-1} k a_k = 2(j+2-n)-r\leq 2j+4-2n -r \leq 2j-2-r.
\]
Since $n\geq 3$, we may write $P(u)=(\pa_x^r u)\, Q(u)$ where $Q$ is another monomial with at most two factors. Moreover, the total number of derivatives among the factors of $Q$ is at most $2j-2-r$. 
Integrating by parts $r-(j-1)$ times, we find that 
\[
\int_{\T} (\pa_x^r u)\, Q(u) \,\mathrm{d}x = (-1)^{r-(j-1)} \int_{\T} (\pa_x^{j-1} u)\, \pa_x^{r-j+1}Q(u) \,\mathrm{d}x.
\]
The total number of derivatives in $\pa_x^{r-j+1}Q(u)$ is thus $2j-2-r + (r-j+1) = j-1$.

Since $F^{(j)}$ is the sum of a degree 2 monomial plus monomials of degree $n\geq 3$, we have proved that, up to a multiplicative constant, $F^{(j)}$ has the form:
\[
F^{(j)}(u)  = \frac{1}{2} \int_{\T} u_j^2 \, \mathrm{d} x + \int_{\T} \sum_{n=3}^{j+2} P_n^{(j)}(u,u_1,\ldots,u_{j-1}) \, \mathrm{d} x = \pi \sum_{k \in \Z^*} k^{2j} |u_k|^2 + \sum_{n=3}^{j+2} \sum_{\k \in \M_n} c_{\k}^{(j)} u^{\k}.
\]
Passing to Fourier variables, notice that in principle $c_{\k}^{(j)} \in \mathbb{Q}(i)[k_1,\ldots,k_{\#\k}]$, but since the rank is an integer number and each space derivative counts as $\frac{1}{2}$ in the rank, we can conclude that there is an even number of space derivatives and that $c_{\k}^{(j)} \in \mathbb{Q}[k_1,\ldots,k_{\#\k}].$
Finally notice that if $n \ge 3$ and $\k \in \M_n$, then the total degree of $c_{\k}^{(j)}$ is $\sum
_{k=0}^{j} ka_k = 2(j-n+2) \le 2j-2.$
\end{proof}

\subsection{Intersection of resonant manifolds} \label{app: proof system}

In this subsection we prove \Cref{cor: system}, which follows from the stronger \Cref{thm: system for resonant 2} below by taking $\alpha_1=\ldots=\alpha_n=1$.

\begin{theorem} \label{thm: system for resonant 2}
Given $n\in\N$, consider the system of $n$ equations
\begin{equation}
\begin{cases}
\alpha_1 k_1 + \alpha_2 k_2+\ldots+ \alpha_n k_n & = 0 \\
\alpha_1 k_1^3 + \alpha_2 k_2^3 + \ldots+ \alpha_n k_n^3 & = 0 \\
& \vdots \\
\alpha_1 k_1^{2n-1} + \alpha_2 k_2^{2n-1} +\ldots + \alpha_n k_n^{2n-1} & =0,
\end{cases}
\label{system 3}
\end{equation}
where $\alpha_j \in \N$, $j=1,\ldots, n$. We require that $k_j \neq 0$ for all $j=1,\ldots,n$. Then:

\begin{itemize}
\item If $S = \alpha_1 + \alpha_2 + \ldots + \alpha_n$ is odd, there is no integer solution to \eqref{system 3}. 

\item If $S$ is even, all the solutions $(k_1,k_2,\ldots,k_n)$ are \emph{weightedly paired}, meaning that if we build the vector 
\[
(k_1',\ldots,k_S')=(k_{1,1},\ldots,k_{1,\alpha_1},k_{2,1},\ldots,k_{2,\alpha_2},\ldots,k_{n,1},\ldots,k_{n,\alpha_n}) \in \Z^{S},
\]
where $k_{j,1}=k_{j,2}=\ldots=k_{j,\alpha_j}=k_j$, there exists a permutation $\sigma \in Sym(S)$ such that 
\[
k_{\sigma(2i-1)}'+k_{\sigma(2i)}'=0 \qquad \text{for $i=1,\ldots,\frac{S}{2}.$}
\]
\end{itemize}
\end{theorem}
\begin{proof}
First, we write \eqref{system 3} as a system:
\[
\begin{bmatrix}
k_1 & k_2  & \dots  & k_n \\
k_1^3 & k_2^3 & \dots  & k_n^3 \\
\vdots & \vdots & \ddots & \vdots \\
k_1^{2n-1} & k_2^{2n-1} & \dots  & k_n^{2n-1}
\end{bmatrix} 
\cdot
\begin{bmatrix}
\alpha_1 \\
\alpha_2 \\
\vdots \\
\alpha_n
\end{bmatrix}
=
\begin{bmatrix}
0 \\
0 \\
\vdots \\
0
\end{bmatrix}.
\]
Since $\alpha_1,\ldots,\alpha_n\neq 0$, we must have a singular matrix: 
\begin{equation} 
\begin{split}
0 & = \mathrm{det}
\begin{bmatrix}
k_1 & k_2  & \dots  & k_n \\
k_1^3 & k_2^3 & \dots  & k_n^3 \\
\vdots & \vdots & \ddots & \vdots \\
k_1^{2n-1} & k_2^{2n-1} & \dots  & k_n^{2n-1}
\end{bmatrix}
= k_1k_2\dots k_n \cdot \mathrm{det}
\begin{bmatrix}
1 & 1  & \dots  & 1 \\
k_1^2 & k_2^2 & \dots  & k_n^2 \\
\vdots & \vdots & \ddots & \vdots \\
k_1^{2n-2} & k_2^{2n-2} & \dots  & k_n^{2n-2}
\end{bmatrix} \\
& = k_1k_2\dots k_n \prod_{1 \le i < j \le n} (k_i^2-k_j^2) = k_1k_2\dots k_n \prod_{1 \le i < j \le n}(k_i-k_j)(k_i+k_j)=0,
\label{vandermonde2}
\end{split}
\end{equation}
where we used the explicit formula for the determinant of the Vandermonde matrix.
This implies that there exist $i,j$ such that $k_i+k_j=0$ or $k_i-k_j=0$. 
The rest of the proof proceeds by induction, in particular we will prove that if the theorem is true for a system of dimension $n-2$ and $n-1$, then it is true for $n$. We need two base cases.

\smallskip

\noindent {\sc Base Case 1}: If $n=1$, then $\alpha_1 k_1 \neq 0$ by hypothesis but \eqref{system 3} forces $\alpha_1 k_1 = 0$, hence there is no solution.

\smallskip

\noindent {\sc Base Case 2:} If $n=2$, \eqref{system 3}  becomes:
\begin{equation}
\begin{cases}
\alpha_1 k_1 + \alpha_2 k_2 = 0 \\
\alpha_1 k_1^3 + \alpha_2 k_2^3 = 0, 
\end{cases} \qquad \qquad \mbox{with}\ \alpha_1,\alpha_2 \in \N.
\label{finale 1,1}
\end{equation}
If we have a solution $(k_1,k_2)$ to \eqref{finale 1,1} with $k_1k_2 \neq 0$, then \eqref{vandermonde2} implies $k_1^2-k_2^2=0$. However, we cannot have $k_1=k_2$, since otherwise \eqref{finale 1,1} implies that $\alpha_1 + \alpha_2 = 0$, which is not possible. Therefore $k_1=-k_2 \neq 0$. Now we distinguish the two cases presented in the statement:
\begin{itemize}
    \item If $S=\alpha_1+\alpha_2$ is odd, there exist no solutions by Base Case 1. 
    \item If $S=\alpha_1+\alpha_2$ is even, then $(k_1,k_2)$ is a solution if and only if $\alpha_1=\alpha_2,$ which means that $(k_1,k_2)$ is weightedly paired.
\end{itemize}

\medskip

\noindent {\sc General Case:} Next we consider \eqref{system 3} with $n\geq 3$.
Let $\k = (k_1,\ldots,k_n)$ be a solution to \eqref{system 3} with $\prod_{j=1}^n k_j \neq 0$.
By \eqref{vandermonde2}, there are three possibilities:

\begin{itemize}
\item If $k_i-k_j=0$, up to a permutation of indices we can assume $k_{n-1}=k_n$ and we arrive to 
\begin{equation}\label{system_ind}
\begin{cases}
\alpha_1' k_1 + \alpha_2' k_2+\ldots+ \alpha_{n-1}' k_{n-1} & = 0 \\
\alpha_1' k_1^3 + \alpha_2' k_2^3 + \ldots+ \alpha_{n-1}' k_{n-1}^3 & = 0 \\
& \vdots \\
\alpha_1' k_1^{2n-3} + \alpha_2' k_2^{2n-3} +\ldots + \alpha_{n-1}' k_{n-1}^{2n-3} & =0,
\end{cases}
\end{equation}
with $\alpha_{n-1}' = \alpha_{n-1}+\alpha_{n}$ and $\alpha_i' = \alpha_i$ for $i=1,\ldots,n-2$. 
By the inductive hypothesis, since $S'=S$, if $S$ is even, the system \eqref{system_ind} has weightedly paired solutions, therefore $\k$ is weightedly paired. If $S$ is odd we have no solutions to \eqref{system_ind}, and thus $\k$ doesn't solve \eqref{system 3}. 

\item If $k_i+k_j=0$ with $\alpha_i-\alpha_j \neq 0$, up to a permutation of indices we arrive to 
\begin{equation} \label{system_ind2}
\begin{cases}
\alpha_1' k_1 + \alpha_2' k_2+\ldots+ \alpha_{n-1}' k_{n-1} & = 0 \\
\alpha_1' k_1^3 + \alpha_2' k_2^3 + \ldots+ \alpha_{n-1}' k_{n-1}^3 & = 0 \\
& \vdots \\
\alpha_1' k_1^{2n-3} + \alpha_2' k_2^{2n-3} +\ldots + \alpha_{n-1}' k_{n-1}^{2n-3} & =0,
\end{cases}
\end{equation}
with $\alpha_{n-1}' = \alpha_{n-1}-\alpha_{n} \in \N$ (without loss of generality) and $\alpha_i' = \alpha_i$ for $i=1,\ldots,n-2$. Notice that the indices we have removed are already weightedly paired. 
Since $S'=S-2\alpha_n$, $S-S' \equiv 0 \, (\mathrm{mod} \, \, 2)$, hence by the inductive hypothesis the system \eqref{system_ind2} has weightedly paired solutions if $S$ is even (and the same holds for $\k$), and no solutions if $S$ is odd. 

\item If $k_i+k_j=0$ with $\alpha_i-\alpha_j = 0$, up to a permutation of indices can assume $k_{n-1}+k_n=0$ and we arrive to 
\begin{equation} \label{syst_ind3}
\begin{cases}
\alpha_1' k_1 + \alpha_2' k_2+\ldots+ \alpha_{n-2}' k_{n-2} & = 0 \\
\alpha_1' k_1^3 + \alpha_2' k_2^3 + \ldots+ \alpha_{n-2}' k_{n-2}^3 & = 0 \\
& \vdots \\
\alpha_1' k_1^{2n-5} + \alpha_2' k_2^{2n-5} +\ldots + \alpha_{n-2}' k_{n-2}^{2n-5} & =0.
\end{cases}
\end{equation}
with $\alpha_j'=\alpha_j$ for $j=1,\ldots,n-2$.
Notice that $k_i$ and $k_j$ that we have removed are already weightedly paired. Again $S-S' \equiv 0 \quad (mod \, \, 2)$, hence by the inductive hypothesis the system \eqref{syst_ind3} has weightedly paired solutions if $S$ is even (and the same holds for $\k$), and no solutions if $S$ is odd. 
\end{itemize}
\end{proof}



\subsection{Support of the remainders}
\label{app: resti}

We dedicate this section to the proof of \Cref{cor: azione grande o tre indici grandi}, which is an adaptation of \cite[Lemma 3.7 and Corollary 3.8]{bernier}. We start with a key technical lemma:

\begin{lemma} \label{lemma algebrico bernier}
If $\bm{k} \in \mathcal{D}_n$ with $n \ge 3$ and satisfies $k_1+k_2 \neq 0$, then for every $l \in \N$ we have
\[
\max \left ( (n-2)^{2l} |k_3|^{2l+1} , \left \lvert \sum_{j=1}^n k_j^{2l+1} \right \rvert \right ) \ge \frac{|k_1|^{2l}}{2}.
\]
\end{lemma}
\begin{proof}
Without loss of generality we can assume that $k_1$ is positive. 
Then we have two cases:
\begin{itemize}
\item If $k_2$ is positive, then
\[
k_1 \le k_1+k_2 = -(k_3+\ldots+k_n) \le (n-2)|k_3|,
\]
which yields
\[
(n-2)^{2l} |k_3|^{2l+1} \ge \frac{|k_1|^{2l}}{2}.
\]
\item If $k_2$ is negative, then 
\[
k_1^{2l+1} + k_2^{2l+1} = k_1^{2l+1} - |k_2|^{2l+1} = (k_1+k_2) \left ( \sum_{j=0}^{2l} k_1^{2l-j}|k_2|^j \right ).
\]
Dividing by $k_1+k_2 \neq 0$ we obtain
\begin{equation}\label{eq:k1k3}
\left \lvert \sum_{j=1}^n k_j^{2l+1} \right \rvert \ge \left \lvert k_1^{2l+1} + k_2^{2l+1} \right \rvert - (n-2)|k_3|^{2l+1} \ge k_1^{2l} - (n-2) |k_3|^{2l+1}.
\end{equation}
Thus, if $\left \lvert \sum_{j=1}^n k_j^{2l+1} \right \rvert \le k_1^{2l}/2$, then \eqref{eq:k1k3} yields $(n-2)^{2l}|k_3|^{2l+1} \ge (n-2)|k_3|^{2l+1} \ge k_1^{2l}/2$.
\end{itemize}
\end{proof}

\begin{cor} \label{secondo lemma algebrico bernier}
Let $\k \in \mathcal{D}_n \cap (\mathcal{M}_n \setminus \mathcal{R}_n^{l})$ for some $n \ge 3$ and $l \in \N$. 
If $N > 2$ is such that 
\[
\left \lvert \frac{k_1^{2l-1}}{k_1^{2l+1}+\ldots+k_n^{2l+1}} \right \rvert \ge N
\]
then either there exists $\k' \in \M_{n-2}$ such that 
\begin{equation}\label{eq:action_case1}
u^{\k} = |u_a|^2 u^{\k'} \quad \text{where} \quad a \ge N^{\frac{1}{2l-1}}
\end{equation}
or 
\begin{equation}\label{eq:index_case2}
|k_3|^{2l+1} \ge \frac{N^{\frac{2l}{2l-1}}}{2(n-2)^{2l}}.
\end{equation}
\end{cor}
\begin{proof}
It is clear that $|k_1| \ge N^{\frac{1}{2l-1}}$. 
If $k_1+k_2=0$, then \eqref{eq:action_case1} is satisfied. Otherwise, since 
\[
|k_1^{2l+1}+\ldots+k_n^{2l+1}| \le \frac{|k_1|^{2l-1}}{N} < \frac{|k_1|^{2l}}{2},
\]
then by \Cref{lemma algebrico bernier},
\[
|k_3|^{2l+1} \ge \frac{|k_1|^{2l}}{2(n-2)^{2l}} \ge \frac{N^{\frac{2l}{2l-1}}}{2(n-2)^{2l}}.
\]
\end{proof}

\begin{proof}[Proof of \Cref{cor: azione grande o tre indici grandi}]
By \Cref{secondo lemma algebrico bernier}, we have \eqref{eq:action_case1} or \eqref{eq:index_case2}. \eqref{eq: azione grande} follows from \eqref{eq:action_case1} and the inequality $2n \ge 2l -1$. In order to prove \eqref{eq: terzo indice grande}, we use \eqref{eq:index_case2}:
\[
\left ( \frac{N^{\frac{2l}{2l-1}}}{2(n-2)^{2l}} \right )^{\frac{1}{2l+1}} \ge \frac{N^{\frac{2l}{4l^2-1}}}{2^{\frac{1}{2l+1}}(n-2)^{\frac{2l}{2l+1}}} \ge N^{\frac{1}{2n}}
\]
provided that $N \geq N_0 (n)$ is sufficiently large.
\end{proof}

\subsection{Proof of \texorpdfstring{\Cref{conv taylor series}}.} \label{app: normal form}

We first prove the existence of the map $\Phi_{G_n^{(l)}}^{\xi}$ and its inverse. By \eqref{eq: stima induttiva coeff g}, $N \norm{\bm{g}}_{Y_n^l} \lesssim_{\rho,n} N^{n-2}$ and by \Cref{lemma: transformations preserve Hs} there exists $\epsilon_0= c(n,s,\rho) N^{-1}$ such that for any
\begin{equation}\label{eq:eps0}
\epsilon \leq c(n,s,\rho)  N^{-1}
\end{equation}  
there exists an invertible canonical transformation
\begin{equation}
\Phi_{G_n^{(l)}}^{\xi} : B_s(0,\epsilon) \rightarrow B_s(0,2\epsilon), \qquad \Phi_{-G_n^{(l)}}^{\xi} : B_s(0,\epsilon/2) \rightarrow B_s(0,\epsilon)
\label{eq: dominio e codominio phi e phi_inv}
\end{equation}
for $\xi \in [0,2]$. 
Next we show \eqref{eq: close to id prop}: $\forall \xi \in [0,2]$,
\[
\norm{\Phi_{\pm G_n^{(l)}}^{\xi}(u)-u}_{\dot{H}^s} \stackrel{\eqref{eq: close to the identity 1}}{\lesssim_{n,s}} N \norm{\bm{g}}_{Y_n^l} \norm{u}_{\dot{H}^s}^{n-1} \stackrel{\eqref{eq: stima induttiva coeff g}}{\lesssim_{n,s,\rho}} N^{n-2} \norm{u}_{\dot{H}^s}^{n-1} \stackrel{\eqref{eq:eps0}}{\lesssim_{n,s,\rho}} N \norm{u}_{\dot{H}^s}^2.
\]
Moreover by construction
\[
\Phi_{G_r^{(l)}}^{\xi} \circ \Phi_{-G_r^{(l)}}^{\xi}(u) = u \qquad \forall u \in B_s(0,\epsilon/2).
\]

\medskip

\noindent {\sc Taylor formula for $F^{(j)} \circ \Phi_{G_n^{(l)}}^{\xi}$}: Firstly we remark that by \Cref{lemma: convergence of hamiltonian function}, since $s \ge j$ we have that $F^{(j)}(u)$ is a smooth function in $B_s(0, \frac{1}{c_s \rho N})$.
Up to reducing $\epsilon_0$ in \eqref{eq:eps0}, we have that $F^{(j)} \circ \Phi_{G_n^{(j)}}^{\xi}$ is well-defined in $B_s(0,\epsilon)$. 
Remember that by \eqref{eq: taylor formula for Hnew integral remainder}, we have the Taylor formula
\begin{equation}\label{eq:Taylor_Fj}
F^{(j)} \circ \Phi_{G_n^{(l)}}^1 = \sum_{\alpha = 0}^m \frac{1}{\alpha !} \mathrm{ad}_{G_n^{(l)}}^{\alpha} F^{(j)} + \int_0^1 \frac{(1-t)^m}{m!} \mathrm{ad}_{G_n^{(l)}}^{m+1}(F^{(j)} \circ \Phi_{G_n^{(l)}}^{\xi}) \, \mathrm{d} \xi
\end{equation}
where $\mathrm{ad}_{G_n^{(l)}} = \{ \cdot, G_n^{(l)} \}$, each term being well defined by \Cref{lemma: facts about poisson bracket}.
We want to prove that the Taylor series converges, i.e. that the integral remainder 
\[
R_m = \int_0^1 \frac{(1-t)^m}{m!} \mathrm{ad}_{G_n^{(l)}}^{m+1}(F^{(j)} \circ \Phi_{G_n^{(l)}}^{\xi}) \, \mathrm{d} \xi
\]
tends to zero as $m \to \infty$, uniformly in $B_s(0,\epsilon).$
By \eqref{formula Fj prop}, we have to prove that
\begin{equation}\label{Rmbeta}
\sum_{\beta \ge 2} R_{m,\beta} \longrightarrow_{m \to +\infty} 0, \quad \text{where} \quad R_{m,\beta} = \int_0^1 \frac{(1-t)^m}{m!} \mathrm{ad}_{G_n^{(l)}}^{m+1} (F_{\beta}^{(j)} \circ \Phi_{G_n^{(l)}}^{\xi}) \, \mathrm{d} \xi.
\end{equation}

\medskip

\noindent {\sc Estimates on the coefficients}: We begin by studying $\frac{1}{\alpha !} \mathrm{ad}_{G_n^{(l)}}^{\alpha} F_{\beta}^{(j)}$ in \eqref{eq:Taylor_Fj}.

\medskip

\noindent $\bullet$ $\beta \ge 3$. By \Cref{lemma: stabilità parentesi poisson} it is a homogeneous polynomial of degree $\beta + \alpha(n-2)$ of the form
\begin{equation}\label{def: dk alpha beta}
\frac{1}{\alpha !} \mathrm{ad}_{G_n^{(l)}}^{\alpha} F_{\beta}^{(j)} = \frac{1}{\alpha !} \sum_{\k \in \mathcal{M}_{\beta + \alpha(n-2)}} d_{\k}(\alpha,\beta) u^{\k}, \qquad \mathrm{with} \, \,  d_{\k} (0,\beta)=c_{\k}\ \mathrm{for}\ \#\k=\beta,
\end{equation}
where, using the recursive formula $\mathrm{ad}_{G_n^{(l)}}^{\alpha} F_{\beta}^{(j)} = \{ \mathrm{ad}_{G_n^{(l)}}^{\alpha-1} F_{\beta}^{(j)} , G_n^{(l)} \}$, we have the recursive bound\footnote{From now on we omit the subscript $n$ in the space $Y^{\theta}_n$ whenever clear from the context.}
\[
\norm{\bm{d}(\alpha,\beta)}_{Y^j} \stackrel{\eqref{d_est}}{\le} \frac{1}{2\pi} N n^{2j-1} [ \beta + (\alpha-1)(n-2) ] \norm{\bm{g}}_{Y^l} \norm{\bm{d}(\alpha-1,\beta)}_{Y^j}
\]
which in turn gives
\begin{align}
\frac{1}{\alpha !} \norm{\bm{d}(\alpha,\beta)}_{Y^j} &\le \frac{1}{\alpha !} \norm{ ( c_{\k} )_{\#\k=\beta} }_{Y^j} \left ( \frac{1}{2\pi} N n^{2j-1} \norm{\bm{g}}_{Y^l}  \right )^{\alpha} \prod_{i=1}^{\alpha} \left [ \beta + (i-1)(n-2) \right] \nonumber\\
& \stackrel{\eqref{formula Fj prop}-\eqref{eq: stima induttiva coeff g}}{\le} \rho^{\beta} N^{\beta-3} \left ( \frac{1}{2\pi} N n^{2j-1} \rho^n N^{n-3} \right )^{\alpha}\,  \prod_{i=1}^{\alpha} \frac{\beta + (i-1)(n-2)}{i}\nonumber \\
& \le \rho^{\beta} N^{\beta-3} \left ( \frac{1}{2\pi} N n^{2j-1} \rho^n N^{n-3} \right )^{\alpha} n^{\alpha} e^{\beta-1} \nonumber
\end{align}
where the last inequality may be found in page 1166 of \cite{bernier}.

Therefore
\begin{equation}
\frac{1}{\alpha !} \norm{\bm{d}(\alpha,\beta)}_{Y^j} \le e^{\beta-1} \left ( \frac{n^j \rho}{\sqrt{2\pi}} \right )^{2\alpha} (\rho N)^{\beta + \alpha(n-2)}N^{-3}.
\label{eq: prima stima sui d}
\end{equation}

\medskip 

\noindent $\bullet$ $\beta=2$. By \eqref{hom_eq prop}, it is a homogeneous $j$-formal polynomial of degree $2 + \alpha(n-2)$ of the form
\[
\frac{1}{\alpha} \frac{1}{(\alpha-1)!} \mathrm{ad}_{G_n^{(l)}}^{\alpha-1} \{ F_2^{(j)} , G_n^{(l)} \} = \frac{1}{\alpha !} \sum_{\k \in \M_{2 + \alpha(n-2)}} d_{\k}(\alpha,2) u^{\k}
\]
where, using \eqref{eq: prima stima sui d} with $\alpha-1$ instead of $\alpha$ and $n=\beta$, 
\begin{equation}
\frac{1}{\alpha !} \norm{\bm{d}(\alpha,2)}_{Y^j} \le \frac{1}{\alpha} e^{n-1} \left ( \frac{n^j \rho}{\sqrt{2\pi}} \right )^{2(\alpha-1)} (\rho N)^{n+(\alpha-1)(n-2)} N^{-3}.
\label{eq: seconda stima sui d}
\end{equation}


\medskip

\noindent {\sc Convergence of the Taylor series}:

\medskip 

\noindent $\bullet$ $\beta \ge 3$.
Proceeding as in \eqref{key_bound_Y}--\eqref{Fn_growth} (with $M=1$),
\begin{equation}\label{eq:est:ad_alpha_F_beta}
\begin{split}
\left \lvert \frac{1}{\alpha!} \mathrm{ad}_{G_n^{(l)}}^{\alpha} F_{\beta}^{(j)}(u) \right \rvert &\le \frac{1}{\alpha!} \sum_{\k \in \M_{\beta + \alpha (n-2)}} |d_{\k}(\alpha,\beta) u^{\k}| \\
&\le \frac{1}{\alpha!} [\beta + \alpha (n-2)]^j \norm{\bm{d}(\alpha,\beta)}_{Y^j} (c_s \norm{u}_{\dot{H}^s})^{\beta+\alpha (n-2)} \\
& \stackrel{\eqref{eq: prima stima sui d}}{\le} [\beta+\alpha (n-2)]^j e^{\beta-1} \left ( \frac{n^j \rho}{\sqrt{2\pi}} \right )^{2\alpha} (c_s \norm{u}_{\dot{H}^s} \rho N)^{\beta+\alpha(n-2)} N^{-3}. 
\end{split}
\end{equation}
Using this estimate with $\alpha=m+1$, we bound $R_{m,\beta}(u)$ in \eqref{Rmbeta}.
\[
\begin{split}
| R_{m,\beta}(u) | &= \left \lvert \int_0^1 \frac{(1-t)^m}{m!} ( \mathrm{ad}_{G_n^{(l)}}^{m+1} F_{\beta}^{(j)} )(\Phi_{G_r^{(l)}}^{\xi}(u)) \, \mathrm{d} \xi \right \rvert \\
&\stackrel{\eqref{eq: dominio e codominio phi e phi_inv}}{\le} [\beta + (m+1)(n-2)]^j e^{\beta-1} \left ( \frac{n^j \rho}{\sqrt{2\pi}} \right )^{2(m+1)} (2 \norm{u}_{\dot{H}^s} c_s \rho N)^{\beta+(m+1)(n-2)} N^{-3}. 
\end{split}
\]

Further restricting $\epsilon$ in \eqref{eq:eps0} such that
\begin{equation}\label{eq:eps0_1}
2 \norm{u}_{\dot{H}^s} c_s \rho N < 2 \epsilon\, c_s \rho N < e^{-1},
\end{equation}
it follows that the series $\sum_{\beta \ge 3} R_{m,\beta} (u)$ is absolutely convergent. Moreover, by further restricting $\epsilon$ such that
\begin{equation}\label{eq:eps0_2}
\left ( \frac{n^j \rho}{\sqrt{2\pi}} \right )^2 (2 \norm{u}_{\dot{H}^s} c_s \rho N)^{n-2} < \left ( \frac{n^j \rho}{\sqrt{2\pi}} \right )^2 (2 \epsilon c_s \rho N)^{n-2}  < 1,
\end{equation}
we have that
\[
\lim_{m \to +\infty} \sum_{\beta \ge 3} R_{m,\beta}(u) = 0.
\]

\noindent $\bullet$ $\beta=2$. Proceeding as in \eqref{key_bound_Y}--\eqref{Fn_growth} (with $M=1$),
\begin{equation}\label{eq:est:ad_alpha_F_2}
\begin{split}
\left \lvert \frac{1}{\alpha!} \mathrm{ad}_{G_n^{(l)}}^{\alpha} F_2^{(j)}(u) \right \rvert &\le \frac{1}{\alpha!} \sum_{\k \in \M_{2+(n-2)\alpha}} |d_{\k}(\alpha,2) u^{\k}| \\
&\le \frac{1}{\alpha!} [n+(\alpha-1)(n-2)]^j \norm{\bm{d}(\alpha,2)}_{Y^j} (c_s \norm{u}_{\dot{H}^s})^{n+(\alpha-1)(n-2)} \\
&\stackrel{\eqref{eq: seconda stima sui d}}{\le} [n+(\alpha-1)(n-2)]^j \frac{1}{\alpha} e^{n-1} \left ( \frac{n^j \rho}{\sqrt{2\pi}} \right )^{2(\alpha-1)} (c_s \norm{u}_{\dot{H}^s} \rho N)^{n+(\alpha-1)(n-2)} N^{-3}.
\end{split}
\end{equation}
Hence, taking $\alpha=m+1$,
\[
\begin{split}
| R_{m,2}(u) | &= \left \lvert \int_0^1 \frac{(1-t)^m}{m!} ( \mathrm{ad}_{G_n^{(l)}}^{m+1} F_2^{(j)} )(\Phi_{G_n^{(l)}}^{\xi}(u)) \, \mathrm{d} \xi \right \rvert \\
&\le [n+m(n-2)]^j \frac{1}{m+1} e^{n-1} \left ( \frac{n^j \rho}{\sqrt{2\pi}} \right )^{2m} (2 \norm{u}_{\dot{H}^s} c_s \rho N)^{n+m(n-2)} N^{-3}. 
\end{split}
\]
Using \eqref{eq:eps0_2}, it follows that 
$\lim_{m \to +\infty} R_{m,2}(u) = 0$.

\medskip

We may therefore take $m\rightarrow\infty$ in \eqref{eq:Taylor_Fj}, which yields
\begin{equation}\label{eq:taylor_Fj}
F^{(j)} \circ \Phi_{G_n^{(l)}}^1 (u) = \sum_{\alpha = 0}^{+\infty} \frac{1}{\alpha !} \mathrm{ad}_{G_n^{(l)}}^{\alpha}F^{(j)} (u) = \sum_{\alpha = 0}^{+\infty} \frac{1}{\alpha !} \mathrm{ad}_{G_n^{(l)}}^{\alpha} F_2^{(j)}(u) + \sum_{\alpha = 0}^{+\infty} \sum_{\beta \ge 3} \frac{1}{\alpha !} \mathrm{ad}_{G_n^{(l)}}^{\alpha} F_{\beta}^{(j)}(u).
\end{equation}

Using \eqref{eq:est:ad_alpha_F_beta} and \eqref{eq:est:ad_alpha_F_2}, one quickly sees that \eqref{eq:taylor_Fj} converges absolutely in $\alpha$ and $\beta$ provided \eqref{eq:eps0_1} and \eqref{eq:eps0_2} are satisfied.


\medskip

\noindent {\sc Estimates on the new coefficients}: We first isolate the homogeneous term of degree $m \ge 3$ in \eqref{eq:taylor_Fj}:
\[
Q_m^{(j)}(u) = \sum_{ \substack{ \alpha \ge 0, \, \, \beta \ge 3 \\ \beta + \alpha(n-2) = m} } \left [ \frac{1}{\alpha !} \mathrm{ad}_{G_n^{(l)}}^{\alpha} F_{\beta}^{(j)}(u) + \frac{\delta(\beta=n)}{(\alpha+1)!} \mathrm{ad}_{G_n^{(l)}}^{\alpha+1} F_2^{(j)}(u) \right ] = \sum_{\k \in \M_m} q_{\k} u^{\k}
\]
where, using \eqref{def: dk alpha beta},
\begin{equation}\label{eq: formula q}
q_{\k} = \sum_{\substack{ \alpha \ge 0, \, \, \beta \ge 3 \\ \beta + \alpha(n-2) = \#\k}} \left [ \frac{1}{\alpha !} d_{\k}(\alpha,\beta) + \frac{\delta(\beta=n)}{(\alpha+1)!} d_{\k}(\alpha+1,2) \right ].
\end{equation}
We have therefore proved \eqref{formula Fj trans prop}, i.e. on  $B_s(0,\epsilon)$ we have
\[
F^{(j)} \circ \Phi_{G_n^{(l)}}^1 = F_2^{(j)} + \sum_{m \ge 3} Q_m^{(j)}.
\]

Finally, we must prove the estimates on the coefficients $q_{\k}$ in \eqref{eq: stima coeff q}.

\begin{itemize}
\item If $\# \k< n$ then it follows from \eqref{eq: formula q} that $q_{\k} = d_{\k}(0,\#\k)= c_{\k}$ by \eqref{def: dk alpha beta}, and \eqref{eq: stima coeff q} follows from hypothesis \eqref{formula Fj prop}. 
\item If $\# \k =n$, then $Q_n^{(j)} = F_n^{(j)} + \{ F_2^{(j)} , G_n^{(l)} \}$ and it follows from \eqref{eq: formula q} that
\[
q_{\k}= d_{\k} (0, \#\k) + d_{\k} (1,2) \stackrel{\eqref{hom_eq prop}}{=} c_{\k} \delta(\k \in \M_n \setminus \mathcal{J}_{n,l,N}),
\]
proving \eqref{eq: formula qn}. Estimate \eqref{eq: stima coeff q} trivially follows from hypothesis \eqref{formula Fj prop}. 

\item If $\#\k > n$, we use \eqref{eq: prima stima sui d}--\eqref{eq: seconda stima sui d} together with \eqref{eq: formula q}:
\[
\begin{split}
\norm{\bm{q}}_{Y_{\#\k}^{j}} &\le 2 \!\!\!\!\!\!\!\!\! \sum_{ \substack{ \alpha \ge 0, \, \, \beta \ge 3 \\ \beta+\alpha(n-2)=\#\k } } \!\!\!\!\!\!\!\!\! e^{\beta-1} \left ( \frac{n^j\rho}{\sqrt{2\pi}} \right )^{2\alpha} \!\!\! (\rho N)^{\#\k} N^{-3} = N^{-3} \!\!\! \sum_{\alpha=0}^{\lfloor \frac{\#\k}{n-2} \rfloor} e^{\#\k-\alpha(n-2)} \left ( \frac{n^j\rho}{\sqrt{2\pi}} \right )^{2\alpha} \!\!\! (\rho N)^{\#\k} \leq \tilde{\rho}^{\#\k} N^{\#\k -3}
\end{split}
\]
where 
\[
\tilde{\rho} \le \, \max\left\lbrace e^{\frac{n-1}{n-2}}\rho , \left( \frac{e n^{2j}\rho^n}{2\pi}\right)^{\frac{1}{n-2}} \right\rbrace \lesssim_{n,j,\rho} 1.
\]
\end{itemize}
This proves \eqref{eq: stima coeff q} and concludes the proof of \Cref{conv taylor series}. \hfill \qedsymbol

\end{document}